\documentclass[11pt,a4paper]{amsart}
\usepackage[alphabetic]{amsrefs}
\input{preamble}\mytheoremstyleenglish
\usepackage{tikz}
\usetikzlibrary{cd}

\title
{$\mu_p$- and $\alpha_p$-actions on K3 surfaces in characteristic $p$}
\author{Yuya Matsumoto}
\date{2021/01/31}
\address{\tusaddressfull}
\email{\gmail}
\email{\tusmail}
\thanks{This work was supported by JSPS KAKENHI Grant Numbers 15H05738, 16K17560, and 20K14296.}
\subjclass[2010]{14J28 (Primary) 14L15, 14L30 (Secondary)}
\begin{document}

\begin{abstract}
We consider $\mu_p$- and $\alpha_p$-actions on RDP K3 surfaces 
(K3 surfaces with rational double point singularities allowed) in characteristic $p > 0$.
We study possible characteristics, quotient surfaces, and quotient singularities.
It turns out that these properties of $\mu_p$- and $\alpha_p$-actions
are analogous to those of $\bZ/l\bZ$-actions (for primes $l \neq p$) and $\bZ/p\bZ$-quotients respectively.
We also show that conversely an RDP K3 surface with a certain configuration of singularities 
admits a $\mu_p$- or $\alpha_p$- or $\bZ/p\bZ$-covering by a ``K3-like'' surface, which is often an RDP K3 surface but not always, 
as in the case of the canonical coverings of Enriques surfaces in characteristic $2$.
\end{abstract}

\maketitle

\section{Introduction}
K3 surfaces are proper smooth surfaces $X$ with $\Omega^2_X \cong \cO_X$ and $H^1(X, \cO_X) = 0$.
The first condition implies that $X$ has a global non-vanishing $2$-form and it is unique up to scalar.

Actions of (finite or infinite) groups on K3 surfaces have been vastly studied.
For example, the quotient of a K3 surface by an action of a finite group of order prime to the characteristic is birational to a K3 surface
if and only if the action preserves the global $2$-form, 
and moreover the list of possible such finite groups is determined in characteristic $0$.
Much less studied are infinitesimal actions, or \emph{derivations}, on K3 surfaces in positive characteristic
(with the exception of those with Enriques quotients in characteristic $2$).
Perhaps this is because it is known that smooth K3 surfaces admit no nontrivial global derivations.
However
we find many examples of nontrivial global derivations
when we will look at \emph{RDP K3 surfaces}, by which we mean we allow rational double point singularities (RDPs), the simplest $2$-dimensional singularities.

In this paper we consider derivations that correspond to actions of group schemes $\mu_p$ and $\alpha_p$.
We study possible characteristic, quotient surfaces, and quotient singularities.
It turns out that these properties of $\mu_p$- and $\alpha_p$-actions
are quite similar to those of $\bZ/p\bZ$-actions in characteristic $\neq p$ and characteristic $p$ respectively.

The actions of $\mu_p$, and more generally of $\mu_{p^e}$ and $\mu_{n}$, on K3 surfaces 
are also discussed in our previous paper \cite{Matsumoto:k3mun}.

\medskip

The content and the main results of this paper are as follows.

In Section \ref{sec:derivations} we introduce fundamental notions of derivations, 
such as $p$-closedness and fixed loci, and give their properties.
Then in Section \ref{sec:derivations on RDP} we describe local behaviors of derivations related to RDPs.
We classify $p$-closed derivations on RDPs without fixed points (Theorem \ref{thm:mu_p alpha_p RDP})
and RDPs arising as $p$-closed derivation quotients of regular local rings (Lemma \ref{lem:degree of isolated fixed point}(\ref{case:RDP})).

We show that a $\mu_p$- or $\alpha_p$-quotient $Y$ of an RDP K3 surface $X$ in characteristic $p$
is either an RDP K3 surface, an RDP Enriques surface, or a rational surface (Proposition \ref{prop:structure of quotient}).
For $\mu_p$-actions the author proved in \cite{Matsumoto:k3mun} that the quotient is an RDP K3 surface if and only if the induced action on the global $2$-forms is trivial
(this is parallel to the case of the actions of finite groups of order not divisible by $p$).
For $\alpha_p$-actions we could not find a similar criterion, since in this case the action on the $2$-form is always trivial (this is parallel to $\bZ/p\bZ$-actions).

In \cite{Matsumoto:k3mun} we proved that $\mu_p$-actions on RDP K3 surfaces in characteristic $p$ occurs precisely if $p \leq 19$.
In this paper we prove that the corresponding bound for $\alpha_p$-actions is $p \leq 11$ (Theorem \ref{thm:bound p}).

Suppose both $X$ and the quotient $Y$ are RDP K3 surfaces.
We determine the possible characteristic $p$
for both $\mu_p$ and $\alpha_p$,
and we moreover determine the possible singularities of $Y$ (Theorem \ref{thm:alpha_p K3 K3}).
Again the results are parallel to $\bZ/l\bZ$ (for a prime $l \neq p$) and $\bZ/p\bZ$ respectively.
We also determine the possible singularities of $X$ when the quotient $Y$ is a supersingular Enriques surface (Theorem \ref{thm:alpha_p K3 Enr}).

We also consider the inverse problem:
whether an RDP K3 surface $Y$ with a suitable configuration of singularities (and certain additional properties)
can be written as the $G$-quotient of an RDP K3 surface $X$.
It is known (at least to experts) that the answer is affirmative if $G = \bZ/l\bZ$.
We show a similar result (Theorem \ref{thm:singularities of quotient K3}) when $G$ is $\bZ/p\bZ$, $\mu_p$, or $\alpha_p$,
although if $G = \mu_p$ or $G = \alpha_p$ then $X$ is only ``K3-like'' (Definition \ref{def:k3-like}) in general and it may fail to be an RDP K3 surface.
This behavior is analogous to that of the canonical $\mu_2$- and $\alpha_2$-coverings of Enriques surfaces in characteristic $2$.

Now suppose $\map{\pi}{X}{Y}$ is a finite purely inseparable morphism of degree $p$ between RDP K3 surfaces.
It is not necessarily the quotient morphism by a (regular) action of $\mu_p$ or $\alpha_p$.
We show (Theorem \ref{thm:insep K3}) that $\pi$ admits a finite ``covering'' $\map{\bar{\pi}}{\bar{X}}{\bar{Y}}$
that is a $\mu_p$- or $\alpha_p$-quotient morphism between either RDP K3 surfaces or abelian surfaces.
We determine the possible covering degree and the characteristic for each case.

In Sections \ref{sec:enriques}--\ref{sec:examples} we give explicit examples of RDP K3 surfaces and derivations.

\medskip

Throughout the paper we work over an algebraically closed field $k$ of characteristic $p \geq 0$,
and whenever we refer to $\mu_p$, $\alpha_p$, or $p$-closed derivations we assume $p > 0$.

\section{Preliminary on derivations} \label{sec:derivations}

We recall basic facts on derivations, 
and relate differential forms on $X$ to those on the derivation quotient $X^D$.

\subsection{General properties of derivations} \label{subsec:derivation}

Let $X$ be a scheme over $k$.
A \emph{(regular) derivation} on $X$ 
is a $k$-linear endomorphism $D$ of $\cO_X$ 
satisfying the Leibniz rule $D(fg) = f D(g) + D(f) g$.

Suppose for simplicity that $X$ is integral.
Then a \emph{rational derivation} on $X$ 
is a global section of $\Der(\cO_X) \otimes_{\cO_X} k(X)$, 
where $\Der(\cO_X)$ is the sheaf of derivations on $X$.
Thus, a rational derivation is locally of the form $f^{-1} D$ with $f$ a regular function and $D$ a regular derivation.

\begin{lem} \label{lem:extending derivation}
	If $A$ is a local RDP and $D$ is a derivation on $(\Spec A)^{\sm}$ (the complement of the closed point),
	then $D$ extends to a derivation on $\Spec A$.
\end{lem}

\begin{proof}
	Indeed, for each $f \in A$ we have $D(f) \in H^0((\Spec A)^{\sm}, \cO_A) = H^0(\Spec A, \cO_A) = A$
	since $A$ is normal.
\end{proof}

\begin{lem} 
	Suppose $A$ is the localization of a finitely generated $k$-algebra at a maximal ideal $\fm$,
	and $D$ is a derivation on $A$.
	Then $D$ extends to a derivation on the completion $\hat{A} = \varprojlim_n A/\fm^n$, 
	and the completion $\widehat{A^D}$ of $A^D$ at $\fn := \fm \cap A^D$ is equal to $(\hat{A})^D$.
\end{lem}

\begin{proof}
	Any derivation $D$ satisfies $D(\fm^n) \subset \fm^{n-1}$,
	hence $D$ induces a morphism $ \varprojlim_n A/\fm^n \to  \varprojlim_n A/\fm^{n-1}$.

	There is a canonical injection $\widehat{A^D} \to (\hat{A})^D$.
	Let us show the surjectivity of this map.
	Suppose  $([a_n])_n$ is an element of $\hat{A}$ (i.e.\ $a_n \in A$ and $a_{n+l} \equiv a_n \pmod{\fm^n}$)
	that belong to $(\hat{A})^D$ (i.e.\ $D(a_n) \in \fm^{n-1}$).
	It suffices to find an element $b_n \in \fm^n$ with $D(b_n) = D(a_n)$,
	since then $([a_n]) = ([a_n - b_n]) \in \widehat{A^D}$.
	Since $D(a_n) = D(a_{n+l}) - D(a_{n+l} - a_n) \in \fm^{n+l-1} + D(\fm^n)$,
	it suffices to show
	$D(\fm^n) = \bigcap_{l \geq 0} (D(\fm^n) + \fm^{n+l})$.
	Suppose $\fm$ is generated by $N$ elements.
	This follows from Krull's intersection theorem,
	since $\pthpower{A}$ is a Noetherian local ring,
	$A$ and hence $D(\fm^l)$ are finitely generated $\pthpower{A}$-modules,
	and $\fm^{n+l} \subset \fm^{l} \subset (\pthpower{\fm})^{\ceil{(l-N(p-1))/p}} A$.
\end{proof}

\begin{defn}
	Suppose $D$ is a derivation on a scheme $X$.
	The \emph{fixed locus} $\Fix(D)$
	is the closed subscheme of $X$ 
	corresponding to the sheaf $(\Image (D))$ of ideals generated by $\Image(D) = \set{D(a) \mid a \in \cO_X}$.
	Equivalently, this sheaf is $\Image (\bar{D})$, where $\map{\bar{D}}{\Omega^1_X}{\cO_X}$ is the morphism defined below in Definition \ref{def:derivation on differential forms}.
	A \emph{fixed point} of $D$ is a point of $\Fix(D)$.
	
	Assume $X$ is a smooth irreducible variety and $D \neq 0$. 
	Then $\Fix(D)$ consists of its divisorial part $\divisorialfix{D}$ and non-divisorial part $\isolatedfix{D}$.
	If we write $D = f \sum_i g_i \partialdd{x_i}$ for some local coordinate $x_i$
	with $g_i$ having no common factor,
	then $\divisorialfix{D}$ and $\isolatedfix{D}$ corresponds to the ideal $(f)$ and $(g_i)$ respectively.
	
	Assume $X$ is a smooth irreducible variety and suppose $D \neq 0$ is now a \emph{rational} derivation,
	locally of the form $f^{-1} D'$ for some regular function $f$ and (regular) derivation $D'$.
	Then we define $\divisorialfix{D} = \divisorialfix{D'} - \divisor(f)$ and $\isolatedfix{D} = \isolatedfix{D'}$.
	
	If $X$ is only normal, then we can still define $\divisorialfix{D}$ as a Weil divisor.
\end{defn}
Rudakov--Shafarevich \cite{Rudakov--Shafarevich:inseparable}
uses the term \emph{singularity} for the fixed locus.
We do not use this, as we want to distinguish them from the singularities of the varieties.

The next theorem is proved by Rudakov--Shafarevich \cite{Rudakov--Shafarevich:inseparable}*{Theorem 3} for regular derivations $D$ satisfying some assumptions,
and by Katsura--Takeda \cite{Katsura--Takeda:quotients}*{Proposition 2.1} for general rational derivations.
\begin{thm} \label{thm:c2 derivation}
	Let $D$ be a rational derivation on a smooth proper surface $X$. 
	Then
	\[
	\deg c_2(X) = \deg \isolatedfix{D} - K_X \cdot \divisorialfix{D} - \divisorialfix{D}^2.
	\]
\end{thm}

A derivation $D$ on $X$ acts naturally on the sheaves $\Omega^q_X$, as follows.
\begin{defn} \label{def:derivation on differential forms}
Let $D$ be a derivation on $X$.
Decompose $\map{D}{\cO}{\cO}$ as ${\bar{D} \circ d} \colon \cO \namedto{d} \Omega^1 \namedto{\bar{D}} \cO$.
Then $\bar{D}$ is $\cO$-linear.
Let $\map{\bar{D}_q}{\Omega^q}{\Omega^{q-1}}$ ($q \geq 1$) be the ($\cO$-linear) homomorphism 
defined by 
\[
\bar{D}_q(\beta_1 \wedge \dots \wedge \beta_q) 
= \sum_{j = 1}^{q} (-1)^{j-1} \bar{D}(\beta_j) \cdot 
\beta_1 \wedge \dots \wedge \beta_{j-1} \wedge \beta_{j+1} \wedge \dots \wedge \beta_{q} ,
\]
for $1$-forms $\beta_j$, and for $q = 0$ let $\bar{D}_0$ be the zero map.
We have $\bar{D}_{q_1 + q_2}(\beta_1 \wedge \beta_2) 
= \bar{D}_{q_1}(\beta_1) \wedge \beta_2 + (-1)^{q_1} \beta_1 \wedge \bar{D}_{q_2}(\beta_2)$
for a $q_1$-form $\beta_1$ and a $q_2$-form $\beta_2$.
We define $\map{D_q := d \circ \bar{D}_q + \bar{D}_{q+1} \circ d}{\Omega^q}{\Omega^q}$ ($q \geq 0$).
\end{defn}
\begin{prop} \label{prop:derivation on differential forms}
Then we have the following properties.
\begin{itemize}
	\item $D_0 = D$.
	\item $D_1(df) = d(D_0(f))$.
	\item $D_{q_1 + q_2}(\beta_1 \wedge \beta_2) = D_{q_1}(\beta) \wedge \beta_2 + \beta_1 \wedge D_{q_2}(\beta_2)$ 
	for a $q_1$-form $\beta_1$ and a $q_2$-form $\beta_2$.
	Hence, $D_{q_1 + \dots + q_l}(\beta_1 \wedge \dots \wedge \beta_l)
	= \sum_{i = 1}^l \beta_1 \wedge \dots \wedge \beta_{i-1} \wedge D_{q_i}(\beta_i) \wedge \beta_{i+1} \wedge \dots \wedge \beta_l$
	for $q_i$-forms $\beta_i$.
	\item $[D, D']_q = [D_q, D'_q]$ and $(D^p)_q = (D_q)^p$.
	\item $(h D)_q = h \cdot D_q + dh \wedge \bar{D}_q$, 
	which is equal to $h \cdot D_q$ if for example $h \in \pthpower{k(X)}$.
	\item Hence, If $D^p = h D$ for $h \in \pthpower{k(X)}$, then $(D_q)^p = h D_q$.
	(This is not true for general $h \in k(X)$.)
\end{itemize}
\end{prop}
\begin{proof}
Straightforward.
\end{proof}
We will write simply $D$ in place of $D_q$.

\subsection{General properties of $p$-closed derivations}

We say that a derivation $D$ on an integral scheme $X$ is \emph{$p$-closed} if there exists $h \in k(X)$ with $D^p = h D$.
Quotients by such derivations will be studied in the next subsection.

The next formula is well-known.
\begin{lem}[Hochschild's formula] \label{lem:Hochschild}
Let $A$ be a $k$-algebra in characteristic $p > 0$, $a$ an element of $A$, and $D$ a derivation on $A$.
Then 
\[
(aD)^p = a^p D^p + (aD)^{p-1}(a) D.
\]
In particular, if $D$ is $p$-closed then so is $aD$.
\end{lem}

The following lemmas are useful when analyzing local properties.

\begin{lem} \label{lem:good coord}
	Suppose $B$ is a local domain
	equipped with a $p$-closed derivation $D \neq 0$
	such that $\Fix(D)$ is principal. 
	Then the maximal ideal $\fm$ of $B$ is generated by elements $x_j$ ($j \in J$) and $y$, 
	satisfying $D(x_j) = 0$. 
	If $\fm$ is generated by $n$ elements then we can take $\card{J} = n - 1$.
\end{lem}

If $B$ is smooth, then this is proved in 
\cite{Seshadri:Cartier}*{Proposition 6} (see also \cite{Rudakov--Shafarevich:inseparable}*{Theorem 1 and Corollary}).

\begin{proof}
	Take $f \in B$ with $\divisorialfix{D} = \divisor(f)$.
	By replacing $D$ with the (regular) derivation $f^{-1} D$, which is also $p$-closed by Hochschild's formula (Lemma \ref{lem:Hochschild}),
	we may assume $\divisorialfix{D} = 0$, hence $\Fix(D) = \emptyset$.

	Take $h \in B$ such that $D^p = h D$. 
	Note that then $D(h) = 0$.

	Take an element $y \in B$ with $D(y) \not\in \fm$ (which exists since $\fm \not\in \Fix(D)$).
	We may assume $y \in \fm$.
	Let $w = y^{p-1}$.
	Then $D^k(w) \in y B \subset \fm$ for $0 \leq k \leq p-2$
	and $D^{p-1}(w) \in B^*$.
	We have $u := D^{p-1}(w) - h w \in B^* \cap B^D$.
	
Take elements $(x'_j)_{j \in J'}$ generating $\fm$.
	Let \[ x_j = u x'_j + \sum_{k=0}^{p-2} (-1)^k D^k(w) D^{p-1-k}(x'_j) .\]
	Then we have $D(x_j) = 0$ and, since $x_j \equiv u x'_j \pmod{y B}$,
	it follows that $x_j$ ($j \in J'$) and $y$ generate $\fm$.
	If $\card{J'} < \infty$ then we can remove one of the elements,
	and the remaining elements still generate $\fm$.
	The removed one cannot be $y$ since $(D(x_j)) \subset \fm$,
	hence the removed one is $x_{j_0}$ for some $j_0 \in J'$,
	hence $x_j$ ($j \in J' \setminus \set{j_0}$) and $y$ generate $\fm$.
\end{proof}

\begin{lem} \label{lem:D=1}
	Suppose $B$ is a local domain
	equipped with a $p$-closed derivation $D \neq 0$ of additive type
	such that $\Fix(D) = \emptyset$.
	Then there exists $x \in B$ with $D(x) = 1$.  
\end{lem}
\begin{proof}
	As in the previous lemma, 
	since $\Fix(D) = \emptyset$, there exists $y \in \fm$ with $D(y) \not\in \fm$,
	and then $u := D^{p-1}(y^{p-1}) \in B^* \cap B^D$.
	Then $D^{p-1}(u^{-1} y^{p-1}) = 1$.
\end{proof}

\subsection{$p$-closed derivation quotients and differential forms} \label{subsec:derivation and differential forms}

If $D$ is $p$-closed, then $X^D$ is the scheme with underlying topological space homeomorphic to (and often identified with) $X$, and with structure sheaf
$\cO_{X^D} = \cO_X^D = \set{a \in \cO_X \mid D(a) = 0}$ consisting of the $D$-invariant sections of $\cO_X$.
The natural morphism $X \to X^D$ is finite of degree $p$ (unless $D = 0$).
If $X$ is normal then so is $X^D$.

In this subsection we compare top differential forms on $X$ and the quotient $X^D$ (Propositions \ref{prop:2-forms} and \ref{prop:n-step}).

Special cases of $p$-closed derivations correspond to (non-reduced) group schemes, as follows, which are the main subject of this paper.
\begin{prop}
	Let $G = \mu_p$ (resp.\ $G = \alpha_p$).
	Then the $G$-actions on a scheme $X$ 
	correspond bijectively to the derivations $D$ on $\cO_X$ of multiplicative type (resp.\ of additive type),
	that is, $D^p = D$ (resp.\ $D^p = 0$).
	The quotient scheme $X / G$ always exists, and coincides with $X^D$.
\end{prop}
\begin{proof}
	Well-known.
\end{proof}

\begin{lem} \label{lem:canonical 1-forms}
Let $X$ be a smooth variety of dimension $m$ (not necessarily proper)
equipped with a $p$-closed rational derivation $D$ 
such that $\Delta := \Fix(D)$ is divisorial. 
Let $\map{\pi}{X}{X^D}$ be the quotient map.
The morphism $\map{\pi^*}{\pi^* \Omega^1_{X^D}}{\Omega^1_X}$ 
induced by the pullback of $1$-forms
fits into a canonical exact sequence
\[ 
 0 
 \to \cO_X(-p \Delta) 
 \namedto{\overline{\pi'^*}} \pi^* \Omega^1_{X^D}
 \namedto{\pi^*} \Omega^1_X
 \namedto{\bar{D}} \cO_X(-\Delta)
 \to 0,
\]
where $F_X = \pi' \circ \pi \colon {X} \to X^D \to {\pthpower{X}}$ is the Frobenius,
$\bar{D}$ is defined as in Definition \ref{def:derivation on differential forms} (i.e.\ $\bar{D} \circ d = D$),
and $\overline{\pi'^*}$ is the morphism defined in the diagram
\[
\begin{tikzcd}
	F_{X^D}^* \Omega^1_{\pthpower{X^D}}  \arrow[r,"{\pthpower{\pi}}^*"] \arrow[rd, "0"] & 
	\pi'^* \Omega^1_{\pthpower{X}} \arrow[r, "\overline{\pthpower{D}}"] \arrow[d,"\pi'^*"] &
	\pi'^* \cO_{\pthpower{X}}(-\pthpower{\Delta}) \arrow[r] \arrow[dl, dotted, "\overline{\pi'^*}"] &
	0 \\
	&
	\Omega^1_{X^D} 
\end{tikzcd}
\]
and the equality $F_X^* (\cO_{\pthpower{X}}(-\pthpower{\Delta})) = \cO_X(-p \Delta)$.

Let $\eta$ (resp.\ $\xi$) be the image (resp.\ preimage) of $1$ by the induced isomorphism
$\cO_X \isomto \Ker (\pi^* \otimes \cO_X(p \Delta))$
(resp.\     $\Coker (\pi^* \otimes \cO_X(  \Delta)) \isomto \cO_X$).
Then 
$\eta = \frac{d(f^p)}{D(f)^p}$ and $\xi = \frac{df}{D(f)}$
for any local section $f \in \cO_X$ satisfying $\divisor(D(f)) = \Delta$.
Moreover $d \eta = 0$.
\end{lem}
\begin{proof}
	By the result of Seshadri (Lemma \ref{lem:good coord}),
	we can take a local coordinate $x_0, \dots, x_{m-1}$ of $X$
	such that $x_0^p, x_1, \dots, x_{m-1}$ is a local coordinate of $X^D$.
	Then $D = \phi \partialdd{x_0}$ 
	for some meromorphic function $\phi$ on $X$, and then $\Delta = \divisor(\phi)$.
	Then the sequence is 
	\[
		0 \to \spanned{\phi^p} 
		\to \spanned{d(x_0^p), dx_1, \dotsm dx_{m-1}}
		\to \spanned{dx_0, dx_1, \dotsm dx_{m-1}}
		\to \spanned{\phi}
		\to 0
	\]
	with $\phi^p \mapsto d(x_0^p)$ and $dx_0 \mapsto \phi$, which is clearly exact.
	The formulas of $\eta$ and $\xi$ are clear from the construction.
	$d \eta = 0$ follows either by computation using the formula
	or from the observation that 
	$d \eta \in \Image (\map{\bigwedge^2 \pi'^{*}}{F_X^* \Omega^2_{\pthpower{X}}}{\pi^* \Omega^2_{X^D}}) = 0$ 
	(since $\rank \pi'^* = 1$).
\end{proof}

\begin{prop} \label{prop:2-forms}
Let $D$ and $\map{\pi}{X}{X^D}$ as in Lemma \ref{lem:canonical 1-forms}. 
Then there is an isomorphism
\[
\Omega^m_{X/k}(\Delta)
\cong \pi^* (\Omega^m_{X^D/k}(\pi_*(\Delta)))
= \pi^* \Omega^m_{X^D/k} \otimes \cO_X(p \Delta)
\]
of $\cO_X$-modules, preserving the zero loci of forms, 
and sending 
\[ f_0 \cdot df_1 \wedge \dots \wedge df_{m-1} \wedge D(g)^{-1} dg
\mapsto
   f_0 \cdot df_1 \wedge \dots \wedge df_{m-1} \wedge D(g)^{-p} d(g^p) \]
for local sections $f_i, g$ of $\cO_X$
if $D(f_i) = 0$ for $1 \leq i < m$
and $D(g)^{-1} \in \cO_X(\Delta)$.

Taking powers and then the $D$-invariant parts, we also obtain an isomorphism
\[
((\pi_* \Omega^m_{X/k}(\Delta))^{\otimes n})^D
\cong (\Omega^m_{X^D/k}(\pi_*(\Delta)))^{\otimes n}
\]
of $\cO_{X^D}$-modules, 
satisfying the same property when $n = 1$ if $D(f_0) = 0$.

In particular, if $D$ is regular and fixed-point-free,
then we have isomorphisms 
\begin{align*}
(\pi_* (\Omega^m_{X/k})^{\otimes n}) ^{D}
&\cong (\Omega^m_{X^D/k})^{\otimes n} 
\quad \text{and} \\
H^0(X, (\Omega^m_{X/k})^{\otimes n})^{D} 
&\cong H^0(X^D, (\Omega^m_{X^D/k})^{\otimes n})
\end{align*}
with the same properties.
\end{prop}

This refines the Rudakov--Shafarevich formula \cite{Rudakov--Shafarevich:inseparable}*{Corollary 1 to Proposition 3}
$K_X \sim \pi^* K_{X^D} + (p-1) \divisorialfix{D}$ (linear equivalence).
We note that, by Lemma \ref{lem:good coord}, 
there indeed exist local sections $f_0, f_1, \dots, f_{m-1}, g$ 
for which the $m$-forms in the statement are generators.

\begin{proof}
	This follows immediately from the exact sequence in Lemma \ref{lem:canonical 1-forms}
	and the description of the elements $\eta$ and $\xi$.
\end{proof}

\begin{lem} \label{lem:exterior}
Suppose 
$V_n \namedto{G_{n-1}} V_{n-1} \namedto{G_{n-2}} \dots \namedto{G_0} V_0$
is a sequence of morphisms between locally-free sheaves of equal finite rank $m$ on an irreducible scheme
 such that $\Coker G_i$ are also locally-free
and $\sum_{i \in \clopint{0}{n}} \rank \Coker G_{i}$
is equal to the rank of $\Coker G_{\clopint{0}{n}}$ at the generic point, where $G_{\clopint{0}{i}} := G_0 \circ \dots \circ G_{i-1}$.
Then $\Coker G_{\clopint{0}{n}}$ is also locally-free and
there is a unique isomorphism 
${\bigotimes_i (\det \Coker G_i)} \isomto {\det \Coker G_{\clopint{0}{n}}}$
taking $(v_i)_i$ to $\bigwedge_i G_{\clopint{0}{i}}(v_i)$
for local sections $v_i$ of $V_i$.
\end{lem}
\begin{proof}
For $0 \leq p \leq q \leq n$, let $G_{\clopint{p}{q}} := G_p \circ G_{p+1} \circ \dots \circ G_{q-1}$.
The assumption on the rank implies that 
$\Coker G_{\clopint{p}{q}}$ has rank equal to $\sum_{i \in \clopint{p}{q}} \rank \Coker G_i$ at the generic point.
We show the following.
\begin{enumerate}
\item \label{item:exterior:locally-free} For $p \leq r$, $\Coker G_{\clopint{p}{r}}$ is locally-free.
\item \label{item:exterior:exact} For $p \leq q \leq r$, the sequence 
 $0 \to \Coker G_{\clopint{q}{r}} \namedto{\beta} \Coker G_{\clopint{p}{r}} \to \Coker G_{\clopint{p}{q}} \to 0$ is exact.
\end{enumerate}
(\ref{item:exterior:locally-free}) is clear if $r - p \leq 1$.
(\ref{item:exterior:exact}) is clear if $p = q$ or $q = r$.
It suffices to show that if $p < q < r$
and (\ref{item:exterior:locally-free}) holds for $(p,q)$ and $(q,r)$
then (\ref{item:exterior:locally-free}) holds for $(p,r)$ 
and (\ref{item:exterior:exact}) holds for $(p,q,r)$.
The exactness at the middle and the right is clear.
Since $\Ker \beta$ is a subsheaf of a locally-free sheaf (by the assumption)
and its rank at the generic point is $0$, we have $\Ker \beta = 0$.
Thus (\ref{item:exterior:exact}) is true by the assumptions, and this together with the induction hypothesis imply (\ref{item:exterior:locally-free}).

Now, from (\ref{item:exterior:exact}) we obtain isomorphisms 
$\det \Coker G_{\clopint{p}{q}} \otimes \det \Coker G_{\clopint{q}{r}} \isomto \det \Coker G_{\clopint{p}{r}} \colon v \otimes w \mapsto v \wedge G_{\clopint{p}{q}}(w)$.
Composing these isomorphisms inductively, we obtain the desired isomorphism.
\end{proof}

\begin{prop} \label{prop:n-step}
	Suppose $X_0 \namedto{\pi_0} X_1 \namedto{\pi_1} \dots \namedto{\pi_{m-1}} X_m = \pthpower{X_0}$
	is a sequence of purely inseparable morphisms of degree $p$ between $m$-dimensional integral normal varieties,
	with each $\pi_i$ given by a $p$-closed rational derivation $D_i$ on $X_i$.
	Then $K_{X_0} \sim -\sum_{i = 0}^{m-1} (\pi_{i-1} \circ \dots \circ \pi_0)^* \divisorialfix{D_i}$.
\end{prop}

\begin{proof}
	As the conclusion does not depend on closed subschemes of codimension $\geq 2$,
	we may assume that $\Sing(X_i) = \emptyset$ and $\isolatedfix{D_i} = \emptyset$ by restricting to the complement.

We write $\map{\pi_{\clopint{0}{i}} := \pi_{i-1} \circ \dots \circ \pi_1 \circ \pi_0}{X_0}{X_i}$
and let 
$\map{G_i}{\pi_{\clopint{0}{i+1}}^* \Omega^1_{X_{i+1}}}{\pi_{\clopint{0}{i}}^* \Omega^1_{X_i}}$
be the pullback of 
$\map{\pi_i^*}{\pi_i^* \Omega^1_{X_{i+1}}}{\Omega^1_{X_i}}$
to $X_0$.
Then $\Coker (G_i \otimes \cO_{X_0}(\pi_{\clopint{0}{i}}^* \divisorialfix{D_i}))$ is free of rank $1$,
since it is the pullback of $\Coker (\pi_i^* \otimes \cO_{X_i}(\divisorialfix{D_i}))$,
which is free of rank $1$ by Lemma \ref{lem:canonical 1-forms}.
Since $G_0 \circ \dots \circ G_{m-1} = 0$, 
we can apply Lemma \ref{lem:exterior} to $G_0 \circ \dots \circ G_{m-1}$.
Then the invertible sheaf 
\begin{align*}
\Omega^m_{X_0} \otimes \cO_{X_0} \Bigl( \sum_i \pi_{\clopint{0}{i}}^* \divisorialfix{D_i} \Bigr)
&= (\Coker (G_0 \circ \dots \circ G_{m-1})) 
\otimes \bigotimes_i \cO_{X_0}(\pi_{\clopint{0}{i}}^* \divisorialfix{D_i}) \\
&\cong \bigotimes_i \Coker (G_i \otimes \cO_{X_0}(\pi_{\clopint{0}{i}}^* \divisorialfix{D_i}))
\end{align*}
is trivial.
\end{proof}

The next proposition, which we will use in Section \ref{sec:restoring},
is a slight generalization 
of arguments in \cite{Bombieri--Mumford:III}*{Sections 3 and 5}
(where only derivations of multiplicative or additive type are considered).
\begin{prop} \label{prop:covering}
	Let $D$ be a nontrivial $p$-closed derivation on an integral scheme $X$, 
	and let $\map{\pi}{X}{X^D = Y}$ be the quotient map.
	Suppose $\Fix(D) \subset \pi^{-1}(\Sing(Y))$
	and $\Sing(X) \subset \pi^{-1}(Y^{\sm})$.
	Then,
	\begin{enumerate}
		\item \label{item:Gorenstein}
		$X$ is Gorenstein.
		\item \label{item:1-form}
		There is a canonical closed $1$-form $\eta$ on $Y^{\sm}$
		that coincides with the one given in Lemma \ref{lem:canonical 1-forms} on $Y^{\sm} \cap \pi(X^{\sm})$.
		It satisfies $\Sing(X) = \pi^{-1}(\Zero(\eta))$.
		$X$ is normal if and only if $\codim \Zero(\eta) \geq 2$. 
		\item \label{item:dualizing}
		Suppose $X$ and $Y$ are proper, $Y$ admits a dualizing sheaf $\omega_Y$,
		and it is trivial ($\omega_Y \cong \cO_Y$).
		Then $X$ admits a dualizing sheaf $\omega_X$ and it is trivial.
		\item \label{item:derivation}
		Suppose $X$ and $Y$ are surfaces. 
		Suppose $Y^{\sm}$ admits a global non-vanishing $2$-form $\omega$, and fix such a $2$-form.
		Then there is a unique $p$-closed derivation $D_Y$ on $Y$
		satisfying $D_Y(f) \omega = df \wedge \eta$ on $Y^{\sm}$.
		It moreover satisfies $\Zero(\eta) = \Fix(D_Y \restrictedto{Y^{\sm}})$,
		$Y^{D_Y} = \pthpower{(\normalization{X})}$,
		and $D_Y(\omega) = 0$.
	\end{enumerate}
\end{prop}
\begin{proof}
	First note that $Y$ is normal.
	Indeed, for each point $y \in \Sing(Y)$, the point $\pi^{-1}(y) \in X$ is smooth by assumption,
	in particular normal, and normality inherits to derivation quotients.

	(\ref{item:1-form})
	Let $Y' = Y^{\sm}$ and $X' = \pi^{-1}(Y')$.
	Since $\Fix(D) \cap X' = \emptyset$ there exists locally a section 
	$s \in \cO_{X'}$ with $D(s) \in \cO_{X'}^*$.
	Consider the $1$-form $\eta = d(s^p) / D(s)^p$ on $Y'$.
	By Lemma \ref{lem:canonical 1-forms}, the restriction of $\eta$ 
	to $\pi(X^{\sm}) \cap Y'$ (which is dense since $X$ is integral) 	
	does not depend on the choice of $s$, is defined globally, and is killed by $d$,
	hence $\eta$ itself satisfies the same properties.
	
	Two special cases are the following.
	If $D$ is of multiplicative type
	then we can take $s$ satisfying $D(s) = s$ (\cite{Matsumoto:k3mun}*{Lemma 2.13}), 
	and then $\eta = \dlog(s^p)$.
	If $D$ is of additive type
	then we can take $s$ satisfying $D(s) = 1$ (Lemma \ref{lem:D=1}), 
	and then $\eta = d(s^p)$.

	By assumption $X$ is regular above $\Sing(Y)$.
	Locally on $Y'$, we have $\cO_{X'} = \cO_{Y'}[S] / (S^p - b)$, where $b = s^p$.
	Hence $X$ is complete intersection, in particular Gorenstein, and 
	we have $\Sing(X) = \pi^{-1}(\Zero(db)) = \pi^{-1}(\Zero(\eta))$.

	Since $X$ is regular at the generic point, $\eta$ is not identically $0$.
	$X$ is normal if and only if $\Sing(X)$ or equivalently $\Zero(\eta)$ is of codimension $> 1$.

	(\ref{item:Gorenstein}) is proved above.

	(\ref{item:dualizing})
	Since $X$ is proper and $\codim \Fix(D) \geq 2$,
	we have $h \in k$, where $D^p = h D$.
	We may assume $h \in \set{0,1}$.
	We follow \cite{Bombieri--Mumford:III}*{Proposition 9}.
	It suffices to give an $\cO_Y$-linear isomorphism 
	$\map{\phi}{\pi_* \cO_X}{\sHom(\pi_* \cO_X, \cO_Y)}$.
	Let $\phi$ be the morphism $x \mapsto t(x \cdot \functorspace)$, where
	$\map{t = \pr_0}{\pi_* \cO_X}{(\pi_* \cO_X)^{D = 0}} = \cO_Y$ if $D$ is of multiplicative type (i.e.\ $h = 1$) and
	$\map{t = D^{p-1}}{\pi_* \cO_X}{\cO_Y}$ if $D$ is of additive type (i.e.\ $h = 0$).
	Since $\Fix(D) \cap X' = \emptyset$, $\phi \restrictedto{Y'}$ is an isomorphism,
	and then $\phi$ itself is an isomorphism 
	since $\pi_* \cO_X$ and $\cO_Y$ are normal at $\Sing(Y)$.

	(\ref{item:derivation})
	We define a derivation $D_{Y'}$ on $Y' = Y^{\sm}$ by 
	$D_{Y'} \colon \cO_{Y'} \namedto{d} \Omega_{Y'}^1 \namedto{\wedge \eta} \Omega_{Y'}^2 \twonamedfrom{\otimes \omega}{\sim} \cO_{Y'}$,
	hence $D_{Y'}(f) \omega = df \wedge \eta$.
	Then $\Fix(D_{Y'}) = \Zero(\eta)$.
	Write $\cO_{X'} = \cO_{Y'}[S] / (S^p - b)$ locally on $Y'$ as in the proof of (\ref{item:1-form})
	and then $\eta = u \cdot db$ for a unit $u \in \cO_{Y'}^*$.
	Then it is clear that $b \in \cO_{Y'}^{D_{Y'}}$ and $\pthpower{(\cO_{Y'})} \subset \cO_{Y'}^{D_{Y'}}$,
	hence $\pthpower{(\cO_{X'})} \subset \cO_{Y'}^{D_{Y'}}$.
	Since $\cO_{Y'}^{D_{Y'}}$ is normal (since $\cO_Y$ is normal) we obtain 
	${\pthpower{(\normalization{(\cO_{X'})})}} \subset \cO_{Y'}^{D_{Y'}}$. 
	Since $Y$ is normal, $D_{Y'}$ extends to a derivation $D_Y$ on $Y$ by Lemma \ref{lem:extending derivation},
	and we have ${\pthpower{(\normalization{(\cO_{X})})}} \subset \cO_{Y}^{D_Y}$. 
	Comparing the degree with respect to $k(X)$
	($p^2 = [k(X) : k(\pthpower{X})] \geq [k(X) : k(Y^{D_Y})] 
	= [k(X) : k(Y)] \cdot [k(Y) : k(Y^{D_Y})] \geq p^2$)
	we observe that this is equality at the generic point,
	and then since both sides are normal we obtain the equality.
	We also obtain $[k(Y) : k(Y^{D_Y})] = p$ and hence $D_Y$ is $p$-closed.
	
	We have $D_Y(\eta) = 0$ since $\eta$ is the pullback of a $1$-form on $\pthpower{\cO_X} \subset \cO_Y^{D_Y}$.
	Comparing $D_Y(D_Y(f) \omega) = D_Y(df \wedge \eta)$ and
	$D_Y(D_Y(f)) \omega = d(D_Y(f)) \wedge \eta$
	(both of which follow from $D_Y(f) \omega = df \wedge \eta$),
	we obtain $D_Y(\omega) = 0$.
\end{proof}

\section{Local properties of derivations on smooth points and RDPs} \label{sec:derivations on RDP}

In this section we will recall basic properties of RDPs and
then consider derivations on RDPs.

\begin{defn}[RDPs]
	\emph{Rational double point} singularities in dimension $2$, RDPs for short,
	are the $2$-dimensional canonical singularities.
	
	The exceptional curves of the resolution of singularity and their intersection numbers form a Dynkin diagram of type $A_n$, $D_n$, or $E_n$.
	We say that the RDP is of type $A_n$, $D_n$, or $E_n$.
	For 
	$D_n$ and $E_n$ in characteristic $2$,
	$E_n$ in characteristic $3$,
	and $E_8$ in characteristic $5$,
	and in no other cases,
	there are more than one, finitely many, isomorphism classes of singularity sharing the same Dynkin diagram.
	They are classified and named as $D_n^r$ and $E_n^r$ by Artin \cite{Artin:RDP}, 
	where the range of $r$ is a certain finite set of non-negative integers 
	depending on the characteristic and the Dynkin diagram.
	In these cases, and also in the cases of $A_n$ with $p \divides (n+1)$ 
	and $D_n$ with $p \divides (n-2)$, and in no other cases,
	the fundamental groups are different from those of the corresponding RDPs in characteristic $0$, again see \cite{Artin:RDP}.
	
	We refer to $n$ and $r$ as the \emph{index} and \emph{coindex} of the RDP.
	
	If $A$ is the localization of a surface at an RDP, or the completion of such an algebra,
	then we call $\Spec A$ a \emph{local RDP} for short.
	
	If $\Spec A$ is a local RDP or a $2$-dimensional regular local ring,
	then we denote $\Pic(A) = \Pic((\Spec A)^{\sm})$ and call this the local Picard group of $A$.
	If $A$ is Henselian (e.g.\ if it is complete) 
	then by \cite{Lipman:rationalsingularities}*{Proposition 17.1},
	this group is determined from the Dynkin diagram as in Table \ref{table:Picard group of RDP} and is independent of the characteristic and the coindex. 
\end{defn}

\begin{table}
	\caption{Local Picard groups of Henselian RDPs (in any characteristic)} \label{table:Picard group of RDP}
	\begin{tabular}{ccccccc}
		\toprule
		smooth & $A_n$ & $D_{2m}$ & $D_{2m+1}$ & $E_6$ & $E_7$ & $E_8$ \\
		\midrule
		$0$ & $\bZ/(n+1)\bZ$ & $(\bZ/2\bZ)^2$ & $\bZ/4\bZ$ & $\bZ/3\bZ$ & $\bZ/2\bZ$ & $0$ \\
		\bottomrule
	\end{tabular}
\end{table}

\begin{defn}[RDP surfaces]
\emph{RDP surfaces}
are surfaces that have only RDPs as singularities (if any).
In particular, any smooth surface is an RDP surface.

\emph{RDP K3 surfaces}
are proper RDP surfaces whose minimal resolutions are (smooth) K3 surfaces. 
We similarly define \emph{RDP Enriques surfaces}.

Since abelian surfaces and (quasi-)hyperelliptic surfaces do not admit smooth rational curves,
any RDP abelian or RDP (quasi-)hyperelliptic surface is smooth.
\end{defn}

\begin{thm} \label{thm:mu_p alpha_p RDP}
Let $X$ be a surface equipped with a nontrivial $p$-closed derivation $D$, 
and $w \in X$ a closed point.
Let $\pi \colon X \to Y = X^D$ be the quotient morphism.
\begin{enumerate}
\item \label{thm:mu_p alpha_p RDP:non-fixed}
Assume $w \notin \Fix(D)$.
If $w$ is a smooth point then $\pi(w)$ is also a smooth point.
If $w$ is an RDP then $\pi(w)$ is either a smooth point or an RDP,
and more precisely $(\hat{\cO}_{X,w}, D)$ is isomorphic to $(k[[x,y,z]]/(F), u \cdot \partial / \partial z)$ 
where $u$ is a unit and 
$F$ is a power series $\in k[[x,y,z^p]]$ that is 
one in Table \ref{table:non-fixed RDP}.
In either case $X \times_Y \tilde{Y} \to X$ is crepant, where $\tilde{Y} \to Y$ is the minimal resolution at $\pi(w)$.
\item \label{thm:mu_p alpha_p RDP:fixed RDP symplectic}
If $w \in \Fix(D)$, then $D$ uniquely extends to a derivation $D_1$ on $X_1 = \Bl_w X$.
Suppose moreover that $\divisorialfix{D} = 0$,
that $w$ is an RDP, and that $\pi(w)$ is either a smooth point or an RDP.
Then $\pi(w)$ is an RDP,
$\divisorialfix{D_1} = 0$,
the image of each point above $w$ is either a smooth point or an RDP,
$g \colon Y_1 = (X_1)^{D_1} \to Y$ is crepant,
and $\Fix(D_1) \neq \emptyset$.
\end{enumerate}
\end{thm}

\begin{proof}
(\ref{thm:mu_p alpha_p RDP:non-fixed})
Assume $w$ is a smooth point
(this case is already proved in \cite{Seshadri:Cartier}*{Proposition 6}).
Taking a coordinate $x,y$ as in Lemma \ref{lem:good coord} 
(i.e.\ $D(x) = 0$ and $D(y) \in \cO_{X,w}^*$),
we have $\hat \cO_{Y,\pi(w)} \cong k[[x,y^p]]$,
hence $\cO_{Y,\pi(w)}$ is smooth.

Assume $w$ is an RDP.
By Lemma \ref{lem:good coord} we have a coordinate $x,y,z$ satisfying 
$D(x) = D(y) = 0$ and 
$D(z) \in \cO_{X,w}^*$.

We recall the classification \cite{Matsumoto:k3mun}*{Proposition 4.8} 
of all formal power series $F \in k[[x,y,z^p]]$
such that $k[[x,y,z]] / (F)$ defines an RDP at the origin,
up to multiples by units, 
and up to coordinate change preserving the invariant subalgebra $k[[x,y,z^p]] \subset k[[x,y,z]]$.
The result is displayed in Table \ref{table:non-fixed RDP}. 
We observed that in each case $\pi(w)$ is either a smooth point or an RDP and that $X \times_Y \tilde{Y}$ is an RDP surface crepant over $X$,
where $\tilde{Y} \to Y$ is the resolution at $\pi(w)$.
(The entries of the singularities of $X \times_Y \tilde{Y}$ is omitted if $Y$ is already smooth.)

(\ref{thm:mu_p alpha_p RDP:fixed RDP symplectic})
Take a $2$-form $\chi$ on $Y$, nonzero on a neighborhood of $\pi(w)$.
Let $\omega$ be the $D$-invariant $2$-form on $X$ 
corresponding to $\chi$ under the isomorphism in Proposition \ref{prop:2-forms}.
Let $\omega_1 = q^* \omega$, where $\map{q}{X_1}{X}$ is the blow-up.
Let $\chi_1$ be the $2$-form on $Y_1$ corresponding to $\omega_1$.
Then we have
\begin{align*}
\divisor(\omega)  = \pi ^*(\divisor(\chi) ) + (p-1) \divisorialfix{D}, \quad 
\divisor(\omega_1) = \pi_1^*(\divisor(\chi_1)) + (p-1) \divisorialfix{D_1},
\end{align*}
and we have
\[
\divisor(\omega_1) = q^* (\divisor(\omega)) + K_{X_1/X} , \quad 
\divisor(\chi_1)   = g^* (\divisor(\chi)  ) + K_{Y_1/Y} .
\]
By assumption we have $\divisorialfix{D} = 0$ and $K_{X_1/X} = 0$.
Hence we have 
\[
\pi_1^* K_{Y_1/Y} + (p-1) \divisorialfix{D_1} = 0.
\]
Since both terms are effective (since $\pi(w)$ is an RDP) we have $\pi^* K_{Y_1/Y} = (p-1) \divisorialfix{D_1} = 0$,
and since $\pi$ is a homeomorphism we have $K_{Y_1/Y} = 0$.
In particular $\pi(w)$ is not smooth,
and there are no non-RDP singularities above $\pi(w)$.
Finally, $\Fix(D_1) \neq \emptyset$ is proved in the same way as the corresponding assertion
in \cite{Matsumoto:k3mun}*{Lemma 4.9(2)}.
\end{proof}

\begin{table} 
\caption{Non-fixed $p$-closed derivations on RDPs} \label{table:non-fixed RDP} 
\begin{tabular}{llllll} 
\toprule
 $p$ & equation            &              & $X$          & $Y = X^D$ & $X \times_Y \tilde{Y}$ \\ 
\midrule
 any & $xy + z^{mp}$       & ($m \geq 2$) & $A_{mp-1}$   & $A_{m-1}$ & $mA_{p-1}$ \\
 any & $xy + z^{p}$        &              & $A_{p-1}$    & smooth    & --- \\
\midrule
 $5$ & $x^2 + y^3 + z^5$   &              & $E_8^0$      & smooth    & --- \\
\midrule
 $3$ & $x^2 + z^3 + y^4$   &              & $E_6^0$      & smooth    & --- \\
 $3$ & $x^2 + y^3 + yz^3$  &              & $E_7^0$      & $A_1$     & $E_6^0$  \\
 $3$ & $x^2 + z^3 + y^5$   &              & $E_8^0$      & smooth    & --- \\
\midrule
 $2$ & $z^2 + x^2y + xy^m$ & ($m \geq 2$) & $D_{2m}^0$   & smooth    & --- \\
 $2$ & $x^2 + yz^2 + xy^m$ & ($m \geq 2$) & $D_{2m+1}^0$ & $A_1$     & $D_{2m}^0$  \\
 $2$ & $x^2 + xz^2 + y^3$  &              & $E_6^0$      & $A_2$     & $D_4^0$  \\
 $2$ & $z^2 + x^3 + xy^3$  &              & $E_7^0$      & smooth    & --- \\
 $2$ & $z^2 + x^3 + y^5$   &              & $E_8^0$      & smooth    & --- \\
\bottomrule
\end{tabular} 

\end{table}

\begin{defn} [cf.\ \cite{Matsumoto:k3mun}*{Definition 4.6}] \label{def:maximal} 
We say that an RDP surface $X$ equipped with a $p$-closed derivation $D$
is \emph{maximal} at a closed point $w \in X$ (not necessarily fixed)
if either $w \in X$ is a smooth point or $\pi(w) \in X^D$ is a smooth point.

We say that $X$, or the quotient morphism $\map{\pi}{X}{Y = X^D}$, is \emph{maximal} with respect to the derivation 
if it is maximal at every closed point.
We define the maximality of $\mu_p$- and $\alpha_p$-actions similarly.
\end{defn}

\begin{cor} \label{cor:maximize}
Let $\map{\pi}{X}{Y = X^D}$ as in the previous theorem. 
Assume that $\divisorialfix{D} = 0$ and that $X$ and $Y$ are RDP surfaces.
Then there exists an RDP surface $X'$ and a derivation $D'$ on $X'$, whose quotient morphism denoted $\map{\pi'}{X'}{Y'}$,
fitting into a diagram
	\[
	\begin{tikzcd}
	X' \arrow[r,"\pi'"] \arrow[d,"g"] &
	Y' \arrow[r,equal] \arrow[d] &
	X'^{D'} \\
	X \arrow[r,"\pi"] &
	Y \arrow[r,equal]  &
	X^{D} 
	\end{tikzcd}
	\]
with $X' \to X$ and $Y' \to Y$ surjective birational and crepant, 
$D' = D$ on the isomorphic locus of $X' \to X$,
$\Fix(D')$ isolated, $g(\Fix(D')) = \Fix(D)$,
and $\pi'$ maximal.

$X'$ is characterized as the maximal partial resolution of $X$
to which the derivation extends.
If $D$ is of multiplicative type (resp.\ of additive type, resp.\ fixed-point-free), then so is $D'$.
\end{cor}
\begin{proof}
If $D$ has a fixed RDP $w$ (which is an isolated fixed point by assumption) then consider $X_1 = \Bl_w X \to X$ and $\map{\pi_1}{X_1}{X_1^{D_1} = Y_1}$.
where $D_1$ is the induced derivation on $X_1$.
By Theorem \ref{thm:mu_p alpha_p RDP}(\ref{thm:mu_p alpha_p RDP:fixed RDP symplectic}),
$D_1$ on $X_1$ satisfies the same condition, and $X_1 \to X$ and $Y_1 \to Y$ are crepant.
Repeating this finitely many times, we may assume $X_1$ has no fixed RDP.

If $D_1$ has a non-fixed RDP $w$ whose image $\pi(w)$ is an RDP,
then consider $X_2 = X_1 \times_{Y_1} \tilde{Y_1}$ and the induced derivation $D_2$, where $\tilde{Y_1} \to Y_1$ is the minimal resolution at $\pi(w)$.
Since $w \not\in \Fix(D_1)$ is equivalent to the existence of $f \in \cO_{X_1,w}$ with $D_1(f) \in \cO_{X_1,w}$, 
and since this property inherits to points above $w$,  $\Fix(D_2)$ does not meet the fiber above $w$.
Comparing $2$-forms as in the proof of Theorem \ref{thm:mu_p alpha_p RDP}(\ref{thm:mu_p alpha_p RDP:fixed RDP symplectic}),
we obtain $K_{X_2/X_1} = (p-1) \divisorialfix{D_2} = 0$.
Therefore $X_2 \to X_1$ and $Y_2 \to Y_1$ are crepant and 
$D_2$ on $X_2$ satisfies the same condition.
Repeating this for the (finitely many) points $w$,
we obtain $X'$ with the desired properties.

The characterization follows from Lemma \ref{lem:derivation on resolution}, 
which states that each exceptional curve above the remaining singularities appears in $\divisorialfix{D'}$ with nonzero coefficient.

The final assertion is obvious for multiplicative and additive type,
and for fixed-point-freeness this follows from $g(\Fix(D')) = \Fix(D)$.
\end{proof}

Next, we classify RDPs that can be written as derivation quotients of smooth points,
and give a lower bound for $\deg \isolatedfix{D}$ 
of derivations $D$ with non-RDP quotients. 
The classification, as in (\ref{case:RDP}), of such RDPs in characteristic $2$ 
is also proved by Tziolas \cite{Tziolas:alphapmup}*{Proposition 3.6}.

\begin{lem} \label{lem:degree of isolated fixed point}
Let $D$ be a nonzero $p$-closed derivation on $B = k[[x,y]]$ in characteristic $p$.
Suppose that $\Supp \Fix(D)$ consists precisely of the closed point.
Let $s = \deg \isolatedfix{D} = \dim_k B / (D(x),D(y))$.
\begin{enumerate}
\item \label{case:alpha}
If $D$ is of additive type then $s \geq 2$.
\item \label{case:RDP}
Assume $B^D$ is an RDP. 
\begin{enumerate}
	\item \label{subcase:degree}
	Then $(p, s, B^D)$ is one of those listed in Table \ref{table:derivation quotient RDP}.
	In particular, we have $s = n / (p-1)$ in every case, where $n$ is the index of the RDP.
	The table also shows an example of $D$ (satisfying $D^p = h D$) realizing each case.
	\item \label{subcase:mu}
	If $D$ is of multiplicative type,
	then $B^D$ is of type $A_{p-1}$.
	\item \label{subcase:alpha}
	If $D$ is of additive type, then 
	$(p, B^D)$ is one of $(5, E_8^0)$, $(3, E_6^0)$, $(2, D_{4m}^0)$, or $(2, E_8^0)$.
	\item \label{subcase:alpha35}
	If $D$ is of additive type and $(p, B^D) = (5, E_8^0),  (3, E_6^0)$,
	then $\Image (D^j \restrictedto{\Ker D^{j+1}})$ is equal to the maximal ideal $\fn$ of $B^D$ 
	for each $1 \leq j \leq p-1$.
\end{enumerate}

\item \label{case:alpha non-RDP}
Assume $D$ is of additive type and $B^D$ is a non-RDP.
If $p = 2$ then $s \geq 12$.
If $p = 3$ then $s > 3$.
If $p = 5$ then $s > 2$.
\end{enumerate}
\end{lem}
The following corollary is an immediate consequence of this lemma
and will be used in Section \ref{sec:mu_p alpha_p K3}.
\begin{cor} \label{cor:configuration of isolated fixed point}
Suppose $A_i = k[[x,y]]$, $1 \leq i \leq N$, are respectively equipped with derivations $D_i$ of additive type
and suppose $\Supp \Fix(D_i)$ consists precisely of the closed point for each $i$.
Let $s_i = \deg \isolatedfix{D_i} = \dim_k A_i / (D_i(x),D_i(y))$.
Assume $\sum s_i = 24 / (p+1)$.
Then either 
\begin{itemize}
\item 
$N = 1$ and $A_1^{D_1}$ is a non-RDP and $p \geq 3$, or 
\item
each $A_i^{D_i}$ is an RDP, 
and more precisely $(p, \set{A_i^{D_i}})$ is 
$(2, 2 D_4^0)$, $(2, 1 D_8^0)$, $(2, 1 E_8^0)$,
$(3, 2 E_6^0)$, or
$(5, 2 E_8^0)$.
\end{itemize}
\end{cor}

\begin{table}
\caption{RDPs arising as quotients of smooth points by $p$-closed derivations,
	and examples of derivations} \label{table:derivation quotient RDP} 
\begin{tabular}{lllll}
\toprule
$p$ & $\deg \isolatedfix{D}$
             & RDP          & example of $D(x), D(y)$               & $h$ \\
\midrule                                                
any & $1$    & $A_{p-1}$    & $x, -y$                    & $1$ \\
\midrule                                                
$5$ & $2$    & $E_8^0$      & $y, x^2$                   & $0$ \\
\midrule                                                
$3$ & $3$    & $E_6^0$      & $y^3, x$                   & $0$ \\
$3$ & $4$    & $E_8^0$      & $y^4, x$                   & $y^3$ \\
\midrule                                                
$2$ & $4m$   & $D_{4m}^0$   & $x^2, y^{2m}$              & $0$ \\
$2$ & $4m+2$ & $D_{4m+2}^0$ & $x^2 + x y^{2m}, y^{2m+1}$ & $y^{2m}$ \\
$2$ & $7$    & $E_7^0$      & $x y^2, x^2 + y^3$         & $y^2$ \\
$2$ & $8$    & $E_8^0$      & $y^4, x^2$                 & $0$ \\
\bottomrule
\end{tabular}

\end{table}

\begin{proof}[Proof of Lemma \ref{lem:degree of isolated fixed point}]
(\ref{case:alpha})
Since $D^p = 0$, it follows that $D \restrictedto{\fm/\fm^2}$ is nilpotent,
hence for some coordinate $x,y \in \fm$ we have $D(x) \in \fm^2$.

(\ref{subcase:degree}--\ref{subcase:alpha}) 
We observe that the derivation $D$ described in Table \ref{table:derivation quotient RDP} 
satisfies $\divisorialfix{D} = 0$ and $D^p = h D$, and it realizes the RDP.

Suppose $B^D$ is an RDP.
Since the composite $\Pic(B^D) \to \Pic(B) \to \Pic(\thpower{(B^D)}{1/p}) \cong \Pic(B^D)$ is equal to the $p$-th power map,
and since $\Pic(B)$ is trivial,
$\Pic(B^D)$ is a $p$-torsion group and has no nontrivial prime-to-$p$ torsion.
The natural morphism $\Spec B^D \to \Spec \pthpower{B}$ is the quotient morphism 
with respect to some rational $p$-closed derivation $D'$ on $B^D$.
	Then by the Rudakov--Shafarevich formula we have
	\[
	K_{B^D} \sim \pi^* K_{\pthpower{B}} + (p-1) \divisorialfix{D'},
	\]
	but since both canonical divisors are trivial, we have $(p-1) \divisorialfix{D'} \sim 0$,
	and by above we have in fact $\divisorialfix{D'} \sim 0$.
	Replacing $D'$ with $g^{-1} D'$ where $\divisorialfix{D'} = \divisor(g)$, 
	we may assume $D'$ is regular with $\divisorialfix{D'} = 0$.
	Then, by Theorem \ref{thm:mu_p alpha_p RDP}, the closed point is not an isolated fixed point either,
	and  $(p, B^D, D')$ is one of $(p, X, D$) listed in Table \ref{table:non-fixed RDP} with $X^D$ smooth.
	Hence, after a coordinate change, 
	$(p, D)$ is one of those listed in Table \ref{table:derivation quotient RDP}
	up to replacing $D$ by a unit multiple.
	We obtain (\ref{subcase:degree}).

It remains to check the impossibility for the derivation to be of multiplicative or additive type.
Suppose $D_1$ is a derivation on $B$ satisfying $\divisorialfix{D_1} = 0$ and realizing the RDP.
Then $D_1 = f D$ for some $f \in B^*$, where $D$ is the derivation given in Table \ref{table:derivation quotient RDP}.
By Hochschild's formula (Lemma \ref{lem:Hochschild}) we have $D_1^p = (f^{p-1} h + D(g)) D_1$
where $g = (f D)^{p-2}(f)$.
If $h \in \fm$ and $(\Image(D)) \subset \fm$ then $f^{p-1} h + D(g) \neq 1$ for any $f \in B^*$.
Thus we obtain (\ref{subcase:mu}).
If $h \not\in (\Image(D))$ then $f^{p-1} h + D(g) \neq 0$ for any $f \in B^*$.
Thus we obtain (\ref{subcase:alpha}).

(\ref{case:alpha non-RDP}) 
If $p > 5$ then there is nothing to prove.
We will check that if $p \leq 5$ and $D$ is of additive type with $s$ less than the bound then $B^D$ is an RDP.

Suppose $p = 5$ and $s = 2$.
We have $D \restrictedto{\fm/\fm^2} \neq 0$ and $(D \restrictedto{\fm/\fm^2})^2 = 0$.
We may assume $D(y) = x$, $D(x) = f = y^2 + g$, $g \in (x^2, xy, y^{3})$.
we say that the monomial $x^i y^j$ has degree $3 i + 2 j$
and let $I_n$ be the ideal generated by the monomials of degree $\geq n$.
We have $D(I_n) \subset I_{n+1}$,
$f \equiv y^2 \pmod{I_5}$, and
$D^2(f) - 2(x^2 + y^3) =: h \in {I_7}$.
Let $B' = k[[X,Y,Z]] / (-Z^5 + 2(X^2 + Y^3) + h^5) \subset B^D$
where $X = x^5$, $Y = y^5$, $Z = D^2(f) = D^4(y)$.
Since $h^5 \in \thpower{I_7}{5} = (X^3, X^2 Y, X Y^2, Y^4)_{k[[X,Y]]}$,
$B'$ is normal and hence $B' = B^D$, 
and it is an RDP of type $E_8^0$.

Suppose $p = 3$ and $s = 2,3$.
We have $D \restrictedto{\fm/\fm^2} \neq 0$ and $(D \restrictedto{\fm/\fm^2})^2 = 0$.
We may assume $D(y) = x$, $D(x) = f$, $D(f) = 0$,
$f = y^s + g$, $g \in (x^2, xy, y^{s+1})$.
Then since $D(f) = 0$ it follows that $s \neq 2$,
hence $s = 3$, and that $g \in (x^2, xy^2, y^4)$.
We say that the monomial $x^i y^j$ has degree $2 i + j$
and let $I_n$ be the ideal generated by the monomials of degree $\geq n$.
We have $D(I_n) \subset I_{n+1}$ and $g \in I_4 = (x^2, x y^2, y^4)$.
Let $B' = k[[X,Y,Z]] / (-Z^3 + X^2 + Y f^3) \subset B^D$
where $X = x^3$, $Y = y^3$, $Z = x^2 + yf$.
Since $f^3 = Y^3 + g^3$ with $g^3 \in \thpower{I_4}{3} = (X^2, XY^2, Y^4)_{k[[X,Y]]}$,
$B'$ is normal and hence $B' = B^D$, 
and it is an RDP of type $E_6^0$.

Suppose $p = 2$.
By Theorem \ref{thm:derivation2} there exists $h \in k[[x,y]]^*$ and $R,S,T \in k[[x,y]]$ such that 
$D' = h^{-1} D$ satisfies
$D'(x) = S^2 + T^2 x$, $D'(y) = R^2 + T^2 y$, and $D'^2 = T^2 D'$.
(This derivation $D'$ is $p$-closed but not necessarily of additive type.)
Suppose $s < 12$ and that $B^D = B^{D'}$ is not an RDP.
	Then by Corollary \ref{cor:derivation2-12} we have
	$R, S \in \fm^2$, $T \in \fm$, and $T \not\in \fm^2$.
	Since $D = h D'$ is of additive type we have $D'(h) + h T^2 = 0$,
	but this is impossible since $\Image (D') \subset \fm^3$ and $h T^2 \not\in \fm^3$.

(\ref{subcase:alpha35})
We use the description given in the proof of (\ref{case:alpha non-RDP}).
Suppose $D$ is additive and $(p, B^D) = (3, E_6^0), (5, E_8^0)$.
Since $\Image(D^{p-1}) \subset \Image(D^j \restrictedto{\Ker D_{j+1}}) \subset \Image(D \restrictedto{\Ker D^2}) 
\subset \fm \cap B^D = \fn$,
it suffices to show $\fn \subset \Image(D^{p-1})$.
If $(p, s) = (5, 2)$,
a straightforward calculation yields
$D^4(y) = Z$, 
$D^4(x^2) \equiv y^5 = Y \pmod{I_{11} \cap B^D}$,
$D^4(x^3 y) \equiv x^5 = X \pmod{I_{16} \cap B^D}$.
Since the initial terms of the elements $Z,Z^2,Y,X$ have different degrees $6,12,10,15$,
these elements are linearly independent modulo $I_{15+1}$,
hence $D^4(y),D^4(x^2),D^4(x^3 y)$ generate $\fn/\fn^2$.
If $(p, s) = (3, 3)$,
a straightforward calculation yields
$D^2(y) = Y + g$ ($g \in I_4 \cap B^D$),
$D^2(y^2) = 2 Z$, 
$D^2(x y^2) = 2 X$.
Clearly these elements generate $\fn/\fn^2$.
\end{proof}

\begin{thm} \label{thm:derivation2}
	Let $k$ be an algebraically closed field of characteristic $2$.
	Let $D$ be a nonzero $p$-closed derivation on $B = k[[x,y]]$.
	Then there exist $h, R, S, T \in k[[x,y]]$,
	such that $D = h D'$ where $D'$ is the $p$-closed derivation defined by 
	$D'(x) = S^2 + T^2 x$ and $D'(y) = R^2 + T^2 y$.
	It follows that 
	\begin{align*}
	B^D = B^{D'} &= k[[x^2, y^2, R^2 x + S^2 y + T^2 xy]] \\
	&\cong k[[X, Y, Z]] / (Z^2 + (\secondpower{R})^2 X + (\secondpower{S})^2 Y + (\secondpower{T})^2 XY),
	\end{align*}
	where $X = x^2$, $Y = y^2$, $Z = R^2 x + S^2 y + T^2 xy$.
	We have $D'^2 = T^2 D'$ and $D^2 = (D'(h) + h T^2) D$.
\end{thm}
Here $\secondpower{R} = \secondpower{R}(X,Y) \in k[[X,Y]]$ is the power series 
satisfying $\secondpower{R}(x^2, y^2) = R(x,y)^2$, 
and $\secondpower{S}$ and $\secondpower{T}$ are defined in the same way.

We can give a classification of quotient singularities with small $\deg \isolatedfix{D}$,
using which we can complete the proof of Lemma \ref{lem:degree of isolated fixed point}.
\begin{cor} \label{cor:derivation2-12}
	Let $D$, $h$, $R$, $S$, $T$, and $D'$ be as in the previous theorem.
	Assume $\divisorialfix{D} = 0$.
	\begin{enumerate}
		\item If $R$ or $S$ is a unit, then $B^D$ is smooth and $\deg \isolatedfix{D} = 0$.

		Hereafter we assume this is not the case, and we implicitly make similar assumptions cumulatively.
		
		\item If $T$ is a unit, then $B^D$ is an RDP of type $A_1$.
		\item If $R$ and $S$ generate $\fm$, 
		then $B^D$ is an RDP of type $D_4^0$.
		\item Suppose $R$ and $S$ generate a $1$-dimensional subspace of $\fm/\fm^2$.
		We may assume $R \not\in \fm^2$ and $S \in \fm^2$.
		Suppose moreover that $x$ and $R$ generate $\fm$.
		Let $m = \dim_k B / (R, S)$ and $n = \dim_k B / (R, T)$ (so $2 \leq m \leq \infty$ and $1 \leq n \leq \infty$).
		Since $\divisorialfix{D'} = 0$, 
		at least one of $m$ and $n$ is finite. (e.g.\ $(R,S,T) = (y, x^m, 0), (y, 0, x^n)$.)
		Then $B^D$ is an RDP of type $D_{\min \set{4m,4n+2}}^0$.
		\item Suppose $R \not\in \fm^2$, $S \in \fm^2$, and that $x$ and $R$ do not generate $\fm$.
		\begin{itemize}
			\item If $\dim_k B / (R, T) = 1$ (e.g.\ $(R,S,T) = (x,0,y)$), then $B^D$ is an RDP of type $E_7^0$.
			\item If $\dim_k B / (R, T) > 1$ and $\dim_k B / (R, S) = 2$ (e.g.\ $(R,S,T) = (x,y^2,0)$), then $B^D$ is an RDP of type $E_8^0$.
			\item If $\dim_k B / (R, T) > 1$ and $\dim_k B / (R, S) = 3$ (e.g.\ $(R,S,T) = (x,y^3,0)$), 
			then $B^D$ is an elliptic double point of the form $Z^2 + X^3 + Y^7 + \varepsilon = 0$,
			where $\varepsilon \in (X^5, X^3 Y, X^2 Y^3, X Y^4, Y^9)$,
			and $\deg \isolatedfix{D} = 12$.
		\end{itemize}
		\item Suppose $R, S \in \fm^2$, $T \not\in \fm^2$.
		We may assume $T \equiv x \pmod {\fm^2}$.
		\begin{itemize}
			\item If $\dim_k B / (T, S) = 2$ (e.g. $(R,S,T) = (0, y^2, x)$),
			then $B^D$ is an elliptic double point of the form $Z^2 + X^3 Y + Y^5 + \varepsilon = 0$,
			where $\varepsilon \in (X^5, X^4 Y, X^3 Y^2, X^2 Y^3, X Y^4, Y^7)$,
			and $\deg \isolatedfix{D} = 11$.
			\item If $\dim_k B / (T, S) > 2$ and $\dim_k B/ (T, R) = 2$ (e.g. $(R,S,T) = (y^2, 0, x)$),
			then $B^D$ is an elliptic double point of the form $Z^2 + X^3 Y + X Y^4 + \varepsilon = 0$,
			where $\varepsilon \in (X^5, X^4 Y, X^3 Y^2, X^2 Y^3, X Y^5, Y^7)$,
			and $\deg \isolatedfix{D} = 12$.
		\end{itemize}
		\item In all other cases, $\deg \isolatedfix{D} > 12$ and $B^D$ is not an RDP.
	\end{enumerate}
	If $B^D$ is $A_n$, $D_n$, or $E_n$, 
	then we have $\deg \isolatedfix{D} = n$.
\end{cor}

\begin{proof}[Proof of Theorem \ref{thm:derivation2}]
	$B^D$ satisfies $k[[x^2, y^2]] \subset B^D \subset k[[x,y]]$
	and hence there exists $f \in k[[x,y]]$ such that $B^D = k[[x^2, y^2, f]]$.
	Write $f = Q^2 + R^2 x + S^2 y + T^2 xy$ with $Q,R,S,T \in k[[x^2,y^2]]$. 
	We have $\gcd(Q,R,S,T) = 1$.
	We may assume $Q = 0$.
	
	Since $D(f) = 0$ we have $(R^2 + T^2 y) D(x) + (S^2 + T^2 x) D(y) = 0$.
	There exists $h \in \Frac B$ such that $D(x) = (S^2 + T^2 x) h$ and $D(y) = (R^2 + T^2 y) h$.
	It remains to show $h \in B$.

	It suffices to show that $R^2 + T^2 y$ and $S^2 + T^2 x$ have no nontrivial common factor.
	Suppose there exists an irreducible non-unit power series $P \in k[[x,y]]$ 
	dividing both $S^2 + T^2 x$ and $R^2 + T^2 y$.
	Since $P$ does not divide $T$ (since $\gcd(R,S,T) = 1$), 
	we have $x = S^2/T^2$ and $y = R^2/T^2$ in the quotient ring $B / P$, hence $B / P = \secondpower{(B/P)}$, contradiction.	
\end{proof}
\begin{proof}[Proof of Corollary \ref{cor:derivation2-12}]
Straightforward.
\end{proof}

\begin{conv} \label{conv:exceptional curves}
We use the following numbering for the exceptional curves of the resolutions of RDPs. 
\begin{itemize}
\item $A_n$: $e_1, \dots, e_n$, where $e_i \cdot e_{i+1} = 1$.
\item $D_n$: $e_1, \dots, e_n$, where 
	$\set{(i,j) \mid i < j, \, e_i \cdot e_j = 1} = 
	\set{(1,2), \dots, (n-2, n-1)} \cup \set{(n-2,n)}$.
\item $E_6$: $e_1, e_{2\pm}, e_{3\pm}, e_4$, where
$e_1 \cdot e_4 = e_{2+} \cdot e_{3+} = e_{2-} \cdot e_{3-} = e_{3\pm} \cdot e_4 = 1$.
\item $E_7$: $e_1, \dots, e_7$, where 
	$\set{(i,j) \mid i < j, \, e_i \cdot e_j = 1} = 
	\set{(1,2), \dots, (5,6)} \cup \set{(4,7)}$.
\item $E_8$: $e_1, \dots, e_8$, where 
	$\set{(i,j) \mid i < j, \, e_i \cdot e_j = 1} = 
	\set{(1,2), \dots, (6,7)} \cup \set{(5,8)}$.
\end{itemize}
\end{conv}

\begin{lem} \label{lem:derivation on resolution}
Let $X = \Spec B$ be a local RDP of index $n$ in characteristic $p$,
equipped with a $p$-closed derivation $D$,
with $\Fix(D) = \emptyset$ and $X^D = \Spec B^D$ smooth.
Let $\tilde{X}$ be the resolution of $X$ and $\tilde{D}$ the rational derivation on $\tilde{X}$ induced by $D$.
Then $(\tilde{D})^2 = -2n/(p-1)$ and $\deg \isolatedfix{\tilde{D}} = n(p-2)/(p-1)$.
\end{lem}
\begin{proof}
For each case of $(p, \Sing(X))$,
a straightforward computation yields the following description of $(\tilde{D})$ and $\isolatedfix{\tilde{D}}$,
from which the stated equalities follow.
The cases for $p = 2$ also appear in \cite{Ekedahl--Hyland--Shepherd-Barron}*{Lemma 6.5}.

If $p = 2$, then $\isolatedfix{\tilde{D}} = 0$.
For every case, each closed point in $\Supp \isolatedfix{\tilde{D}}$ appears with degree $1$, so we write only the support.
We denote by $q_{ij}$ the intersection of $e_i$ and $e_j$,
 and by $q'_i$ a certain point on $e_i$ (not lying on the other components).

$(p, A_{p-1})$:
$\divisorialfix{\tilde{D}} = - \sum e_i$,
$\isolatedfix{\tilde{D}} = \set{q_{i, i+1} \mid 1 \leq i \leq p-2}$.

$(2, D_{2m}^0)$:
$\divisorialfix{\tilde{D}} = - (\sum_{i = 1}^{m-1} (2i e_{2i-1} + 2i e_{2i}) + m e_{2m-1} + m e_{2m})$.

$(2, E_{7}^0)$:
$\divisorialfix{\tilde{D}} = -(3 e_1 + 4 e_2 + 7 e_3 + 8 e_4 + 6 e_5 + 2 e_6 + 5 e_7)$.

$(2, E_{8}^0)$:
$\divisorialfix{\tilde{D}} = -(2 e_1 + 6 e_2 + 8 e_3 + 12 e_4 + 14 e_5 + 10 e_6 + 4 e_7 + 8 e_8)$.

$(3, E_{6}^0)$:
$\divisorialfix{\tilde{D}} = -(2 e_1 + 2 e_{2+} + 2 e_{2-} + 3 e_{3+} + 3 e_{3-} + 3 e_4)$,
$\isolatedfix{\tilde{D}} = \set{q'_1, q_{2+,3+}, q_{2-,3-}}$.

$(3, E_{8}^0)$:
$\divisorialfix{\tilde{D}} = -(2 e_1 + 3 e_2 + 6 e_3 + 8 e_4 + 9 e_5 + 7 e_6 + 4 e_7 + 5 e_8)$,
$\isolatedfix{\tilde{D}} = \set{q'_1, q_{34}, q_{67}, q'_8}$.

$(5, E_{8}^0)$:
$\divisorialfix{\tilde{D}} = -(2 e_1 + 3 e_2 + 4 e_3 + 5 e_4 + 5 e_5 + 4 e_6 + 2 e_7 + 3 e_8)$,
$\isolatedfix{\tilde{D}} = \set{q_{12}, q_{23}, q_{34}, q_{67}, q'_7, q'_8}$.
\end{proof}

\section{\texorpdfstring{$\mu_p$- and $\alpha_p$-actions}{mu\_p- and alpha\_p-actions} on RDP K3 surfaces} \label{sec:mu_p alpha_p K3}

\begin{prop} \label{prop:structure of quotient}
Let $G = \mu_p$ or $G = \alpha_p$.
Let $X$ be an RDP K3 surface or an RDP Enriques surface
equipped with a nontrivial $G$-action
and let $D$ be the corresponding derivation.
If the divisorial part $\divisorialfix{D}$ of $\Fix(D)$ is zero and each point in $\pi(\Fix(D))$ is either smooth or an RDP,
then $X/G$ is an RDP K3 surface or an RDP Enriques surface.
Otherwise, $X/G$ is a (possibly singular) rational surface. 

If $X$ is an RDP K3 surface,
then $X/G$ is an RDP Enriques surface if and only if the $G$-action is fixed-point-free ($\Fix(D) = \emptyset$), 
and in this case we have $p = 2$.
\end{prop}
\begin{proof}
Let $Y = X/G$.
By the Rudakov--Shafarevich formula, $\pi^* K_Y \sim K_X - (p-1) \divisorialfix{D}$,
hence $K_Y \leq 0$ in $(\Pic(Y) \otimes \bQ) / {\equiv}$, and $K_Y \equiv 0$ if and only if $\divisorialfix{D} = 0$.
We have $\Sing(Y) \subset \pi(\Sing(X) \cup \Fix(D))$,
and each point of $\pi(\Sing(X) \setminus \Fix(D))$ is either a smooth point or an RDP by Theorem \ref{thm:mu_p alpha_p RDP}(\ref{thm:mu_p alpha_p RDP:non-fixed}).
Let $\map{\rho}{\tilde{Y}}{Y}$ be the resolution.
Then $K_{\tilde{Y}} \leq \rho^* K_Y$ and the equality holds if and only if $\Sing(Y)$ consists only of RDPs.
We deduce that $K_{\tilde{Y}} \equiv 0$ if and only if $\divisorialfix{D} = 0$ and each point in $\pi(\Fix(D))$ is either smooth or an RDP.
In this case $Y$ is a proper RDP surface with $\kappa(\tilde{Y}) = 0$.
Otherwise we have $\kappa(\tilde{Y}) = -\infty$. 

Next we will show that $Y$ is not birational to abelian, (quasi-)hyperelliptic, or non-rational ruled surface.
Since $\pi$ is purely inseparable we have $b_1(\tilde{X}) = b_1(X) = b_1(Y) = b_1(\tilde{Y})$,
where $b_i = \dim_{\bQ_l} \Het^i(\functorspace, \bQ_l)$ are the $l$-adic Betti numbers for an auxiliary prime $l \neq p$.
Since $\tilde{X}$ is a K3 surface or an Enriques surface we have $b_1(\tilde{X}) = 0$.
Hence $\tilde{Y}$ is not abelian, (quasi-)hyperelliptic, nor non-rational ruled, since such surfaces have $b_1 > 0$.
Thus the first assertion follows.

Suppose $X$ is an RDP K3 surface.
To show the equivalence of freeness and Enriques quotient,
we may assume that $\Fix(D)$ is isolated and that, by Corollary \ref{cor:maximize}, $\pi$ is maximal.
By the equality $s = n / (p-1)$ of Lemma \ref{lem:degree of isolated fixed point}(\ref{subcase:degree}),
\cite{Matsumoto:k3mun}*{Proposition 6.10} (which is stated for $\mu_p$-actions) holds also for $\alpha_p$-actions,
from which the equivalence and $p = 2$ follows.
\end{proof}

\begin{rem}
	Suppose $X$ is an RDP K3 surface.
	If $G = \mu_p$, the author showed \cite{Matsumoto:k3mun}*{Theorems 6.1 and 6.2} that $X/\mu_p$ is an RDP K3 surface
	if and only if the action is \emph{symplectic} (\cite{Matsumoto:k3mun}*{Definition 2.6}) in the sense that the nonzero global $2$-form $\omega$ on $X^{\sm}$,
	which is unique up to scalar, is $D$-invariant (i.e.\ $D(\omega) = 0$).
	Note that since $D^p = D$ we always have $D(\omega) = i \omega$ for some $i \in \bF_p$.
	If $G = \alpha_p$, then this criterion fails since, in fact, any action is symplectic in this sense, since $D^p = 0$.
	This difference is parallel to that of actions of tame and wild finite groups (i.e.\ of order not divisible or divisible by $p$).
\end{rem}

\begin{thm} \label{thm:duality}
	Let $X$ and $Y$ be RDP surfaces
	with $K_X$ numerically trivial and $K_Y$ trivial. 
	If $\map{\pi}{X}{Y}$ is the quotient morphism by either a $\mu_p$-action or an $\alpha_p$-action,
	then so is the induced morphism $\map{\pi'}{Y}{\pthpower{X}}$ 
	(not necessarily by the same group).
\end{thm}
\begin{proof}
	Let $D$ be the derivation on $X$ corresponding to the action.
	By the Rudakov--Shafarevich formula $K_X \sim \pi^* K_Y + (p-1) \divisorialfix{D}$,
	we have $(p-1) \divisorialfix{D} \equiv 0$.
	Since $\divisorialfix{D}$ is effective and numerically trivial, it follows that $\divisorialfix{D} \sim 0$.

	Let $D'$ be a rational $p$-closed derivation on $Y$ inducing $\pi'$, i.e.\ $Y^{D'} = \pthpower{X}$.
	(To find one, take a generator $h$ of $k(Y)/k(\pthpower{X})$ (so $h^p \in k(\pthpower{X})$),
	and define $D'$ by $D' \restrictedto{k(\pthpower{X})} = 0$ and $D'(h) = 1$.
	Then $D'^p = 0$, in particular $D'$ is $p$-closed.)
	By Proposition \ref{prop:n-step},
	we have $K_Y \sim - \divisorialfix{D'} - \pi'^*(\divisorialfix{\pthpower{D}}$.
	Since $K_Y \sim 0$ and $\divisorialfix{\pthpower{D}} \sim 0$, 
	we have $\divisorialfix{D'} = \divisor(g)$ for some rational function $g \in k(Y)^*$.
	Then $D'' := g^{-1} D'$ is a regular derivation on $Y$ with 
	$Y^{D''} = Y^{D'} = \pthpower{X}$ and $\divisorialfix{D''} = 0$.
	By Hochschild's formula $D''$ is also $p$-closed, hence $D''^p = \lambda D''$ 
	for some everywhere regular function $\lambda$ on $Y$, hence $\lambda \in k$,
	and by replacing $D''$ with a scalar multiple we may assume $\lambda = 0$ or $\lambda = 1$,
	and then $D''$ gives either an $\alpha_p$- or $\mu_p$-action respectively.
\end{proof}

\begin{rem}
There exist finite inseparable morphisms of degree $p$ between RDP K3 surfaces
that are not $\mu_p$- nor $\alpha_p$-quotients.
Classification of such morphisms will be given in Section \ref{sec:insep}.

Theorem \ref{thm:duality} fails also if $\pi$ is a $\mu_2$-quotient with $Y$ an Enriques surface (so that $K_Y$ is nontrivial), as in the next proposition, proved by the same way as Theorem \ref{thm:duality}.
\end{rem}

\begin{prop}[cf.\ \cite{Cossec--Dolgachev:enriques}*{Section 1.2}] \label{prop:classical Enriques non-duality}
Let $X$ be an RDP K3 surface in characteristic $p = 2$ and 
$\map{\pi}{X}{Y}$ a $\mu_2$-quotient morphism with $Y$ an RDP classical Enriques surface.
Then $\map{\pi'}{Y}{\thpower{X}{2}}$ is not the quotient morphism by a $p$-closed (regular) derivation. 
Instead $\pi'$ is the quotient morphism by a $p$-closed rational derivation $D'$ on $Y$
with $\divisorialfix{D'} \sim K_Y$.
\end{prop}

Suppose $X$ and $Y$ are RDP K3 surfaces.
We will determine possible characteristics and singularities.

\begin{thm} \label{thm:alpha_p K3 K3} 
Let $\map{\pi}{X}{Y}$ be a $G$-quotient morphism between RDP K3 surfaces in characteristic $p$,
where $G \in \set{\mu_p, \alpha_p}$.
If $G = \mu_p$ then $p \leq 7$.
If $G = \alpha_p$ then $p \leq 5$.

If moreover $\pi$ is maximal,
then $\Sing(Y)$ are as follows.
\begin{itemize}
	\item $\frac{24}{p+1} A_{p-1}$ if $G = \mu_p$.
	\item $2D_4^0$, $1D_8^0$, or $1E_8^0$ if $G = \alpha_2$.
	\item $2E_6^0$ if $G = \alpha_3$.
	\item $2E_8^0$ if $G = \alpha_5$.
\end{itemize}
	By Theorem \ref{thm:duality}, $X$ is a $G'$-quotient of $\thpower{Y}{1/p}$ for $G' \in \set{\mu_p,\alpha_p}$, 
	and hence $\Sing(X)$ is also as described above.
	In particular, the total index of RDPs of $X$ and that of $Y$
	are both equal to $24(p-1)/(p+1)$.
\end{thm}

\begin{rem}
	Suppose $X$ is a smooth K3 surface 
	and $G \subset \Aut(X)$ a cyclic subgroup of prime order $p$.
	Assume $Y = X/G$ is an RDP K3 surface.
	If $\charac(k) \neq p$ then it is well-known that $\Sing(Y)$ is $\frac{24}{p+1} A_{p-1}$, and in particular 
	the total index of RDPs of $Y$ is equal to $24(p-1)/(p+1)$.
	We will see below (Theorem \ref{thm:singularities of quotient K3})
	that this value is equal to $24(p-1)/(p+1)$ even in characteristic $p$.
	Consequently, this value $24(p-1)/(p+1)$ appears for actions of \emph{any} group scheme $G$ of order $p$ in \emph{any} characteristic!
\end{rem}

\begin{proof}[Proof of Theorem \ref{thm:alpha_p K3 K3}]
	We may assume $\pi$ is maximal.
	First we prove the assertion for the total indices of $\Sing(X)$ and $\Sing(Y)$.
	Let $\set{w_i} \subset X$ and $\set{v_j} \subset Y$ be the RDPs, 
	of indices $m_i$ and $n_j$ respectively.
	Let $\tilde{X}$ be the resolution of $X$ and $\tilde{D}$ the induced rational derivation on $\tilde{X}$.
	Using Lemma \ref{lem:degree of isolated fixed point}(\ref{case:RDP}) and Lemma \ref{lem:derivation on resolution} we obtain
	\begin{align*}
	\divisorialfix{\tilde{D}}^2  = - \sum_i \frac{2}{p-1} m_i, \qquad
	\deg \isolatedfix{\tilde{D}} =   \sum_i \frac{p-2}{p-1} m_i + \sum_j \frac{1}{p-1} n_j.
	\end{align*}
	By Theorem \ref{thm:c2 derivation} we have 
	$24 = \deg \isolatedfix{\tilde{D}} - \divisorialfix{\tilde{D}}^2
	 = \sum_i \frac{p}{p-1} m_i + \sum_j \frac{1}{p-1} n_j$.

	We can apply the same argument to $\map{\pi'}{Y}{\pthpower{X}}$ to obtain another equality.
	Also, since $\pi$ is purely inseparable we have $\dim \Het^2(X, \bQ_l) = \dim \Het^2(Y, \bQ_l)$ and hence $\sum_i m_i = \sum_j n_j$.
	By either way, we obtain $\sum_i m_i = \sum_j n_j = 24(p-1)/(p+1)$.

	Each $v_j$ is one of those appearing in Table \ref{table:derivation quotient RDP}. 
	If $G = \alpha_p$ then we have $p \leq 5$ and then $\Sing(Y)$ is as stated.
	If $G = \mu_p$ then $\Sing(Y)$ is as stated, and hence $(p+1) \divides 24$ and $24(p-1)/(p+1) < 22$.
	This implies $p \leq 11$.
	We refer to \cite{Matsumoto:k3mun}*{Theorem 7.1} for a proof of $p \neq 11$.
\end{proof}

\section{Inseparable morphisms of degree $p$ between RDP K3 surfaces} \label{sec:insep}
Suppose $\map{\pi}{X}{Y}$ is a finite inseparable morphism of degree $p$ between RDP K3 surfaces.
It is not always a quotient morphism by a global regular derivation.
However it can be covered by such a quotient morphism, and we have a classification as in Theorem \ref{thm:insep K3}.

\begin{lem} \label{lem:conjugate by SL2}
	Let $r > 1$ be an integer prime to $p = \charac k$.
	Suppose either 
	$M = \begin{pmatrix}
	\lambda & 0 \\ 0 & -\lambda
	\end{pmatrix}$ ($\lambda \in k^*$),
	or 
	$r$ is even and 
	$M = \begin{pmatrix}
	0 & 1 \\ 0 & 0
	\end{pmatrix}$.
	Then there is no $g \in \SL_2(k)$ of order $r$ such that $g^{-1} M g = \zeta M$ with a primitive $r$-th root $\zeta$ of $1$.
\end{lem}
\begin{proof}
If $2 \divides r$, then $g^{r/2} \in \SL_2(k)$ is of order $2$, hence $g^{r/2} = -I_2$, which is central.
If $r > 2$ in the former case, then $M$ and $\zeta M$ have different eigenvalues.
\end{proof}

\begin{thm} \label{thm:insep K3}
	Suppose $\map{\pi}{X}{Y}$ is a finite inseparable morphism of degree $p$ between RDP K3 surfaces.
	Then for some $r \geq 1$ and some $G \in \set{\mu_p,\alpha_p}$,
	there exists a $\bZ/r\bZ$-equivariant $G$-quotient morphism $\map{\bar{\pi}}{\bar{X}}{\bar{Y}}$ between proper RDP surfaces equipped with $\bZ/r\bZ$-actions,
	fitting into a commutative diagram
	\[
	\begin{tikzcd}
	\bar{X} \arrow[d,"\phi_X"] \arrow[r,"\bar{\pi}"] &
	\bar{Y} \arrow[d,"\phi_Y"] \\
	      {X}                 \arrow[r,"      {\pi}"] &
		  {Y} 
	\end{tikzcd}
	\]
	such that $\map{\phi_X}{\bar{X}}{X}$ and $\map{\phi_Y}{\bar{Y}}{Y}$ are the $\bZ/r\bZ$-quotient morphisms.

	Among such ``coverings'' $\bar{\pi}$, there exists a minimal one
	(i.e.\ any other such covering admits $\bar{\pi}$ as a subcovering).
	If $\bar{\pi}$ is minimal, then 
	$r \in \set{1,2,3,4,6}$ and $r \divides p-1$, 
	the $\bZ/r\bZ$-actions on $\bar{X}$ and $\bar{Y}$ are symplectic (in the usual sense on abelian and K3 surfaces),
	and moreover exactly one the following holds:
	\begin{enumerate}
		\item \label{case:abelian} $\bar{X}$ and $\bar{Y}$ are (smooth) abelian surfaces, and $r \neq 1$;
		\item \label{case:K3 mu} $\bar{X}$ and $\bar{Y}$ are RDP K3 surfaces, $G = \mu_p$, $p \leq 7$, and $(p,r) \neq (7,2),(7,6)$; or
		\item \label{case:K3 alpha} $\bar{X}$ and $\bar{Y}$ are RDP K3 surfaces, $G = \alpha_p$, $p \leq 5$, and $(p,r) \neq (5,4)$.
	\end{enumerate}
	
	Every case and every remaining $(p,r)$ occurs.

	If $\bar{\pi}$ is minimal and moreover maximal (in the sense of Definition \ref{def:maximal}), then $\Sing(Y)$ is as described in Table \ref{table:insep K3}.
\end{thm}

\begin{table}
\caption{Structure of purely inseparable morphisms of degree $p$ between RDP K3 surfaces} \label{table:insep K3}
\begin{tabular}{llll}
\toprule                                    
covering       & $p$                 & $r$ & $\Sing(Y)$ \\
\midrule                                    
abelian        & $\equiv 1 \pmod{6}$ & $6$ & $A_5 + 4 A_2 + 5 A_1$ \\
abelian        & $\equiv 1 \pmod{4}$ & $4$ & $4 A_3 + 6 A_1$ \\
abelian        & $\equiv 1 \pmod{3}$ & $3$ & $9 A_2$  \\
abelian        & $\equiv 1 \pmod{2}$ & $2$ & $16 A_1$ \\
\midrule                                    
K3, $\mu_7$    & $7$                 & $3$ & $  A_6 + 6 A_2$ \\
K3, $\mu_7$    & $7$                 & $1$ & $3 A_6$ \\
K3, $\mu_5$    & $5$                 & $4$ & $  A_4 + 4 A_3 + 2 A_1$ \\
K3, $\mu_5$    & $5$                 & $2$ & $2 A_4 + 8 A_1$ \\
K3, $\mu_5$    & $5$                 & $1$ & $4 A_4$ \\
K3, $\mu_3$    & $3$                 & $2$ & $3 A_2 + 8 A_1$ \\
K3, $\mu_3$    & $3$                 & $1$ & $6 A_2$ \\
K3, $\mu_2$    & $2$                 & $1$ & $8 A_1$ \\
\midrule                                    
K3, $\alpha_5$ & $5$                 & $2$ & $  E_8^0 + 8 A_1$ \\
K3, $\alpha_5$ & $5$                 & $1$ & $2 E_8^0$ \\
K3, $\alpha_3$ & $3$                 & $2$ & $  E_6^0 + 8 A_1$ \\
K3, $\alpha_3$ & $3$                 & $1$ & $2 E_6^0$ \\
K3, $\alpha_2$ & $2$                 & $1$ & $2 D_4^0$, $1 D_8^0$, or $1 E_8^0$ \\
\bottomrule
\end{tabular}
\end{table}

\begin{proof}
	As in the proof of Theorem \ref{thm:duality},
	take a rational derivation $D$ with $Y = X^{D}$. 
	Then we have $(p-1) \divisorialfix{D} = 0$ in $\Pic(X^{\sm})$.
	Let $\map{\phi}{\overline{X^{\sm}}}{X^{\sm}}$ be the \'etale covering trivializing $\divisorialfix{D}$ (so $r = \deg \phi$ divides $p-1$).
	Then the normalization $\bar{X}$ of $X$ in $k(\overline{X^{\sm}})$ is an RDP surface.

	We claim that $\bar{X}$ is an RDP K3 surface or an abelian surface.
	This is trivial if $r = 1$. Assume $r \geq 2$, hence $p \geq 3$.
	By construction $\bar{X}$ has trivial canonical divisor.
	If $\bar{X}$ is not RDP K3 nor abelian, then it is (quasi-)hyperelliptic surface in characteristic $p = 3$.
	Hence $r = 2$.
	Comparing the $l$-adic Euler--Poincar\'e characteristic (which is $0$ and $24$ for (quasi-)hyperelliptic and K3 surfaces), 
	we observe that the involution $g$ on the resolution $\tilde{\bar{X}}$ has $16$ fixed points,
	but then we have 
	\[
	22 - 16    = \dim \Het^2(\tilde{\bar{X}}/\spanned{g}, \bQ_l) 
	           = \dim \Het^2(\tilde{\bar{X}}, \bQ_l)^{\spanned{g}} 
	        \leq \dim \Het^2(\tilde{\bar{X}}, \bQ_l)
	           = 2, 
	\]
	a contradiction.

	We have $\phi^{-1}(\divisorialfix{D}) = \divisor(h)$ for some $h \in k(\bar{X})$, 
	and then $\bar{D} := h^{-1} \cdot \phi^*(D)$ is a regular derivation.
	Write $\bar{X}^{\bar{D}} = \bar{Y}$.
	Take a generator $g_X$ of the $\bZ/r\bZ$-action on $\bar{X}$.
	Then $g_X$ acts on $\bar{D}$ by multiplication by a $r$-th root $\lambda$ of unity. 
	This $\lambda$ is in fact a primitive $r$-th root of unity since,
	if $\lambda^{s} = 1$, then $\bar{D}$ descends to $\bar{X} / \spanned{g_X^{s}}$,
	hence $\divisorialfix{D}$ is trivialized on $\bar{X} / \spanned{g_X^{s}}$, 
	hence $g_X^s = 1$, hence $r \divides s$.
	Hence $g_X$ induces an automorphism $g_Y$ on $\bar{Y}$ of order $r$ with $\bar{Y} / \spanned{g_Y} = Y$.
	
	We show the minimality.
	Let $\map{\psi}{\bar{X}'}{X}$ with $\bar{D}'$ be another covering of $\pi$ with the required properties.
	Then the pullback $\psi^*(D)$ of $D$ to $\bar{X}'$ coincide with $\bar{D}'$ up to $k(X)^*$, in particular $\divisorialfix{\psi^*(D)} \sim 0$ on $\Pic(\psi^{-1}(X^{\sm}))$.
	Hence $\psi \restrictedto{\psi^{-1}(X^{\sm})}$ factors through $\phi \restrictedto{\phi^{-1}(X^{\sm})}$,
	and $\psi$ factors through $\phi$.

	We show $r \in \set{2,3,4,6}$ and the description of the singularities in the case $\bar{X}$ is an abelian surface.
	It is proved by Katsura \cite{Katsura:generalizedkummer}*{Theorem 3.7 and Table in page 17} that,
	if $\bar{X}$ is an abelian surface and $g$ is a nontrivial symplectic automorphism (fixing the origin) of order $r$ prime to $p = \charac k$, then
	$r \in \set{2,3,4,5,6,8,10,12}$,
	$\bar{X}/\spanned{g}$ is an RDP K3 surface,
	and $\Sing(\bar{X}/\spanned{g})$ are as in Table \ref{table:singularities of abelian-quotient K3}
	(in \cite{Katsura:generalizedkummer} the coefficient of $A_7$ in order $8$ is written as $1$, but this is a misprint and actually it is $2$). 
	In particular, if $r \in \set{5,8,10,12}$ then 
	(since the exceptional curves of the resolution of $\bar{X}/\spanned{g}$ generate a rank $20$ negative-definite lattice)
	$\bar{X}/\spanned{g}$ is a supersingular RDP K3 surface and 
	$\bar{X}$ is a supersingular abelian surface.
	It is showed \cite{Katsura:generalizedkummer}*{Lemma 6.3}
	that supersingular abelian surfaces in characteristic $p$ do not have symplectic automorphisms of order $r = 5$ if $p \equiv 1 \pmod{5}$.
	One observes that the proof of this lemma 
	relies only on the fact that $[\bQ(\zeta_5):\bQ] = 4$, 
	therefore it remains valid if we replace $5$ with $8$, $10$, or $12$.
	Hence we obtain $r \in \set{2,3,4,6}$ in our case.

	Suppose $\bar{X}$ is an RDP K3 surface and $\bar{\pi}$ is a $\mu_p$-quotient or an $\alpha_p$-quotient.
	Then respectively $p \leq 7$ or $p \leq 5$ by Theorem \ref{thm:alpha_p K3 K3}.
	We show that if $r > 1$ then $g_X$ does not fix any point of $\Fix(\bar{D})$.
	If $G = \mu_p$, then the action of $\bar{D}$ on the tangent space of a point of $\Fix(\bar{D})$
	is diagonalizable with eigenvalues $\pm i$ ($i \in \bF_p^*$).
	If $G = \alpha_p$, then $p \in \set{3,5}$ and hence $r \in \set{2,4}$, 
	and the action of $\bar{D}$ on the tangent space is nilpotent and nontrivial (otherwise $s \geq 3$ in Lemma \ref{lem:degree of isolated fixed point}).
	Hence in either case it is impossible by Lemma \ref{lem:conjugate by SL2}.

	Using this, we show that $(G,r) = (\mu_7, 2), (\mu_7, 6), (\alpha_5, 4)$ cannot happen. 
	If $(G, r) = (\mu_7, 2)$, then $\Fix(\bar{D})$ consists of $3$ points  
	$w_1, w_2, w_4 \in \bar{X}$, on whose tangent spaces $\bar{D}$ acts by eigenvalues $\pm 1, \pm 2, \pm 4$ respectively.
	Since $g_X^* \bar{D} = -\bar{D}$, we have $g_X(w_1) \neq w_2, w_4$, 
	and we have $g_X(w_1) \neq w_1$ by above. Contradiction.
	The case $(G, r) = (\mu_7, 6)$ is reduced to the previous case.
	If $(G, r) = (\alpha_5, 4)$, then $\Fix(\bar{D})$ consists of $2$ points,
	hence $g_X^2$ fixes each point. Contradiction.

	The assertion on $\Sing(Y)$ follows from the description of $\Sing(\bar{Y})$ (Theorem \ref{thm:alpha_p K3 K3}),
	the description of the fixed locus and the quotient singularities of a symplectic automorphism of finite order prime to the characteristic
	(Nikulin \cite{Nikulin:auto}*{Section 5} ($p = 0$) and Dolgachev--Keum \cite{Dolgachev--Keum:auto}*{Theorem 3.3} ($p > 0$)),
	and the observation above that $\spanned{g_X}$ acts freely on $\Fix(\bar{D})$. 

	We will see in Examples \ref{ex:k3:2}--\ref{ex:k3:7m:non-normal} ($r = 1$),
	\ref{ex:insep K3 with abelian cover} ($r > 1$, $\bar{X}$ abelian), \ref{ex:insep K3 with K3 cover} ($r > 1$, $\bar{X}$ K3)
	that all cases indeed occur.
\end{proof}

\begin{table}
\caption{RDP K3 surfaces arising as symplectic cyclic quotients of abelian surfaces \cite{Katsura:generalizedkummer}*{Table in page 17}} \label{table:singularities of abelian-quotient K3} 
\begin{tabular}{ll}
\toprule
$r$  & $\Sing(X)$ \\
\midrule
$2$  & $16 A_1$ \\
$3$  & $9 A_2$ \\
$4$  & $4 A_3 + 6 A_1$ \\
$5$  & $5 A_4$ \\
$6$  & $A_5 + 4 A_2 + 5 A_1$ \\
$8$  & $2 A_7 + A_3 + 3 A_1$ \\
$10$ & $A_9 + 2 A_4 + 3 A_1$ \\ 
$12$ & $A_{11} + A_3 + 2 A_2 + 2 A_1$ \\
\bottomrule
\end{tabular}
\end{table}

\section{\texorpdfstring{$\bZ/p\bZ$-, $\mu_p$-, $\alpha_p$-coverings of RDPs}{Z/pZ-, mu\_p-, alpha\_p-coverings of RDPs}} \label{sec:covering}

In this section we describe $\bZ/p\bZ$-, $\mu_p$-, and $\alpha_p$-coverings of certain RDPs 
that are related to $\bZ/p\bZ$-, $\mu_p$-, and $\alpha_p$-coverings of RDP K3 surfaces discussed in 
Section \ref{sec:restoring}.

\subsection{\texorpdfstring{$\mu_p$-coverings}{mu\_p-coverings}} \label{subsec:covering:mu}

Let $Z = \Spec A$ be a local ring that is an RDP of type $A_{n-1}$, in characteristic $p \geq 0$ (possibly dividing $n$).
	Let $\tilde{Z} \to Z$ be the minimal resolution and 
	let $e_j$ ($1 \leq j \leq n-1$) be the exceptional curves numbered 
	as in Convention \ref{conv:exceptional curves}
	(i.e.\ $e_j \cdot e_{j'} = 1$ if and only if $\abs{j - j'} = 1$).

\begin{lem} \label{lem:Picard group of A}
	\ 
	\begin{enumerate}
	\item \label{lem:Picard:cyclic}
	There is a canonical injection from $\Pic(Z^{\sm})$ to a cyclic group of order $n$.
	It is compatible with \'etale extensions of $A$ 
	and it is an isomorphism if $A$ is Henselian.
	
	In the following assertions, we assume that the injection in (\ref{lem:Picard:cyclic}) is an isomorphism.
	\item \label{lem:Picard:algebra*}
	For each $0 < h < n$, let $L_h$ be a line bundle on $Z^{\sm}$ belonging to the class $h \in \bZ/n\bZ \cong \Pic(Z^{\sm})$.
	Let $L_0 = \cO_{Z^{\sm}}$.
	Let $\phi_0 = \id_{L_0}$ and $\phi_1 = \id_{L_1}$.
	Take isomorphisms 
	$\map[\isomto]{\phi_h}{L_h}{L_1^{\otimes h}}$ ($2 \leq h < n$)
	and $\map[\isomto]{\psi}{L_0}{L_1^{\otimes n}}$.
	Then the morphisms 
	\begin{align*}
		{\phi_{h+{h'}}^{-1} \circ (\phi_h \otimes \phi_{h'})} &\colon
		{L_h \otimes L_{h'}} \to {L_{h+h'}} \quad (h + h' < n), \\
		{(\phi_{h+h'-n}^{-1} \otimes \psi^{-1}) \circ (\phi_h \otimes \phi_{h'})} &\colon
		{L_h \otimes L_{h'}} \to {L_{h+h'-n}} \quad (h + h' \geq n) 
	\end{align*}
	define an $\cO_{Z^{\sm}}$-algebra structure on $V := \bigoplus_{h = 0}^{n-1} L_h$.
	\item \label{lem:Picard:algebra}
	Let $\bar{L}_h := \iota_* L_h$ and $\bar{V} = \iota_* V = \bigoplus_{h = 0}^{n-1} \bar{L}_h$,
	where $\map{\iota}{Z^{\sm}}{Z}$ is the inclusion.
	Then the $\cO_{Z^{\sm}}$-algebra structure on $V$
	extends to an $\cO_{Z}$-algebra structure on $\bar{V}$,
	and $\bar{V}$ is regular.
	$U := \Spec \bar{V} \to Z$ is a $\mu_n$-covering.
	\item \label{lem:Picard:coefficients}
	Let $\tilde{L}_h = \tilde{\iota}_* L_h$,
	where $\map{\tilde{\iota}}{Z^{\sm}}{\tilde{Z}}$
	is the inclusion.
	Then $I_h := \Image( (\tilde{L}_h)^{\otimes n} \to \cO_{\tilde{Z}})$ is an invertible sheaf
	and, writing $I_h = \cO_{\tilde{Z}}(- \sum b_{h,j} e_j)$,
	there exists $a \in (\bZ/n\bZ)^*$ such that $b_{h,j} \equiv a h j \pmod{n}$.
	More precisely, we have
	\[
	I_h = \cO(-\sum_j f_n((a h \bmod n), j) e_j).
	\]
	\end{enumerate}
	Here $m \bmod n$ denotes the remainder modulo $n$, i.e., the unique integer $\in \set{0, \dots, n-1 }$ congruent to $m$ modulo $n$,
	and the function $\map{f_n}{\set{1, 2, \dots, n-1}^2}{\bZ}$
	is defined by 
	\[
		f_n(h, j) = \begin{cases}
			h j & (j \leq n - h) \\
			(n - h) (n - j) & (j \geq n - h) .
		\end{cases}
	\]
\end{lem}

\begin{proof}
	(\ref{lem:Picard:cyclic})
	This is \cite{Lipman:rationalsingularities}*{Proposition 17.1}.
	
	(\ref{lem:Picard:algebra*})
	Straightforward.
	
	(\ref{lem:Picard:algebra})
	We may assume that $A$ is complete.
	By changing the isomorphism $\bZ/n\bZ \cong \Pic(Z^{\sm})$,
	we may assume that $A = k[[x^n,y^n,xy]] \subset B = k[[x,y]]$
	and identify $\bar{L}_h$ with $x^h A + y^{n-h} A \subset B$ for $0 < h < n$,
	and $\phi_h^{-1}$ with the multiplication in $B$.
	We have $\psi^{-1}(x^{\otimes n}) = a x^n$ with $a \in A^*$.
	Replacing $B = k[[x,y]]$ with $k[[x',y']]$ ($x' = a^{1/n} x$, $y' = a^{-1/n} y$),
	and identifying $x^h A + y^{n-h} A \isomto x'^h A + y'^{n-h} A$
	by the multiplication by $(a^{1/n})^h$,
	we may assume $a = 1$.
	Then $\bar{V} = B$ and is regular.

	(\ref{lem:Picard:coefficients})
	Straightforward (cf.\ \cite{Matsumoto:k3mun}*{Lemma 4.16}).
\end{proof}

\begin{rem}
Suppose $A$ is Henselian.
If $p \notdivides n$, 
then $U \to Z$ is independent of the choices (since, under the notation in the proof, $a^{1/n}$ exists in $A^*$)
and $U \restrictedto{Z^{\sm}} \to Z^{\sm}$ is the fundamental covering.
To the contrary,
if $p \divides n$,
then $U \to Z$ does depend on the choice of the isomorphisms $\phi_h$ and $\psi$, and is not unique.
\end{rem}

\subsection{\texorpdfstring{$\bZ/p\bZ$- and $\alpha_p$-coverings}{Z/pZ- and alpha\_p-coverings}} \label{subsec:covering:alpha}

\begin{lem} \label{lem:Ext1}
	Let $A$ be a Noetherian Gorenstein $2$-dimensional local $k$-algebra,
	and $I \subset A$ an ideal with $\Supp(A/I) \subset \set{\idealm_A}$
	(equivalently $I \supset \idealm_A^n$ for some $n$).
	Then $\dim_k \Ext^1(I, A) = \dim_k A/I$.
	For any other such ideal $I'$ with $I' \subset I$,
	the induced map $\Ext^1(I, A) \to \Ext^1(I', A)$ is injective.
	The map $\Ext^1(I, A)  \isomto \Ext^2(A/I, A) \to \localcoh{\idealm_A}{A}$ is injective
	and its image is the $I$-torsion part $\localcohtorsion{\idealm_A}{A}{I}$.

		If $x,y$ is a regular sequence in $\idealm_A$,
		then we have an isomorphism 
		$\localcoh{\idealm_A}{A} \isomto \Coker \Bigl( A[x^{-1}] \oplus A[y^{-1}] \to A[(xy)^{-1}] \Bigr)$.
\end{lem}
\begin{proof}
	See \cite{Matsumoto:k3rdpht}*{Lemma 3.1}.
	The assertion on the dimension follows from $\dim \Ext^1(\idealm, A) = \dim \Ext^2(k, A) = 1$, which follows from Gorenstein.
\end{proof}

\begin{lem} \label{lem:Ext1-Frobenius} 
Let $A$ and $I$ be as in the previous Lemma.
Then there are canonical semilinear maps 
$\map{F}{\Ext^1(I, A)}{\Ext^1(\pthpower{I}, A)}$ and
$\map{F}{\Ext^2(A/I, A)}{\Ext^2(A/\pthpower{I}, A)}$,
	which we call the Frobenius,
	satisfying the following properties.
	\begin{itemize}
		\item 
		$F$ commute with the boundary maps
		and the pullbacks by inclusions $I' \injto I$ of ideals.
		\item 
		$h_{\pthpower{I}}(F(e)) = (h_I(e))^p$,
		where $h_I$ is the map 
		$\Ext^1(I, A)  \isomto \Ext^2(A/I, A) \to \localcoh{\idealm_A}{A} \isomto \Coker \Bigl( A[x^{-1}] \oplus A[y^{-1}] \to A[(xy)^{-1}] \Bigr)$ 
		defined in Lemma \ref{lem:Ext1}. 
	\end{itemize}
\end{lem}
\begin{proof}
	We define the maps on the local cohomology groups $\localcoh{\idealm_A}{A}$,
	and use the identification of Lemma \ref{lem:Ext1}.
\end{proof}

Now let $Z = \Spec A$ be a local RDP in characteristic $p$
and suppose $(p, \Sing(Z))$ is one of the following, and define an integer $m \geq 1$ accordingly.
\begin{itemize}
\item $(p, \Sing(Z)) = (2, D_{4m}^r)$, $m \geq 1$, $r \in \set{0, \dots, m}$.
\item $(p, \Sing(Z)) = (2, E_8^r)$, $r \in \set{0,1,2}$, let $m = 2$.
\item $(p, \Sing(Z)) = (3, E_6^r)$, $r \in \set{0,1}$, let $m = 1$.
\item $(p, \Sing(Z)) = (5, E_8^r)$, $r \in \set{0,1}$, let $m = 1$.
\end{itemize}
Thus, in each case, the range of $r$ is $\set{0, \dots, m}$. 

(The RDPs of type $D_{4m}^{r}$ ($r \in \set{m+1, \dots, 2m-1}$) and $E_{8}^r$ ($r \in \set{3,4}$) in characteristic $2$
will not be discussed in this paper.)

We assume $A$ is complete, 
and we fix the presentation $A = k[[x,y,z]] / (F)$ as follows, for each case of $(p, \Sing(Z))$.
		\begin{alignat*}{2}
			(2, D_{4m}^r) &\colon F =  z^2 + x^2 y + x y^{2m}+ \lambda z x y^m,      &\qquad \lambda &= 0, y^{m-r} \quad (r = 0, r > 0), \\
			(2, E_8^r)    &\colon F =  z^2 + x^3   + y^5     + \lambda z x y^2,      &\qquad \lambda &= 0, y, 1 \quad (r = 0,1,2), \\
			(3, E_6^r)    &\colon F = -z^2 + x^3   + y^4     + \lambda x^2 y^2,      &\qquad \lambda &= 0, 1 \quad (r = 0,1), \\
			(5, E_8^r)    &\colon F =  z^2 + x^3   + y^5     + (\lambda/2) x y^4,    &\qquad \lambda &= 0, 2 \quad (r = 0,1).
		\end{alignat*}

		Write $x_1 = x$ and $x_2 = y$.
		Let $Z_i = \Spec A[x_i^{-1}]$.
		Define $\bar{q}_i \in A[x_i^{-1}]$ as below and let $\bar{\varepsilon} := z/(x y^{m})$.
		Then we have $\bar{\varepsilon}^p - \lambda \bar{\varepsilon} = \bar{q}_1 - \bar{q}_2$.
\[
\bar{q}_1 := \begin{cases}
x^{-1}, \\
x^{-2} y, \\
x^{-3} y z, \\
x^{-5} (y^5 + \lambda x y^4 + (\lambda^2/4) x^2 y^3 + 2 x^3)z,
\end{cases}
\quad
\bar{q}_2 := \begin{cases}
y^{-(2m-1)}, \\
y^{-4} x, \\
-y^{-3} z , \\
-y^{-5} x z. 
\end{cases}
\]
		Note that $\bar{\varepsilon}$ itself cannot be written as $\bar{\varepsilon} = q'_1 - q'_2$ with $q'_i \in A[x_i^{-1}]$.

Define an ideal $I \subset A$ to be $I = (x, y^m, z)$ according to the convention on $m$ and the presentation given above.
In fact, this ideal can be defined intrinsically (without assuming completeness) as follows:
\begin{itemize}
	\item If $(p, \Sing(Z))$ is $(2, D_4^r)$, $(3, E_6^r)$, or $(5, E_8^r)$,
	then $I$ is the maximal ideal $\fm$. 
	\item If $(p, \Sing(Z))$ is $(2, D_{4m}^r)$ (resp.\ $(2, E_8^r)$), 
	then $I$ consists of the elements that vanish on the component $e_{2m}$ (resp.\ $e_{4}$) 
	with order $\geq 2m$ (resp.\ $\geq 8$),
	where the components are numbered as in Convention \ref{conv:exceptional curves}.
\end{itemize}

\begin{lem} \label{lem:RDP:bis:dim H1L1}
			$\Ext^1_A(I, A)$ is $m$-dimensional as a $k$-vector space,
and generated by the class $\bar{e}$ of $\bar{\varepsilon}$ as an $A$-module
(under the identification of Lemma \ref{lem:Ext1}).
\end{lem}

\begin{proof}
		By Lemma \ref{lem:Ext1},
		we have $\dim \Ext^1_A(I, A) = \dim_k (A/I) = m$.
		We also have $I \subset \Ann(\bar{e})$.
		It remains to show that $\Ann(\bar{e}) \subset I$.
		It suffices to show $y^{m-1} \notin \Ann(\bar{e})$.
		Using the isomorphism $A = k[[x,y]] \oplus k[[x,y]] z$ of $k[[x,y]]$-modules, we see that 
		the class of $y^{m-1} \bar{\varepsilon} = z y^{m-1} / (x y^m) = z / (xy)$ in
		$\Coker \Bigl( A[x^{-1}] \oplus A[y^{-1}] \to A[(xy)^{-1}] \Bigr)$ is nontrivial.
\end{proof}

	\begin{lem} \label{lem:cohomology of RDP:bis}
		Let $q_i \in A[x_i^{-1}]$ and $\varepsilon \in A[(xy)^{-1}]$.
		Suppose $\varepsilon^p - \lambda \varepsilon = q_1 - q_2$, and 
		the class $[\varepsilon]$
		is a generator of $\Ext^1_A(I, A)$.
		Let $U_i \to Z_i$ be the coverings given by 
		$\cO_{U_i} = \cO_{Z_i}[t_i] / (t_i^p - \lambda t_i - q_i)$,
		glue them on $Z_1 \cap Z_2$ by 
		$t_1 - t_2 = \varepsilon$, 
		and let $U = \Spec B \to Z = \Spec A$ be the normalization of $U_1 \cup U_2 \to Z_1 \cup Z_2 = Z^{\sm} \subset Z$.
		Then the following assertions hold.
		\begin{enumerate}
		\item \label{item:RDP:bis:unit}
			Let $e$ be the class of $\varepsilon$ in $\Ext^1_A(I, A)$ 
			Then $e = h \cdot \bar{e}$ 
			for some $h \in (k[y]/y^m)^*$.
			We have ($\iota^*(e) \neq 0$ and) $F(e) = \lambda \cdot \iota^*(e)$,
			where 
			$\map{F}{\Ext^1(I, A)}{\Ext^1(\pthpower{I}, A)}$ is the Frobenius (Lemma \ref{lem:Ext1-Frobenius}),
			and $\map{\iota^*}{\Ext^1(I, A)}{\Ext^1(\pthpower{I}, A)}$ is the morphism induced from the inclusion $\map{\iota}{\pthpower{I}}{I}$.
			If $r = m$ then $h \in \mu_{p-1} \subset k^*$.
		\item \label{item:RDP:bis:U smooth}
			$U_1 \cup U_2$ is regular, and $U$ is regular.
		\item \label{item:RDP:bis:delta}
			There is a unique endomorphism $\delta \in \End(B)$ (of the $A$-module $B$) satisfying
			$\delta \restrictedto{A} = 0$,
			$\delta(t_i) = 1$, 
			$\delta(b c) = \delta(b) c + b \delta(c) + \lambda^{1/(p-1)} \delta(b) \delta(c)$,
			and $\delta^p = 0$.
			Here we fix a $(p-1)$-th root $\lambda^{1/(p-1)}$ of $\lambda$.

			If $r = m$ (resp.\ $0 < r < m$), then
			$g := \id + \lambda^{1/(p-1)} \delta$ is an automorphism of order $p$
			generating $\Aut_Z(U)$, and
			$\pi$ is a $\bZ/p\bZ$-covering with $\Supp \Fix(g)$ consisting precisely of the closed point (resp.\ $\dim \Supp \Fix(g) = 1$).
			If $r = m$, this means that $U \times_{Z} Z^{\sm} \to Z^{\sm}$ is the fundamental covering.

			If $r = 0$,
			then $\delta$ is a derivation of additive type, and
			$\pi$ is an $\alpha_p$-covering with $\Supp \Fix(\delta)$ consisting precisely of the closed point.
		\item \label{item:RDP:bis:V}
			We have $\Image (\delta^{j} \restrictedto{\Ker \delta^{j+1}}) = I$ for all $1 \leq j \leq p-1$.
		\item \label{item:RDP:bis:e}
			Let $V = \Ker \delta^2 \subset B$.
			The extension
			\[ 0 \to A \to V \namedto{\delta} I \to 0 \]
			is non-split.
			The corresponding class in $\Ext^1(I, A)$ is $e$.
		\end{enumerate}
	\end{lem}
		These descriptions of the coverings for the cases $r > 0$ 
		are essentially the ones given in \cite{Artin:RDP}*{Sections 4--5}.
		(We note that the equations for $p = 3,5$ given there should be fixed as $- \alpha^3 - \alpha$ for $p = 3$ and $\alpha^5 - 2 \alpha$ for $p = 5$.)

\begin{proof}
		(\ref{item:RDP:bis:unit})
		By Lemma \ref{lem:RDP:bis:dim H1L1}, the first assertion is clear.
		The assumption on $\varepsilon$ yields the equality $F(e) - \lambda \cdot \iota^*(e) = 0$.
		Suppose $r = m$.
		Since $\bar{e}$ satisfies the same equality and since $\lambda \in k^*$,
		we have $h^p = h$ in $k[y]/(y^m)$, hence $h \in \mu_{p-1}$.
		
		(\ref{item:RDP:bis:U smooth})
		For the cases $r = m$, this is proved by Artin \cite{Artin:RDP}*{Sections 4--5}.

		Suppose $(p, \Sing(Z)) = (2, D_{4m}^r)$ (resp.\ $(2, E_8^r)$).
		Let $u = x t_1$ and $v = y^m t_2$.
		Then we have $u^2 + \lambda x u - x^2 f= x$ (resp.\ $= y$),
		$v^2 + \lambda y^m v - y^{2m} f = y$ (resp.\ $= x$),
		and $y^m u - x v = z$.
		Let $B' = A[u, v]$.
		Then by above $B'$ is integral over $A$, satisfies $A \subsetneq B' \subset \Frac B$, and
		its maximal ideal is generated by $u$ and $v$,
		hence $B'$ is regular, in particular normal, hence $B' = B$.

		Suppose $r = 0$.
		Let $Z' := Z^{\sm} = Z_1 \cup Z_2$ and $U' := U \times_Z Z'$.
		Let $\omega$ be the $2$-form on $Z'$ satisfying 
		$F_{x_i} \omega = dx_{i+1} \wedge dx_{i+2}$,
		where we write $(x_1,x_2,x_3) = (x,y,z)$
		and consider the indices modulo $3$.
		Applying Proposition \ref{prop:covering} to $U' \to Z'$,
		we obtain a $1$-form $\eta$ satisfying $\eta = dq_i = d\bar{q}_i + df$ on $Z_i$
		and a derivation $D'$ satisfying $D'(g) \omega = dg \wedge \eta$,
		$\pi(\Sing(U')) = \Zero(\eta) = \Fix(D')$, and 
		$Z'^{D'} = \pthpower{(\normalization{(U')})}$.
		Since $Z$ is normal, $D'$ extends to a derivation $D$ on $Z$,
		and since $U$ is normal we have $Z^D = \pthpower{U}$.
		It remains to show $\Fix(D) = \emptyset$,
		since then by Theorem \ref{thm:mu_p alpha_p RDP}(\ref{thm:mu_p alpha_p RDP:non-fixed})
		it follows that $\pthpower{U}$ and hence $U$ are regular.
		Let $c = 1,1,3$ and $i = 3,1,2$ for $p = 2,3,5$ respectively (hence $F_{x_i} = 0$).
		A straightforward calculation yields 
		$\eta = c F_{x_{i+1}}^{-1} dx_{i+2} + df$ ($= - c F_{x_{i+2}}^{-1} dx_{i+1} + df$).
		Hence we have 
		$D(x_i) = -c + F_{x_{i+2}} f_{x_{i+1}} - F_{x_{i+1}} f_{x_{i+2}} \in \cO_Z^*$
		(where $df = \sum_h f_{x_h} d x_h$),
		hence $\Fix(D) = \emptyset$.

		(\ref{item:RDP:bis:delta})
		On each $U_i$ there exists a unique $\delta \in \End(\cO_{U_i})$ with the required properties.
		They glue to an endomorphism $\delta$ on $\cO_{U_1 \cup U_2}$.
		Since $U$ is normal and $U_1 \cup U_2$ is the complement in $U$ of a codimension $2$ subscheme, this $\delta$ extends to $U$.

			If $r = m$ (resp.\ $0 < r < m$), then
			$g := \id + \lambda^{1/(p-1)} \delta \in \End(B)$ 
			preserves products and satisfies $g^p = \id$, hence is an automorphism.
			It is nontrivial since $\lambda \neq 0$ and $\delta \neq 0$.
			Since the ideal of $\cO_{U_1 \cup U_2}$ generated by $\Image(g - \id)$ is $(y^{m-r})$,
			we have $\Supp \Fix(g) \restrictedto{U_1 \cup U_2} = \emptyset$
			(resp.\ $\Supp \Fix(g) \restrictedto{U_1 \cup U_2} = (y = 0)$).
			Since the image of the closed point of $U$ is singular, the closed point belongs to $\Supp \Fix(g)$.

			If $r = 0$,
			then $\delta$ is a derivation (since $\lambda = 0$) and is of additive type, 
			we have $\Supp \Fix(\delta) \restrictedto{U_1 \cup U_2} = \emptyset$,
			and similarly the closed point belongs to $\Supp \Fix(\delta)$.

		(\ref{item:RDP:bis:V})
		If $(p, \Sing(Z)) = (3, E_6^0), (5, E_8^0)$,
		this is proved in Lemma \ref{lem:degree of isolated fixed point}(\ref{subcase:alpha35}).

		Let $I_j := \Image (\delta^{j} \restrictedto{\Ker \delta^{j+1}})$ 
		for each $1 \leq j \leq p-1$.
		We have $I_{p-1} \subset I_j \subset I_1$.
		By assumption we have $\varepsilon \notin A[x^{-1}] + A[y^{-1}]$,
		hence $I_1 \subsetneq A$, hence $I_1 \subset \fm$.

		Suppose $p = 2$.
		Let $u,v$ be as in the proof of (\ref{item:RDP:bis:U smooth}).
		We have $\delta(1) = 0$, $\delta(u) = x$, $\delta(v) = y^m$, 
		$\delta(u v) = x v + u y^m + \lambda x y^m  = z + \lambda x y^m$.
		Since the $A$-module $B = k[[u,v]]$ is generated by $1, u, v, u v$,
		we obtain $I_1 = \Image(\delta) = (x, y^m, z) = I$.

		Suppose $(p, \Sing(Z)) = (3, E_6^1), (5, E_8^1)$. 
		Let $a_j = \dim_k (A/I_j)$ for each $1 \leq j \leq p-1$.
		Since $I_j \subset I_1 = I = \fm$ we have $a_j \geq 1$.
		It suffices to show $\sum_j (a_j - 1) = 0$.
		Suppose there is an action of $G = \bZ/p\bZ = \spanned{g}$ on a K3 surface $X$ 
		such that the quotient $Y = X/G$ is an RDP K3 surface with $(p, \Sing(Y)) = (3, n E_6^1), (5, n E_8^1)$ 
		and that $\Supp \Fix(G) = \pi^{-1}(\Sing(Y))$,
		where $\map{\pi}{X}{Y}$ is the quotient morphism.
		At each singular point $w$ of $Y$, 
		the morphism $\hat{\cO}_{X,\pi^{-1}(w)} \to \hat{\cO}_{Y,w}$ is as above
		(since it is the fundamental covering of $(\Spec \hat{\cO}_{Y,w})^{\sm}$).
		Let $\delta := g - \id \in \End(\pi_* \cO_X)$ and 
		$\cI_j := \Image (\delta^{j} \restrictedto{\Ker \delta^{j+1}}) \subset \cO_Y$ for each $1 \leq j \leq p-1$.
		We have $\chi_Y(\cI_j) = \chi_Y(\cO_Y) - \chi_Y(\cO_Y/\cI_j) = 2 - n a_j$ and
		$2 = \chi_X(\cO_X) = \chi_Y (\pi_* \cO_X) = \chi_Y(\cO_Y) + \sum_{j = 1}^{p-1} \chi_Y(\cI_j)
		= 2 + \sum_j (2 - n a_j)$.
		Since there indeed exist examples with $n = 2$ (Examples \ref{ex:k3:z3} and \ref{ex:k3:z5}),
		we obtain $\sum_j (a_j - 1) = 0$. 

		(\ref{item:RDP:bis:e})
		Clear.
\end{proof}

	\begin{lem} \label{lem:cohomology of RDP on K3:bis}
		Suppose $Y$ is an RDP K3 surface
		and let $Z_i = \Spec \hat{\cO}_{Y, w_i}$ for $w_i \in \Sing(Y)$.
		\begin{enumerate}
		\item \label{item:RDP on K3:bis:D4 E6 E8}
			Suppose $\Sing(Y) = \set{w_1, w_2}$.
			Let $\cI$ be the ideal $\cI = \Ker( \cO_Y \to \bigoplus_{i = 1,2} \cO_{Y,w_i} / \fm_{w_i} )$,
			where $\fm_{w_i}$ are the maximal ideals.
			Then the restriction $\Ext^1_Y(\cI, \cO_Y) \to \Ext^1_{Z_1}(\fm_{w_1}, \cO_{Z_1})$ is an isomorphism.
		\item \label{item:RDP on K3:bis:D8 E8}
			Suppose $\Sing(Y) = \set{w_1}$
			and $(p, w_1)$ is either $(2, D_8^r)$ or $(2, E_8^r)$ with $r \in \set{0,1,2}$. 
			Let $I \subset \cO_{Y,w_1}$ be the ideal defined above 
			(just before Lemma \ref{lem:RDP:bis:dim H1L1}) and
			$\cI = \Ker( \cO_Y \to \cO_{Y,w_1} / I )$.
			Then $\Ext^1_Y(\cI, \cO_Y) \to \Ext^1_{Z_1}(I, \cO_{Z_1})$ is injective and
			its image is a $1$-dimensional $k$-vector space generated by $a \cdot \bar{e}$ for some $a \in A^*$,
			where $\bar{e}$ is an element as in Lemma \ref{lem:RDP:bis:dim H1L1}.
		\end{enumerate}
	\end{lem}

\begin{proof}
		Let $\cI \subsetneq \cO_Y$ be any ideal on an RDP K3 surface $Y$ with $\dim \Supp (\cO_Y / \cI) = 0$ (hence $\Supp (\cO_Y / \cI) \neq \emptyset$).
		By Serre duality (and the equalities $h^1(\cO_Y) = 0$ and $h^2(\cO_Y) = 1$), we obtain 
		$\dim \Ext^1(\cI, \cO_Y) = h^0(\cO / \cI) - 1$ and  
		$\dim \Ext^2(\cI, \cO_Y) = 0$.

		Comparing the long exact sequences for $0 \to \cI \to \cO \to \cO/\cI \to 0$ on $Y$ and $\coprod_i Z_i$,
		we have (since $H^1(Y, \cO) = H^1(Z_i, \cO) = H^2(Z_i, \cO) = 0$)
		\[
		\begin{tikzcd}
			0                                    \arrow[r] &
			\Ext^1_Y(\cI,     \cO)               \arrow[r] \arrow[d] &
			\Ext^2_Y(\cO/\cI, \cO)               \arrow[r] \arrow[d,"\sim"] &
			H^2(Y, \cO)                          \arrow[r] \arrow[d] &
			0 \\
			0                                      \arrow[r] &
			\bigoplus_i \Ext^1_{Z_i}(\cI,     \cO) \arrow[r] &
			\bigoplus_i \Ext^2_{Z_i}(\cO/\cI, \cO) \arrow[r] &
			0,                      &
		\end{tikzcd}
		\]
		hence we obtain an exact sequence 
		\[
		\begin{tikzcd}
			0                                      \arrow[r] &
			\Ext^1_Y(\cI,     \cO)                 \arrow[r] &
			\bigoplus_i \Ext^1_{Z_i}(\cI,     \cO) \arrow[r] &
			H^2(Y, \cO)                            \arrow[r] &
			0 
		\end{tikzcd}
		\]
		compatible with the Frobenius and the pullbacks by inclusions of ideals. 
		Here, the Frobenius on $\Ext^1_Y(\cI, \cO)$ is induced by the one on $H^2_{\Supp(\cO/\cI))}(Y, \cO)$.
		In particular, for any inclusion $\cI \injto \cJ \subsetneq \cO_Y$, 
		the diagram
		\[
		\begin{tikzcd}
			            \Ext^1_Y    (\cJ, \cO) \arrow[r] \arrow[d] & 
			\bigoplus_i \Ext^1_{Z_i}(\cJ, \cO)           \arrow[d] \\
			            \Ext^1_Y    (\cI, \cO) \arrow[r] &
			\bigoplus_i \Ext^1_{Z_i}(\cI, \cO) 
		\end{tikzcd}
		\]
		is a cartesian diagram.

		(\ref{item:RDP on K3:bis:D4 E6 E8})
		Apply this to $\cJ = \Ker (\cO_Y \to \cO_{Y, w_2} / \fm_{w_2})$.
		
		(\ref{item:RDP on K3:bis:D8 E8}) 
		Let $\cJ = \Ker (\cO_Y \to \cO_{Y, w_1} / \fm_{w_1})$
		and consider the diagram above. 
		Write $A = \hat{\cO}_{Y,w_1}$,
		$M := \Ext^1_{Z_1}(I, \cO)$,
		$M_J := \Image(\Ext^1_{Z_1}(  J, \cO) \to \Ext^1_{Z_1}(  I, \cO))$,  and
		$M_Y := \Image(\Ext^1_Y    (\cI, \cO) \to \Ext^1_{Z_1}(  I, \cO))$. 
		We know that $M$ is generated by an element $\bar{e}$ with $\Ann(\bar{e}) = I = (x, y^2, z)$,
		and that $M_J = M[J] = y M$ by Lemma \ref{lem:Ext1}.
		Now $M_Y \subset M$ is a $1$-dimensional $k$-vector subspace
		with $M_Y \cap M_J = \Ext^1_Y(\cJ, \cO) = 0$ by the above cartesian diagram.
		This shows that $M_Y$ has a basis $a \cdot \bar{e}$ for some $a \in A^*$.
\end{proof}

\section{\texorpdfstring{$\bZ/p\bZ$-, $\mu_p$-, $\alpha_p$-coverings of K3 surfaces by K3-like surfaces}{Z/pZ-, mu\_p-, alpha\_p-coverings of K3 surfaces by K3-like surfaces}} \label{sec:restoring}

Let $G$ be one of $\bZ/l\bZ$, $\bZ/p\bZ$, $\mu_p$, or $\alpha_p$ ($l$ is a prime $\neq p$).
Suppose $\map{\pi}{X}{Y}$ is a $G$-quotient morphism between RDP K3 surfaces in characteristic $p$, 
and suppose moreover that $\pi$ is maximal (Definition \ref{def:maximal}) if $G = \mu_p$ or $G = \alpha_p$
and that $X$ is smooth if $G = \bZ/l\bZ$ or $G = \bZ/p\bZ$.
Let $\map{\rho}{\tilde{Y}}{Y}$ be the minimal resolution.

Quotient singularities on $Y$ and some additional properties on $\Pic(Y^{\sm})$ are known for $G = \bZ/l\bZ$
(Theorem \ref{thm:singularities of quotient K3:tame}(\ref{item:quotient tame})).
We prove its analogue for $G = \mu_p, \bZ/p\bZ, \alpha_p$ 
(Theorem \ref{thm:singularities of quotient K3}(\ref{item:quotient})).
For $G = \bZ/l\bZ$, conversely, such properties on $Y$ recovers a $\bZ/l\bZ$-covering $X \to Y$
(Theorem \ref{thm:singularities of quotient K3:tame}(\ref{item:restore tame})).
We state and prove its analogue for $G = \mu_p, \bZ/p\bZ, \alpha_p$ 
(Theorem \ref{thm:singularities of quotient K3}(\ref{item:restore})).
However, in the converse statement for $\mu_p$ and $\alpha_p$,
the covering is a K3-like surface (Definition \ref{def:k3-like}) but not necessarily birational to a K3 surface.
This situation is similar to the canonical $\mu_2$- or $\alpha_2$-coverings of classical or supersingular Enriques surfaces in characteristic $2$,
where the covering is K3-like (\cite{Bombieri--Mumford:III}*{Proposition 9}) but not necessarily birational to a K3 surface.

\begin{thm} \label{thm:singularities of quotient K3:tame}
	\ 
	\begin{enumerate}
	\item \label{item:quotient tame}
	Let $\pi$ be as above and suppose $G = \bZ/l\bZ$.
	Then $l \leq 7$, 
	$\Sing(Y) = \frac{24}{l+1} A_{l-1}$,
	and $\card{\Pic(Y^{\sm})_{\tors}} = l$.
	The $l$-torsion is given by a divisor on $\tilde{Y}$ whose multiple by $l$ is linearly equivalent to 
	$\sum_{i,j} j a_i e_{i,j}$
	for a suitable numbering $e_{i,j}$ ($1 \leq i \leq \frac{24}{l+1}$, $1 \leq j \leq l-1$, $e_{i,j} \cdot e_{i,j+1} = 1$) of exceptional curves of $\tilde{Y}$.
	Here $(a_1, \dots, a_{24/(l+1)})$ is given by
	$(1, \dots, 1)$, $(1, \dots, 1)$, $(1,1,2,2)$, $(1,2,4)$ for $l = 2,3,5,7$ respectively.
	Every prime $l \leq 7$ occur in every characteristic $\neq l$.
	
	\item \label{item:restore tame}
	Conversely, let $Y$ be an RDP K3 surface in characteristic $\neq l$ with $\Sing(Y) = \frac{24}{l+1} A_{l-1}$ and
	$\Pic(Y^{\sm})_{\tors} \neq 0$.
	Then there exists a smooth K3 surface $X$ and a $\bZ/l\bZ$-quotient morphism $\map{\pi}{X}{Y}$
	with $\Supp \Fix(\bZ/l\bZ) = \pi^{-1}(\Sing(Y))$.
	\end{enumerate}
\end{thm}
\begin{proof}
	(\ref{item:quotient tame})
	The assertions $l \leq 7$ and $\Sing(Y) = \frac{24}{l+1} A_{l-1}$
	are proved by 
	Nikulin \cite{Nikulin:auto}*{Section 5} ($p = 0$) and Dolgachev--Keum \cite{Dolgachev--Keum:auto}*{Theorem 3.3} ($p > 0$).
	Then the eigenspace of $\pi_* \cO_X$ for a nontrivial eigenvalue gives an invertible sheaf whose $l$-th power is isomorphic
	to $\cO_{\tilde{Y}}(- \sum_{i,j} f_l(a_i, j) e_{i,j})$ for a suitable numbering,
	where $f_l$ is the function defined in Lemma \ref{lem:Picard group of A}.
	See \cite{Matsumoto:k3mun}*{Theorem 7.1} for details.
	See the proof of (\ref{item:restore tame}) to show that $\Pic(Y^{\sm})$ has no more torsion.

	Examples for each $l$ are well-known.
	
	\medskip
	
	(\ref{item:restore tame})
	By the exact sequence
	\[
		0 \to \bigoplus_{i,j} \bZ [e_{i,j}] \to \Pic(\tilde{Y}) \to \Pic(Y^{\sm}) \to 0 ,
	\]
	where $e_{i,j}$ runs through the exceptional curves of $\tilde{Y} \to Y$
	over $w_i \in \Sing(Y)$,
	and the fact that discriminant group of the $A_{l-1}$ lattice is cyclic of order $l$,
	we see that a nontrivial element of $\Pic(Y^{\sm})_{\tors}$ 
	is of order $l$ and induces $\Delta \in \Pic(\tilde{Y})$
	satisfying $\sum b_{i,j} e_{i,j} = l \Delta \in l \Pic(\tilde{Y})$ for some coefficients $b_{i,j} \in \bZ$
	not all divisible by $l$.
	By Lemma \ref{lem:Picard group of A}(\ref{lem:Picard:coefficients}), 
	there exist integers $a_i$ satisfying $b_{i,j} \equiv j a_i \pmod{l}$.
	We may assume $a_i \in \set{0, \dots, \floor{l/2}}$ and 
	$b_{i,j} = (j a_i \bmod l) \in \set{0, 1, \dots, l-1}$.
	Computing the intersection number $(l \Delta)^2$, 
	we obtain $\Delta^2 = - l^{-1} \sum_i a_i(l-a_i) \in 2 \bZ$.
	Moreover we have $\Delta^2 \neq -2$ since if $\Delta^2 = -2$ then $\Delta$ or $- \Delta$ is effective, 
	which leads to a contradiction.
	The only solution $(a_i)$ is as in the statement of (\ref{item:quotient tame}), 
	up to the numbering of the RDPs $w_i$.
	
	Suppose there are two $l$-torsion elements 
	$\sum (j a_i \bmod l) e_{i,j}$ and $\sum (j a'_i \bmod l) e_{i,j}$
	with $(a_i)$ and $(a'_i)$ linearly independent in $\bF_l^{24/(l+1)}$.
	Then for some $m \in \bZ$, 
	the elements $a_i - m a'_i \in \bF_l$ are neither all zero nor all nonzero, contradicting the observation above.
	Hence $\Pic(Y^{\sm})_{\tors}$ is of order $l$.

	Now suppose there is a nontrivial $l$-torsion of $Y^{\sm}$.
	Construct a $\mu_l$-covering $\map{\pi}{X}{Y}$ as in Lemma \ref{lem:Picard group of A}.
	Then $X$ is regular above $\Sing(Y)$.

	It is clear from the construction that $\pi$ is finite \'etale outside $\Sing(Y)$. Hence $X$ is a smooth proper surface.
	A non-vanishing $2$-form on $Y^{\sm}$ pullbacks to a non-vanishing $2$-form on $X \setminus \pi^{-1}(\Sing(Y))$, which then extends to $X$.
	For each $0 < k < l$, we have $(\tilde{L}_k)^2 = -4$ by the calculation of $\Delta^2$ above, 
	hence $\chi(\tilde{Y}, \tilde{L}_k) = 0$,
	hence $\chi(Y, \bar{L}_k) = \chi(Y, \rho_* \tilde{L}_k) = \chi(\tilde{Y}, \tilde{L}_k) = 0$
	since $R^i \rho_* \tilde{L}_k = 0$ for $i > 0$.
	Here $\chi$ is the Euler--Poincar\'e characteristic of the sheaf cohomology.
	Hence $\chi(X, \cO) = \chi(Y, \cO) + \sum_{0 < k < l} \chi(Y, \bar{L}_k) = 2 + 0 = 2$.
	Hence $X$ is a K3 surface.

	Alternatively, we can conclude that $X$ is a K3 surface
	from by computing 
	the Euler--Poincar\'e characteristic $\chi$ of the $l'$-adic cohomology 
	for an auxiliary prime $l' \neq \charac k$. 
	Indeed, as $\pi$ is finite \'etale outside $\Sing(Y)$, we have 
	$\chi(X \setminus \pi^{-1}(\Sing(Y))) =  l \cdot \chi(Y^{\sm})$,
	hence $\chi(X) - \frac{24}{l+1} = l \cdot (\chi(Y) - l \frac{24}{l+1})$, 
	therefore $\chi(X) = 24$.
\end{proof}

\begin{defn}[following \cite{Bombieri--Mumford:III}*{Proposition 9}] \label{def:k3-like}
A proper reduced Gorenstein (not necessarily normal) surface $X$ is \emph{K3-like} if
$h^i(X, \cO_X) = 1,0,1$ for $i = 0,1,2$, and the dualizing sheaf $\omega_X$ is isomorphic to $\cO_X$.

RDP K3 surfaces are K3-like.
\end{defn}
\begin{thm} \label{thm:singularities of quotient K3}
	\ 
	\begin{enumerate}
	\item \label{item:quotient}
	Let $G$ be $\mu_p$, $\bZ/p\bZ$, or $\alpha_p$.
	Let $\pi$ be as in the beginning of this section.
	Then $(G, \Sing(Y), \card{\Pic(Y^{\sm})_{\tors}})$ is one of those listed in Table \ref{table:singularities of quotient K3}.
	If $G = \mu_p$, then the $p$-torsion is given by a divisor on $\tilde{Y}$ whose multiple by $p$ is linearly equivalent to 
	$\sum_{i,j} j a_i e_{i,j}$,
	with $a_i$ as in Theorem \ref{thm:singularities of quotient K3:tame}(\ref{item:quotient tame}).
	Every case occur.
	
	\item \label{item:restore}
	Conversely, 
	suppose $Y$ is an RDP K3 surface in characteristic $p$ with $\Sing(Y)$ as in Table \ref{table:singularities of quotient K3},
	let $G$ be the corresponding group scheme, 
	and if $G = \mu_p$ suppose moreover 
	$\Pic(Y^{\sm})_{\tors} \neq 0$.
	Then there exists a $G$-quotient morphism $\map{\pi}{X}{Y}$ from a proper K3-like surface $X$
	with $\Sing(X) \cap \pi^{-1}(\Sing(Y)) = \emptyset$
	and $\Supp \Fix(G) = \pi^{-1}(\Sing(Y))$.
	If $G = \bZ/p\bZ$ then $X$ is a smooth K3 surface.
	If $G = \mu_p$ or $G = \alpha_p$, then one of the following holds:
	\begin{itemize}
		\item $X$ is an RDP K3 surface.
		\item $X$ is a normal rational surface with $\Sing(X)$ consisting of a single non-RDP singularity, and $p \geq 3$.
		\item $X$ is a non-normal rational surface with $\dim \Sing(X) = 1$.
	\end{itemize}
	\end{enumerate}
\end{thm}
	All three cases of (\ref{item:restore}) occur for all $G \in \set{\mu_p \, (p \leq 7), \, \alpha_p \, (p \leq 5)}$
	unless otherwise stated.
	See Section \ref{subsec:K3 quotients} for examples.
\begin{rem} \label{rem:cyclic}
	Dolgachev--Keum studied $\bZ/p\bZ$-actions on K3 surfaces in characteristic $p$.
	Their results in the case of K3 quotients are as follows \cite{Dolgachev--Keum:wild-p-cyclic}*{Theorem 2.4 and Remark 2.6}:
	Suppose $G = \bZ/p\bZ$ acts on a K3 surface $X$ in characteristic $p$ with quotient $Y$ birational to a K3 surface.
	Then 
	\begin{itemize}
	\item $\Fix(G)$ is isolated and $\Sing(Y) = \pi(\Fix(G))$, and each singularity of $Y$ is an RDP.
	\item $1 \leq \# \Sing(Y) \leq 2$ and $p \leq 5$.
	\item If $p = 2$, then $\Sing(Y)$ is one of $1 D_4^1$, $2 D_4^1$, $1 D_8^2$, or $1 E_8^2$.
	\end{itemize}
(The $E_8^2$ on the last is misprinted as $E_8^4$ in \cite{Dolgachev--Keum:wild-p-cyclic}*{Remark 2.6}.)

	Also note that if $G = \bZ/p\bZ$ then each quotient RDP singularity on $Y$ should be one of those
	having fundamental group $\bZ/p\bZ$ and smooth fundamental covering,
	which, due to Artin \cite{Artin:RDP}*{Sections 4--5}, are the following:
	\begin{itemize}
		\item $D_{4r}^r$ ($r \geq 1$) and $E_8^2$ if $p = 2$.
		\item $E_6^1$ if $p = 3$.
		\item $E_8^1$ if $p = 5$.
		\item There are no such RDP if $p \geq 7$.
	\end{itemize}
	Note that these RDPs and their $\bZ/p\bZ$-coverings are discussed in Section \ref{subsec:covering:alpha}.
	Thus, it was known that $\Sing(Y)$ is $1 E_6^1$ or $2 E_6^1$ if $p = 3$, and $1 E_8^1$ or $2 E_8^1$ if $p = 5$.
	Compared to these results, we exclude the possibility of $1 D_4^1$ ($p = 2$), $1 E_6^1$ ($p = 3$), and $1 E_8^1$ ($p = 5$).
\end{rem}

\begin{table} 
	\caption{Singularities of $\bZ/l\bZ$-, $\mu_p$-, $\bZ/p\bZ$-, and $\alpha_p$-quotient K3 surfaces in characteristic $p$} \label{table:singularities of quotient K3} 
	\begin{tabular}{lllll} 
		\toprule
	char. & $G$ & & $\Sing(Y)$ & $\card{\Pic(Y^{\sm})_{\tors}}$ \\
		\midrule
		$p \geq 0$ & $\bZ/l\bZ$ & $l \leq 7$ prime, $l \neq p$ & $\frac{24}{l+1} A_{l-1}$           & $l$ \\
		\midrule
		$p$        & $\mu_p$    & $p \leq 7$                   & $\frac{24}{p+1} A_{p-1}$           & $p$  \\
		\midrule
		$5$        & $\bZ/5\bZ$ &                              & $2 E_8^1$                          & $1$ \\
		$3$        & $\bZ/3\bZ$ &                              & $2 E_6^1$                          & $1$ \\
		$2$        & $\bZ/2\bZ$ &                              & $2 D_4^1$, $1 D_8^2$, or $1 E_8^2$ & $1$ \\
		\midrule 
		$5$        & $\alpha_5$ &                              & $2 E_8^0$                          & $1$ \\
		$3$        & $\alpha_3$ &                              & $2 E_6^0$                          & $1$ \\
		$2$        & $\alpha_2$ &                              & $2 D_4^0$, $1 D_8^0$, or $1 E_8^0$ & $1$ \\
		\bottomrule
	\end{tabular} 
%	char. & $G$ & & $\Sing(Y)$ & moreover: \\
%		\midrule
%		$p \geq 0$ & $\bZ/l\bZ$ & $l \leq 7$ prime, $l \neq p$ & $\frac{24}{l+1} A_{l-1}$           & $\sum_{i,j} j a_i e_{i,j} \in l \Pic(\tilde{Y})$ \\
%		\midrule
%		$p$        & $\mu_p$    & $p \leq 7$                   & $\frac{24}{p+1} A_{p-1}$           & $\sum_{i,j} j a_i e_{i,j} \in p \Pic(\tilde{Y})$ \\

\end{table}

\begin{proof}[Proof of Theorem \ref{thm:singularities of quotient K3}]
	(\ref{item:quotient}) 
	Consider the case $G = \mu_p, \alpha_p$.
	The assertion on $\Sing(Y)$ is showed in Theorem \ref{thm:alpha_p K3 K3}.
	If $G = \mu_p$, 
	the author showed \cite{Matsumoto:k3mun}*{Theorem 7.1} that the eigenspace of 
	$\pi_* \cO_X$ for a nontrivial eigenvalue (of the derivation $D$ of multiplicative type corresponding to the $\mu_p$-action) 
	gives an invertible sheaf whose $p$-th power is isomorphic
	to $\cO_{\tilde{Y}}(- \sum_{i,j} f_p(a_i, j) e_{i,j})$ for a suitable numbering.
	Here $f_p$ is the function defined in Lemma \ref{lem:Picard group of A}.
	The same (characteristic-free) argument as in the proof of Theorem \ref{thm:singularities of quotient K3:tame}(\ref{item:restore tame}) 
	shows $\card{\Pic(Y^{\sm})_{\tors}}= p$.
	A similar calculation shows that 
	if $(p, \Sing(Y))$ is $(2, 2 D_4^r)$ etc.\ then $\Pic(Y^{\sm})_{\tors} = 0$.

	Consider the case $G = \bZ/p\bZ$.
	As in Proposition \ref{prop:structure of quotient} (using the usual ramification formula in place of the Rudakov--Shafarevich formula)
	we have that $\Fix(G)$ is finite
	and each point in $\pi(\Fix(G))$ is an RDP.
	Let $\Sing(Y) = \set{w_i}$.
	Then each $w_i$ is one of the RDPs appearing in Remark \ref{rem:cyclic}, 
	hence in Section \ref{subsec:covering:alpha},
	and let $m_i$ be the integer defined there. 
	Let $\cI_j = \Image (\delta^{j} \restrictedto{\Ker \delta^{j+1}})$ for $1 \leq j \leq p-1$,
	where $\delta = g - \id \in \End(\pi_* \cO_X)$.
	We have $\chi_Y(\cI_j) = \chi_Y(\cO_Y) - \chi_Y(\cO/\cI_j) = 2 - \sum_i m_i$
	by Lemma \ref{lem:cohomology of RDP:bis}(\ref{item:RDP:bis:V}).
	Since $2 = \chi_X (\cO_X) = \chi_Y(\cO_Y) + \sum_j \chi_Y(\cI_j)
	= 2 + (p-1) (2 - \sum_i m_i)$,
	we obtain $\sum_i m_i = 2$.
	This proves the assertion on $\Sing(Y)$.

	\medskip

	(\ref{item:restore})
	Suppose $G = \mu_p$.
	As in the case of $\bZ/l\bZ$ (Theorem \ref{thm:singularities of quotient K3:tame}(\ref{item:restore tame})), 
	with $l$ replaced with $p$, 
	we obtain a $\mu_p$-covering $\map{\pi}{X}{Y}$.
	Since in this case $\pi$ is not \'etale over $Y^{\sm}$, $X$ may be singular.
	
	By Proposition \ref{prop:covering},
	$X$ is Gorenstein with $\omega_X \cong \cO_X$,
	and we have a derivation $D$ on $Y$
	satisfying $\Fix(D) = \pi(\Sing(X))$ and $Y^{D} = \pthpower{(\normalization{X})}$.
	Here $\normalization{-}$ is the normalization.
	Also $X$ is normal if and only if the divisorial part 
	$\divisorialfix{D \restrictedto{Y^{\sm}}}$ of $\Fix(D \restrictedto{Y^{\sm}})$ is zero.
		
	As in the $\bZ/l\bZ$ case,
	we have $\chi(X, \cO_X) = 2$.
	Since $X$ is connected and reduced we have $h^0(X, \cO_X) = 1$,
	and $h^2(X, \cO_X) = h^2(X, \omega_X) = h^0(X, \cO_X) = 1$.
	Thus $X$ is K3-like.

	Let $D' = D \restrictedto{Y^{\sm}}$ 
	and suppose $\divisorialfix{D'} \neq 0$. 
	Then $X$ is non-normal.
	By Proposition \ref{prop:structure of quotient}, 
	$Y^D = \pthpower{(\normalization{X})}$ is rational, and hence $X$ is rational.

	Now suppose $\divisorialfix{D'} = 0$.
	Then $X$ is normal and we have $Y^D = \pthpower{X}$.
	As in the proof of Theorem \ref{thm:duality} we have $D^p = \lambda D$ for some scalar $\lambda$, 
	and we may assume $\lambda = 1$ or $\lambda = 0$ (by replacing $D$ by a suitable multiple).
	
	Suppose $\lambda = 1$.
	Since $D$ is a derivation of multiplicative type with $D(\omega) = 0$ (Proposition \ref{prop:covering}(\ref{item:derivation})),
	where $\omega$ is a global $2$-form on $Y^{\sm}$,
	it follows from \cite{Matsumoto:k3mun}*{Theorem 6.1} that $Y^D$ is an RDP K3 surface.

	Next suppose $\lambda = 0$.
	By Theorem \ref{thm:c2 derivation} and Lemma \ref{lem:derivation on resolution}
	and the assumption on $\Sing(Y)$, we have $\deg \isolatedfix{D'} = 24/(p+1)$.
	Then, by Corollary \ref{cor:configuration of isolated fixed point},
	either every singularity of $X$ is an RDP,
	or $X$ has a single singularity and it is non-RDP and $p \geq 3$.
	In the latter case $X$ is a rational surface by Proposition \ref{prop:structure of quotient}.

	\bigskip
	
	Now we consider the cases $G = \bZ/p\bZ$ and $G = \alpha_p$ simultaneously.
	Write $\Sing(Y) = \set{w_i}_{i=1}^N$.
	Define an ideal $\cI = \cI_1 \subset \cO_Y$ 
	by $\cI = \Ker (\cO_Y \to \bigoplus_i (\cO_{Y, w_i} / I_{w_i}))$,
	where $I_{w_i} \subset \cO_{Y,w_i}$ is as in Section \ref{subsec:covering:alpha}.
	Then we have $h^0(\cO_Y/\cI) = 2$ and hence $\dim \Ext^1(\cI, \cO) = 1$ (as in the proof of Lemma \ref{lem:cohomology of RDP on K3:bis}).
	Take a nonzero element $e \in \Ext^1(\cI, \cO)$ corresponding to a non-split extension
	\[
	0 \to \cO_{Y} \to V \namedto{\delta} \cI \to 0
	\]
	(which is unique up to scalar)
	and let $e_i := e \restrictedto{Z_i} \in \Ext^1_{Z_i}(I_{w_i}, \cO) $ 
	be its restriction to $Z_i = \Spec \cO_{Y, w_i}$.
	By Lemma \ref{lem:cohomology of RDP on K3:bis}, $e_i$ generates this group as an $\cO_{Y,w_i}$-module. 
	
	As in the proof of Lemma \ref{lem:cohomology of RDP on K3:bis}, 
	we have a diagram with exact rows
		\[
		\begin{tikzcd}
			0                                    \arrow[r] &
			\Ext^1_Y(\cI,     \cO)               \arrow[r,"\beta"] \arrow[d,shift left] \arrow[d,shift right] &
			\bigoplus_{i = 1}^N \Ext^1_{Z_i}(I_{w_i},     \cO) \arrow[r,"\gamma"] \arrow[d,shift left] \arrow[d,shift right] &
			H^2(Y, \cO)                          \arrow[r] \arrow[d,shift left] \arrow[d,shift right] &
			0 \\
			0                                    \arrow[r] &
			\Ext^1_Y(\pthpower{\cI},     \cO)               \arrow[r,"\beta'"] &
			\bigoplus_{i = 1}^N \Ext^1_{Z_i}(\pthpower{I_{w_i}},     \cO) \arrow[r,"\gamma'"] &
			H^2(Y, \cO)                          \arrow[r] &
			0 
		\end{tikzcd}
		\]
	where the double vertical arrows are $F$ and $\iota^*$.
	By Lemma \ref{lem:cohomology of RDP:bis}(\ref{item:RDP:bis:unit}),
	we have $\Image(F_{\mathrm{middle}}) \subset \Image(\iota^*_{\mathrm{middle}})$.
	Hence 
	$\Image(F_{\mathrm{left}}) 
	\subset \beta'^{-1}(\Image(F_{\mathrm{middle}}))
	\subset \beta'^{-1}(\Image(\iota^*_{\mathrm{middle}}))
	= \Image(\iota^*_{\mathrm{left}})$,
	where the last equality follows from snake lemma (applied to the commutative diagram for $\iota^*$)
	since $\iota^*_{\mathrm{right}} = \id$.
	Since $\Ext^1_Y(\cI, \cO)$ is $1$-dimensional, we obtain $F(e) = \lambda \cdot \iota^*(e)$ for some $\lambda \in k$.
	Clearly the same equality holds for $e_i$ for each $i$,
	and since $e_i$ is a generator, we have the following equivalence:
	$\lambda \neq 0$ (resp.\ $\lambda = 0$)
	if and only if the coindex $r$ of the RDP(s) is $\neq 0$ (resp.\ $r = 0$)
	if and only if $G = \bZ/p\bZ$ (resp.\ $G = \alpha_p$).

	Since the restriction $e \restrictedto{Y^{\sm}} \in H^1(Y^{\sm}, \cO)$ is annihilated by $F - \lambda$, it induces a $G$-covering $X \restrictedto{Y^{\sm}} \to Y^{\sm}$ as follows.
	Take an open covering $\set{O_h}$ of $Y^{\sm}$ fine enough 
	and take local sections $t_h \in V$ with $\delta(t_h) = 1 \in \cI \restrictedto{Y^{\sm}} = \cO_{Y^{\sm}}$.
	Let $e_{h h'} = t_h - t_{h'} \in \cO$ (this $1$-cocycle represents $e \restrictedto{Y^{\sm}}$).
	Then since $F(e) = \lambda \cdot \iota^*(e)$ there exists $c_h \in \cO$ with $e_{h h'}^p - \lambda e_{h h'} = c_h - c_{h'}$.
	We equip the locally-free sheaf 
	$V_{p-1} := \Sym^{p-1}_{\cO_{Y^{\sm}}}(V \restrictedto{Y^{\sm}})$ on $Y^{\sm}$
	with an $\cO_{Y^{\sm}}$-algebra structure
	by $t_h^p := \lambda t_h + c_h$,
	and let $X \restrictedto{Y^{\sm}} = \sSpec V_{p-1}$.
	Since each $e_i$ is a generator, this $Y^{\sm}$-scheme is regular above a neighborhood of $w_i$ 
	by Lemma \ref{lem:cohomology of RDP:bis}(\ref{item:RDP:bis:U smooth}).
	By filling the holes above $\Sing(Y)$ by normalization, 
	we obtain a $Y$-scheme $X$ that is isomorphic to $X \restrictedto{Y^{\sm}}$ outside $\Sing(Y)$ and regular above a neighborhood of $\Sing(Y)$
	(again by Lemma \ref{lem:cohomology of RDP:bis}(\ref{item:RDP:bis:U smooth})).
	Extend $\map{\delta}{V \restrictedto{Y^{\sm}}}{\cO_{Y}}$ 
	to an endomorphism $\delta \in \End_{\cO_{Y^{\sm}}}(V_{p-1})$
	by $\delta(a \otimes b) = \delta(a) \otimes b + a \otimes \delta(b) + \lambda^{1/(p-1)} \delta(a) \otimes \delta(b)$
	(note that this is compatible with the equality $t_h^p = \lambda t_h + c_h$)
	and then extend it to an endomorphism $\delta \in \End_{\cO_{Y}}(\cO_X)$ (by using normality of $X$ above a neighborhood of $\Sing(Y)$).
	Then $\delta$ corresponds to a $G$-action on $X$, and $\map{\pi}{X}{Y}$ is a $G$-covering.

	Let $\cI_j = \Image (\delta^{j} \restrictedto{\Ker \delta^{j+1}})$ for $1 \leq j \leq p-1$.
	Then we have $\cI_j = \Ker (\cO_Y \to \bigoplus_i (\cO_{Y, w_i} / I_{w_i,j}))$
	with $I_{w_i,j}$ as in Section \ref{subsec:covering:alpha},
	hence we have $\chi(\cI_j) = 2 - \sum_i \dim (\cO_{Y, w_i} / I_{w_i,j}) = 0$.
	Since $\pi_* \cO_X$ has $\cO_Y, \cI_1, \dots, \cI_{p-1}$ as a composition series
	and since $\chi(\cO_Y) = 2$ and $\chi(\cI_j) = 0$ for each $1 \leq j \leq p-1$,
	we have $\chi(\cO_X) = 2$.

	Suppose $G = \bZ/p\bZ$.
	It is clear from the construction that $\pi$ is finite \'etale outside $\Sing(Y)$. 
	Hence $X$ is smooth, 
	and a non-vanishing $2$-form on $Y^{\sm}$ pullbacks to a non-vanishing $2$-form on $X \setminus \pi^{-1}(\Sing(Y))$, which then extends to $X$.
	Hence $X$ is a K3 surface.

	Suppose $G = \alpha_p$.
	We conclude by using Proposition \ref{prop:covering} as in the case of $G = \mu_p$.
\end{proof}

\begin{rem}
In the proof of the case of $G = \bZ/p\bZ, \alpha_p$ of Theorem \ref{thm:singularities of quotient K3}(\ref{item:restore}), we showed $F(e_i) = \lambda \cdot \iota^*(e_i)$ for each RDP $w_i$. 
A similar argument proves an unexpected consequence 
on non-existence of certain RDP K3 surfaces:
	If $Y$ is an RDP K3 surface in characteristic $p$,
	then $(p,\Sing(Y))$ cannot be 
	$(2, D_4^0 + D_4^1)$,
	$(3, E_6^0 + E_6^1)$,
	$(5, E_8^0 + E_8^1)$,
	$(2, D_{4m}^1)$ ($m \geq 2$),
	nor
	$(2, E_8^1)$.
	This implies that RDP K3 surfaces in characteristic $2$ cannot have 
	$D_{2n}^r$ nor $D_{2n+1}^r$
	if $0 < r < n-1$ and $2 \notdivides (n-r)$,
	since a partial resolution of such an RDP produces a 
	$D_{2(n-r+1)}^1$ with $2 \divides (n - r + 1)$ and $n - r + 1 > 2$.

We do not prove this in this paper,
as we give a more general result,
relating singularities to the height of the K3 surface, 
in a subsequent paper \cite{Matsumoto:k3rdpht}*{Theorem 1.2}.
\end{rem}

\section{\texorpdfstring{Bound of $p$ for $\alpha_p$-actions}{Bound of p for alpha\_p-actions}} \label{sec:bound p}

\begin{thm} \label{thm:bound p}
Let $k$ be an algebraically closed field of characteristic $p > 0$.
Then there exists an RDP K3 surface equipped with a nontrivial action of $\mu_p$ (resp. $\alpha_p$) if and only if $p \leq 19$ (resp.\ $p \leq 11$).
\end{thm}

\begin{proof}
The case of $\mu_p$ is given in \cite{Matsumoto:k3mun}*{Theorem 8.2}.
Examples of RDP K3 surfaces with a nontrivial action of $\alpha_p$ in characteristic $p$ for $p \leq 11$ are given in Section \ref{subsec:rational quotients}.
It remains to show that if $p \geq 13$ then there is no such example.

Suppose $p \geq 13$ and $X$ is an RDP K3 surface equipped with a nontrivial action of $\alpha_p$.
Since smooth K3 surface have no global derivations, $X$ has an RDP $x$,
and since RDPs fixed by the $\alpha_p$-action can be blown up, we may assume $x$ and all other RDPs are not fixed.
Since $p \geq 7$, by Theorem \ref{thm:mu_p alpha_p RDP}(\ref{thm:mu_p alpha_p RDP:non-fixed}), $x$ and all other RDPs are of type $A_{mp-1}$.
Since $p \geq 13$, It follows that $m = 1$ and that $x$ is the only RDP.

We observe that the tangent module $T_B$ of the RDP $B = \hat{\cO}_{X,x} = k[[x,y,z]]/(xy - z^p)$ of type $A_{p-1}$ 
is a free $B$-module with basis $e_1 = x \partialdd{x} - y \partialdd{y}, e_2 = \partialdd{z}$,
and that $a_1 e_1 + a_2 e_2$ ($a_1, a_2 \in B$) fixes the closed point if and only if $a_2 \in \idealm_B$.

Let $\Delta := \divisorialfix{D}$.
The morphism $H^0(X, \cO_X(\Delta)) \to H^0(X, T_X) \to T_B \to B/\idealm_B = k$
taking $f$ to the coefficient modulo $\idealm_B$ of $e_2$ in $(f D) \restrictedto{B}$
is injective, since if $f$ is in the kernel then $fD$ extends to an element of $H^0(\tilde{X}, T_{\tilde{X}}) = 0$. 
Hence $\dim H^0(X, \cO_X(\Delta)) = 1$.
It follows that $\Supp(\Delta)$ is a finite disjoint union (possibly empty) of ADE configurations of smooth rational curves.

Let $X' = X \setminus \Supp(\Fix(D)) = X \setminus    (\Supp(\Delta) \cup \Supp(\isolatedfix{D}))$,
$X'' = X' \setminus \set{x}$,
and $Y' = X'^D$, $Y'' = X''^D$.
Then $X''$ and $Y''$ are smooth.
Since  $X''$ is the complement in a K3 surface of a finite disjoint union of ADE configurations and closed points, 
we have $H^0(X'', \cO_X) = k$.
Moreover we can compute $\Pic(X'')[p]$ as in the proof of Theorem \ref{thm:singularities of quotient K3:tame}(\ref{item:restore tame}),
and since there is at most one RDP of type $A_{p-1}$ (since $p \geq 13$) we obtain $\Pic(X'')[p] = 0$.

Fix a nonzero element $\omega_X$ of $H^0(X^{\sm}, \Omega^2_X)$ (which is unique up to $k^*$).
We have $D(\omega_X) = 0$ since $D$ acts nilpotently on this $1$-dimensional space.
Let $\omega_Y$ be the $2$-form on $Y''$ corresponding to $\omega_X \restrictedto{X''}$ via the isomorphism 
$H^0(X'', \Omega^2_{X})^{D} \cong H^0(Y'', \Omega^2_{Y})$
of Proposition \ref{prop:2-forms}.
Let $D_{Y''}$ be the derivation on $Y''$ of Proposition \ref{prop:covering}.
Then $\Fix(D_{Y''}) = \emptyset$, $D_{Y''}$ is $p$-closed, and $Y''^{D_{Y''}} = \pthpower{X''}$.
Write $D_{Y''}^p = h D_{Y''}$ with $h \in k(Y)$. Then $h \in H^0(Y'', \cO_Y) \subset H^0(X'', \cO_X) = k$.

As $D_{Y''}$ extends to $D_{Y'}$ on $Y'$ with $\Fix(D_{Y'}) \subset \set{\pi(x)}$, we have $h \neq 0$, 
since $x$ (of type $A_{p-1}$) is not $\alpha_p$-quotient by Lemma \ref{lem:degree of isolated fixed point}. 
We may assume $h = 1$ (by replacing $D_{Y'}$ with a multiple by $k^*$).
Then $\thpower{Y''}{1/p} \to X''$ is a $\mu_p$-covering,
which corresponds to an element of $\Pic(X'')[p] = 0$.
Since $H^0(X'', \cO_X) = k$, such a covering is non-reduced, which is absurd.
\end{proof}

\section{Coverings of supersingular Enriques surfaces in characteristic 2} \label{sec:enriques}
	In this section we give a restriction on
	the singularities of the canonical $\alpha_2$-covering 
	of a supersingular Enriques surface in characteristic $2$,
	and give some examples.
	A more detailed study will be given in a subsequent paper \cite{Matsumoto:k3cover}.

	Let $X$ be a classical or supersingular (smooth) Enriques surface in characteristic $2$
	(i.e.\ an Enriques surface with $\Pic^{\tau}(X) = \bZ/2\bZ$ or $\alpha_2$ respectively).
	Let $\map{\pi}{Y}{X}$ be its canonical $\mu_2$- or $\alpha_2$-cover.
	We recall some known properties of $Y$.
\begin{itemize}
	\item (\cite{Bombieri--Mumford:III}*{Proposition 9}) 
	$Y$ is K3-like 
	(as in Definition \ref{def:k3-like}, i.e.\ $h^i(Y, \cO_Y) = 1,0,1$ for $i = 0,1,2$, and the dualizing sheaf $\omega_Y$ is isomorphic to $\cO_Y$).
	There exists a global regular $1$-form $\eta \neq 0$ on $X$, unique up to scalar,
	and it satisfies $\Sing(Y) = \pi^{-1}(\Zero(\eta))$.
	The zero locus $\Zero(\eta)$ is nonempty (hence $Y$ is singular somewhere), and if it has no divisorial part then it is of degree $12$.
	\item (\cite{Cossec--Dolgachev:enriques}*{Theorem 1.3.1})
	One of the following holds.
	\begin{itemize}
		\item $Y$ has only RDPs as singularities, and $Y$ is an RDP K3 surface.
		\item $Y$ has only isolated singularities, it has exactly one non-RDP singularity and that is an elliptic double point, 
		and $Y$ is a normal rational surface.
		\item $Y$ has $1$-dimensional singularities, and $Y$ is a non-normal rational surface.
	\end{itemize}
	\item (\cite{Ekedahl--Shepherd-Barron:exceptional})
	Non-normal examples exist. 
	More detailed properties, for example on the structure of the divisorial part of $\Zero(\eta)$, are proved.
	\item (\cite{Ekedahl--Hyland--Shepherd-Barron}*{Corollary 6.16})
	If $Y$ is an RDP K3 surface,
	then $\Sing(Y)$ is one of 
	$12 A_1$, $8 A_1 + D_4^0$, $6 A_1 + D_6^0$, $5 A_1 + E_7^0$,
	$3 D_4^0$, $D_8^0 + D_4^0$, $E_8^0 + D_4^0$, or $D_{12}^0$.
	\item (\cite{Schroer:K3-like}*{Sections 13--14})
	If $Y$ has an elliptic double point singularity, then there are no other singularities on $Y$.
	Such examples exist.
\end{itemize}

By using similar arguments as in Theorem \ref{thm:singularities of quotient K3}(\ref{item:restore}),
we can give some restrictions on the singularities of the canonical $\alpha_2$-coverings of \emph{supersingular} Enriques surfaces in characteristic $2$,
assuming it is an RDP K3 surface.
Since this method depends on the triviality of the canonical divisor of $X$, it cannot be applied to classical Enriques surfaces.
\begin{thm} \label{thm:alpha_p K3 Enr}
	Let $\map{\pi}{Y}{X}$ be the canonical $\alpha_2$-covering of 
	a supersingular Enriques surface $X$.
	If $Y$ is an RDP K3 surface, then 
		$\Sing(Y)$ is one of $12 A_1$, $3 D_4^0$, $D_8^0 + D_4^0$, $E_8^0 + D_4^0$, or $D_{12}^0$.
\end{thm}

\begin{proof}
	By Theorem \ref{thm:duality},
	$X \to \secondpower{Y}$ is the quotient by a derivation $D$ of multiplicative or additive type 
	with $\divisorialfix{D} = 0$.
	Then $\deg \isolatedfix{D} = 12$ by Theorem \ref{thm:c2 derivation}.
	The assertion follows from by Lemma \ref{lem:degree of isolated fixed point}.
\end{proof}

\begin{rem}
$12 A_1$ is the most generic case, and explicit examples are given for example by \cite{Katsura--Kondo:1-dimensional}*{Section 4}.
We give examples with $\Sing(Y)$ being $3 D_4^0$, $D_8^0 + D_4^0$, $E_8^0 + D_4^0$, or single non-RDP, in Example \ref{ex:supersingular Enriques},
and we will give an example of the remaining RDP case (i.e.\ $\Sing(Y) = D_{12}^0$) in a subsequent paper \cite{Matsumoto:k3cover}*{Section 5}.
See also \cite{Schroer:K3-like}*{Sections 13--15} for various examples, 
although classical and supersingular Enriques surfaces are not distinguished explicitly.
\end{rem}

\begin{rem}
We note an error of an example of Bombieri--Mumford \cite{Bombieri--Mumford:III}*{Section 5}.
Let $X$ be a supersingular Enriques surface (in characteristic $2$).
They showed that there exists a regular vector field $\vartheta$ (canonical up to scalar) and
they gave two examples of $X$, second of which is claimed to have $\delta_X = 0$,
where $\delta_X$ is the scalar defined by $\vartheta^2 = \delta_X \vartheta$
(by normalizing $\vartheta$ we may assume $\delta_X \in \set{0,1}$).
However their calculation is incorrect and this $X$ actually has $\delta_X = 1$.
Note that $\delta_X = 1$ (resp. $\delta_X = 0$) is equivalent to
the morphism $X \to \normalization{(\secondpower{Y})}$ being a $\mu_2$-quotient (resp.\ an $\alpha_2$-quotient),
where $Y \to X$ is the canonical covering of the Enriques surface.

Their construction is as follows.
Let $Y \subset \bP^5$ be the complete intersection of the three quadrics
\begin{gather*}
x_1^2 + x_1 x_2 + y_3^2 + y_1 x_2 + x_1 y_2 = 0, \\
x_2^2 + x_2 x_3 + y_1^2 + y_2 x_3 + x_2 y_3 = 0, \\
x_3^2 + x_3 x_1 + y_2^2 + y_3 x_1 + x_3 y_1 = 0.
\end{gather*}
This surface $Y$ has exactly 6 isolated singular points:
\begin{gather*}
(1, 1, 1, 0, 0, 0); \\
(t^3, t, 1, t^3, t, 1), \quad t^3 + t^2 + 1 = 0; \\
(t^2, t, 1, t^3, t^2, t), \quad t^2 + t + 1 = 0.
\end{gather*}
(We corrected the error on the coordinates of the points of the third type.)
Let $X$ be the quotient of $Y$ by the $\alpha_2$-action
$(x_i, y_i) \mapsto (x_i, \varepsilon x_i + y_i)$,
that is, $D(x_i) = 0$ and $D(y_i) = x_i$.
They claim that $X$ is a smooth supersingular Enriques surface,
but actually it has $3 A_2$ singularities at the images of the $3$ points 
$(t^3, t, 1, t^3, t, 1)$, $t^3 + t^2 + 1 = 0$, of type $A_5$.
(The other singularities of $Y$ are all $A_1$ and their images are smooth points.)
Then $\Sing(\tilde{X} \times_X Y)$ is $12 A_1$,
with three $A_1$ above each $A_5$ of $Y$,
where $\tilde{X} \times_X Y$ is the canonical $\alpha_2$-covering of the resolution $\tilde{X}$ of $X$.
Consequently $\tilde{X}$ has $\delta_{\tilde{X}} = 1$.

We will construct supersingular Enriques surfaces with $\delta_X = 0$.
\end{rem}

\begin{example} \label{ex:supersingular Enriques} %%%%%%%%scr:180223,181002 k3alphap notes:5-9
We consider the indices modulo $3$.
Let $F_i \in k[x_1,x_2,x_3,y_1,y_2,y_3]$ ($i=1,2,3$) be homogeneous quadratic polynomials 
belonging to the subring
$
k[x_j^2, y_j^2, t_j, s_j]_{j = 1,2,3}
$
(resp.\ $
k[x_j^2, y_j^2, t_j, u_j]_{j = 1,2,3}
$),
where $t_j = x_{j+1} x_{j+2}$,
$s_j = y_{j+1} y_{j+2}$,
$u_j = x_{j+1} y_{j+2} + x_{j+2} y_{j+1}$,
and let $Y = (F_1 = F_2 = F_3 = 0) \subset \bP^5$.
Endow $Y$ with 
a derivation $D$ of multiplicative (resp.\ additive) type with 
\begin{align*}
(D(x_j), D(y_j)) &= (0, y_j) \\ (\text{resp.} \; 
(D(x_j), D(y_j)) &= (0, x_j) )
\end{align*}
(see the convention before Example \ref{ex:k3:2}).
If $F_i$ are generic, then $Y$ is an RDP K3 surface and 
the quotient $X = Y^D$ is a classical (resp.\ supersingular) Enriques surface.
Liedtke \cite{Liedtke:liftingEnriques}*{Proposition 3.4} showed that any classical (resp.\ supersingular) Enriques surface is birational to an RDP Enriques surface of this form.
(Liedtke's theorem also covers singular Enriques surfaces (i.e.\ those with $\Pictau = \mu_2$), which we do not discuss in this paper.)

As showed in Proposition \ref{prop:classical Enriques non-duality},
in the classical case there is no (regular) derivation $D'$ on $X$ with $X^{D'} = \normalization{(\secondpower{Y})}$.

Consider the supersingular case.
Write $F_i = A_i + B_i + C_i$,
where 
$A_i \in \spanned{x_j^2, y_j^2}_{j}$,
$B_i \in \spanned{t_j}_{j}$,
$C_i \in \spanned{u_j}_{j}$.
For simplicity assume $C_1, C_2, C_3$ are linearly independent,
and then we may assume $C_i = u_i$.
Write $B_i = \sum_j b_{ij} t_j$.
The derivation $D'$ on $X$ defined by 
\begin{align*}
D'(B_i + u_i) &= 0, \\
D'(t_j)       &= b_{j+1,j+2} x_{j+1}^2 + b_{j+2,j+1} x_{j+2}^2 + e t_j + A_j,
\end{align*}
where $e = \sum_j b_{jj}$,
satisfies $X^{D'} = \normalization{(\secondpower{Y})}$ and $D'^2 = e D'$.
(To check that this is well-defined, it suffices to observe 
$D'(t_{j+1}) t_{j+2} + t_{j+1} D'(t_{j+2}) = x_j^2 D'(t_j)$,
and it is straightforward.)
If $e \neq 0$ then $e^{-1} D'$ is of multiplicative type
and if $e = 0$ then $D'$ is of additive type.
One can check (e.g.\ by using the Jacobian criterion) 
that if $F_i$ is generic with $C_i = u_i$ and $e = 0$ then $\Sing(Y)$ is $3D_4^0$
at $(G_1 = G_2 = G_3 = H_1 = H_2 = H_3 = 0)$,
\begin{align*}
	G_j &= \sqrt{A_j} + \sqrt{b_{j+1,j+2}} x_{j+1} + \sqrt{b_{j+2,j+1}} x_{j+2}, \\
	H_j &= B_j + u_j + b_{j+1,j+2} x_{j+1}^2 + b_{j+2,j+1} x_{j+2}^2.
\end{align*}
Note that the subscheme $(H_1 = H_2 = H_3 = 0) \subset \bP^5$ is of codimension $2$ and degree $3$, since $\sum x_j H_j = 0$.

Now, for simplicity let 
$F_i = A_i + u_i$ 
(so $b_{i j} = 0$ and $e = 0$).
\begin{itemize} %%%%%%%%scr: replace y_3 with y_3+x_3
\item %%%%%%%%scr:181002
If $A_1 = x_1^2 + x_3^2$,
$A_2 = y_1^2 + y_2^2$, $A_3 = x_3^2 + y_3^2$,
then $\Sing(Y)$ is $3D_4^0$ 
at $(x_1,x_2,x_3,y_1,y_2,y_3) = (0,1,0,0,0,0)$, $(1,1,1,1,1,1)$, $(0,0,0,1,1,0)$.
\item %%%%%%%%scr:181013/1,2,9,10
If $A_1 = x_1^2 + x_2^2 + y_1^2$, $A_2 = x_3^2$, $A_3 = y_1^2 + y_2^2$,
then $\Sing(Y)$ is $D_8^0 + D_4^0$ at 
$(1,1,0,0,0,0)$, $(0,0,0,0,0,1)$.
\item %%%%%%%%scr:181002
If $A_1 = x_1^2 + x_2^2 + y_1^2$,
$A_2 = y_1^2 + y_2^2$, $A_3 = x_3^2 + y_3^2$,
then $\Sing(Y)$ is $E_8^0 + D_4^0$ at 
$(1,1,0,0,0,0)$, $(0,0,1,0,0,1)$.
\item 
If %%%%%%%%scr:181003
$A_1 = x_1^2 + x_3^2 + y_1^2$, 
$A_2 = x_2^2 + y_1^2 + y_3^2$, 
and $A_3 = y_2^2$,
then $\Sing(Y)$ consists of a single non-RDP singularity at $(1,0,1,0,0,0)$.
\end{itemize}

We will give an example of the remaining RDP case (i.e.\ $\Sing(Y) = D_{12}^0$), and also examples in the classical case, in a subsequent paper \cite{Matsumoto:k3cover}*{Section 5}.
\end{example}

\section{Examples} \label{sec:examples}

\subsection{\texorpdfstring{Local $\alpha_p$-actions}{Local alpha\_p-actions}}

\begin{example} \label{ex:non-RDP}
For $p = 2,3,5,7$, 
let $D$ be the derivation on $A = k[[x,y]]$ defined as in the table.
Then $D$ is of additive type,
with $\divisorialfix{D} = 0$, 
$\deg \isolatedfix{D}$ is as in the table,
and $A^D = k[[X,Y,Z]] / (F)$, 
where $X = x^p$, $Y = y^p$, $Z$ is as in the table, and $F$ is as in the table,
and $A^D$ is a non-RDP.
(cf.\ Lemma \ref{lem:degree of isolated fixed point}.)

The non-RDPs appearing in Examples \ref{ex:supersingular Enriques} and \ref{ex:k3:3}--\ref{ex:k3:5a:non-RDP}
are isomorphic to the quotient singularities listed here,
at least up to terms of high degree.

\begin{tabular}{llllll}
\toprule
$p$ & $D(x)$ & $D(y)$ & $\deg \isolatedfix{D}$ & $F$ & $Z$ \\
\midrule
$2$ & $y^2$  & $x^6$           & $12$ & $X^7 + Y^3 - Z^2$     & $x^7 + y^3$ \\
$3$ & $y$    & $x^6$           & $6$  & $X^7 + Y^2 - Z^3$     & $x^7 + y^2$ \\
$5$ & $xy$   & $2 (x^2 + y^2)$ & $4$  & $2 X^3 + X Y^2 - Z^5$ & $2 x^3 + x y^2$ \\
$7$ & $y$    & $-2 x^3$        & $3$  & $X^4 + Y^2 - Z^7$     & $x^4 + y^2$ \\
\bottomrule
\end{tabular}
\end{example}

\subsection{Actions on RDP K3 surfaces with rational quotients} \label{subsec:rational quotients}

Examples for $G = \bZ/l\bZ$, $l \leq 19$ and $l \neq p$, are well-known.

Examples for $G = \bZ/p\bZ$, $p \leq 11$, are given in \cite{Dolgachev--Keum:wild-p-cyclic}.

Examples for $G = \mu_p$, $p \leq 19$ and $p \neq 5$, are given in \cite{Matsumoto:k3mun}*{Section 9}.
For $G = \mu_5$, 
the derivation $D = t \partial / \partial t$ on 
the elliptic RDP K3 surface $(y^2 + x^3 + x^2 + t^{10} = 0)$ 
gives an example.

Examples for $G = \alpha_p$, $p \leq 7$, are given in Section \ref{subsec:K3 quotients}.
For $G = \alpha_{11}$,
the derivation $D = \partial / \partial t$ on 
the elliptic RDP K3 surface $(y^2 + x^3 + x^2 + t^{11} = 0)$ 
gives an example.

We do not know whether examples with
$G = \alpha_p$, $p = 13,17,19$, exist.

\subsection{Actions on RDP K3 surfaces with RDP Enriques quotients}

As noted in Proposition \ref{prop:structure of quotient}, this is possible only if $p = 2$.
We gave examples in Example \ref{ex:supersingular Enriques}.

\subsection{Actions with RDP K3 quotients} \label{subsec:K3 quotients}

In this section, we give the following examples of $G$-quotient morphisms $\map{\pi}{X}{Y}$ in the following characteristics $p$.
\begin{itemize}
	\item $X$ and $Y$ are RDP K3 surfaces, $X$ is smooth, $G = \bZ/p\bZ$,
	$(p, \Sing(Y)) = (2, 2 D_4^1), (2, D_8^2), (2, E_8^2), (3, 2 E_6^1), (5, 2 E_8^1)$.
	\item $X$ and $Y$ are RDP K3 surfaces, 
	and the induced morphism $\map{\pi'}{Y}{\pthpower{X}}$ is a $G'$-quotient morphism, with
	\begin{itemize}
		\item $(G, G') = (\mu_p, \mu_p)$, $p \leq 7$;
		\item $(G, G') = (\mu_p, \alpha_p)$, $p \leq 5$;
		\item $(G, G') = (\alpha_p, \alpha_p)$, $p \leq 3$.
	\end{itemize}
	(We note that if $\pi$ is an example for $(G, G') = (\mu_p, \alpha_p)$, then $\pi'$ is an example for $(G, G') = (\alpha_p, \mu_p)$.)
	
	When $p = 2$, we give examples with all pairs 
	$(\Sing(X), \Sing(Y)) \in \set{8 A_1, 2 D_4^0, 1 D_8^0, 1 E_8^0}^2$
	except $(1 E_8^0, 1 E_8^0)$.
	\item $Y$ is an RDP K3 surface with $\Sing(Y)$ and $\Pic(\tilde{Y})$ as in Table \ref{table:singularities of quotient K3}, 
	$X$ is the corresponding $G$-covering that is a K3-like rational surface, and 
	\begin{itemize}
		\item $X$ has a single singularity, which is a non-RDP, $G = \mu_p$ ($p \leq 7, p \neq 2$) and $G = \alpha_p$ ($p \leq 5, p \neq 2$).
		\item $X$ is non-normal, $G = \mu_p$ ($p \leq 7$) and $G = \alpha_p$ ($p \leq 5$). 
	\end{itemize}
	In this case $\map{\pi'}{Y}{\normalization{(\pthpower{X})}}$ is an $\alpha_p$-quotient morphism with rational quotient.
\end{itemize}

We prove in a subsequent paper \cite{Matsumoto:k3rdpht}*{Corollary 6.8} that if $X$ and $Y$ are RDP K3 surfaces then the following are impossible:
\begin{itemize}
	\item $(G, G') = (\alpha_5, \alpha_5)$.
	\item $(G, G', \Sing(X), \Sing(Y)) = (\alpha_2, \alpha_2, 1 E_8^0, 1 E_8^0)$.
\end{itemize}

\medskip

Below we use the following description of derivations.
Suppose $X$ is a projective scheme over $k$,
$L$ is an ample line bundle on it,
and $D^* \in \End_k(H^0(X, L))$ is a $k$-linear endomorphism 
that extends to a derivation $D^*$ of the $k$-algebra $\bigoplus_{m \geq 0} H^0(X, mL)$.
Then $D^*$ induces a derivation $D$ on $X$ 
by $D(f/g) = D^*(f)/g - fD^*(g)/g^2$ 
on $(g \neq 0) \subset X$ for $f,g \in H^0(X, mL)$. 
This can be applied for example to $X = (F = 0) \subset \bP^3$
and $D^* \in \End_k (H^0(\cO_{\bP^3}(1)))$ satisfying $D^*(F) = cF$ for some $c \in k$.
Below we write simply $D$ in place of $D^*$.

\begin{example}[$G = \mu_2$ (resp.\ $G = \alpha_2$)] \label{ex:k3:2} %%%%%%%%scr:180221
Let $F \in k[w,x,y,z]$ be a homogeneous quartic polynomial 
belonging to 
\[
 k[w^2, x^2, y^2, z^2, wx, yz] \quad (\text{resp. }
k[w^2, x^2, y^2, z^2, xz, wz + xy] )
\]
and let $X = (F = 0) \subset \bP^3$.
Such $F$ is uniquely written as 
\begin{align*}
F &= H + wx I + yz J + wxyz K \\ (\text{resp.} \; 
F &= H + xz I + (wz+xy) J + xz(wz+xy) K )
\end{align*}
with $H, I, J, K \in k[w^2, x^2, y^2, z^2]$
of respective degree $4, 2, 2, 0$.
Endow $X$ with a derivation $D$ of multiplicative (resp.\ additive) type with 
\begin{align*}
(D(w),D(x),D(y),D(z)) &= (0,0,y,z) \\ (\text{resp.} \; 
(D(w),D(x),D(y),D(z)) &= (x,0,z,0) )
\end{align*}

If $F$ is generic, then
$X$ and the quotient $Y = X^D$ are RDP K3 surfaces.
Let $L'$ be the line bundle on $Y$ with $H^0(Y, L') = H^0(X, p L)^{D}$.
The derivation $D'$ on $H^0(Y, L')$ defined by $D'(w^2) = D'(x^2) = D'(y^2) = D'(z^2) = 0$
and 
\begin{align*}
D'(wx) &= J + wx K, & D'(yz)    &= I + yz K \\ (\text{resp.} \; 
D'(xz) &= J + xz K, & D'(wz+xy) &= I + (wz+xy) K )
\end{align*}
satisfies $Y^{D'} = \secondpower{X}$ and $D'^2 = K D'$.
If $K \neq 0$ then $K^{-1} D'$ is of multiplicative type,
and if $K = 0$ then $D'$ is of additive type.
This gives an $11$- (resp.\ $10$-) dimensional family $Y$ of $\mu_2$-actions which degenerate to $\alpha_2$-actions in codimension $1$.
One can check that if $F$ is generic then $\Sing(X)$ is $8A_1$,
if $F$ is generic with $K = 0$ then $\Sing(X)$ is $2D_4^0$,
and if $F$ is generic with $K = 0$ and $\#(H = I = J = 0) = 1$ then $\Sing(X)$ is $1D_8^0$.
If $G = \mu_2$ and $(H, I, J, K) = (w^4 + y^4, x^2 + y^2, w^2 + x^2 + y^2 + z^2, 0)$
then $\Sing(X)$ is $1 E_8^0$. %%%%%%%%scr:190609/12
If $G = \alpha_2$ and $(H, I, J, K) = (x^4 + z^4 + w^2 y^2, w^2, y^2, 0)$
then $\Sing(X)$ is $1 E_8^0$ and $\Sing(Y)$ is $2 D_4^0$. %%%%%%%%scr:190611/2
If $G = \alpha_2$ and $(H, I, J, K) = (w^4 + x^4 + z^4, w^2, x^2 + y^2 + z^2, 0)$
then $\Sing(X)$ is $1 D_8^0$ and $\Sing(Y)$ is $1 D_8^0$. %%%%%%%%scr:190611/2
If $G = \alpha_2$ and $(H, I, J, K) = (w^4 + y^2 z^2, x^2, y^2, 0)$ 
then $\Sing(X)$ is $1 E_8^0$ and $\Sing(Y)$ is $1 D_8^0$. %%%%%%%%scr:181006/3

If $G = \mu_2$ and $(H, I, J, K) = (y^2 I + x^2 J, w^2 + y^2, x^2 + \lambda^2 z^2, 0)$
(resp.\ $G = \alpha_2$ and $(H, I, J, K) = ((z^2 + w^2) I + x^2 J, w^2 + z^2, x^2 + \lambda^2 y^2, 0)$),
with $\lambda \in k \setminus \bF_2$,
then $\Sing(X) = (I = J = 0)$, hence $X$ is non-normal,
and $\Sing(Y)$ consists of $\pi(\Fix(D)) = 8 A_1$ (resp.\ $\pi(\Fix(D)) = 1 D_8^0$) and $4 A_1$ (resp.\ $1 A_1$) contained in $\pi(\Sing(X))$.
Let $Y' \to Y$ be the resolution of the latter singularities.
Then $X \times_Y Y' \to Y'$ is an example of a non-normal $\mu_2$- (resp.\ $\alpha_2$-) covering. %%%%%%%%scr:181008/9-15
\end{example}

\begin{example}[$G = \mu_3$ (resp.\ $G = \alpha_3$)] \label{ex:k3:3} %%%%%%%%scr:180221
Let $F \in k[x,y,z]$ be a homogeneous sextic polynomial 
belonging to $k[x,y^3,z^3,A]$,
where $A = yz$ (resp.\ $A = xz + y^2$),
and let $X = (w^2 + F = 0) \subset \bP(3,1,1,1)$.
Such $F$ is uniquely written as 
\[
F = H + xA I + (xA)^2 J
\]
with $H, I, J \in k[x^3, y^3, z^3]$
of respective degree $6, 3, 0$.
Endow $X$ with a derivation $D$ of multiplicative (resp.\ additive) type with 
\begin{align*}
(D(w),D(x),D(y),D(z)) &= (0,0,y,-z) \\ (\text{resp.} \; 
(D(w),D(x),D(y),D(z)) &= (0,0,x,y) ).
\end{align*}

If $F$ is generic, then
$X$ and the quotient $Y = X^D$ are RDP K3 surfaces.
The derivation $D'$ on $Y$ defined by 
\[
D'(y^3) = D'(z^3) = 0, \quad 
D'(w) = I + 2 J xA, \quad 
D'(xA) = w
\]
satisfies $Y^{D'} = X^{(3)}$ and $D'^3 = 2J D'$.
If $J \neq 0$ then $(2J)^{-1/2} D'$ is of multiplicative type
and if $J = 0$ then $D'$ is of additive type.
This gives a $7$- (resp.\ $6$-) dimensional family $Y$ of $\mu_3$-actions which degenerate to $\alpha_3$-actions in codimension $1$.
One can check that if $F$ is generic then $\Sing(X)$ is $6A_2$,
and if $F$ is generic with $J = 0$ then $\Sing(X)$ is $2E_6^0$.

If %%%%%%%%scr:181203
$(H, I, J) = ((\lambda^3 x^3 + y^3)^2 + (y^3 - z^3)^2, y^3 - z^3, 0)$ with $\lambda \in k \setminus \bF_3$,
then $X$ has a single singularity at $(0,1,\lambda,\lambda)$ (resp.\ $(0,0,1,0)$), which is a non-RDP, $X$ is a rational surface,
and $Y$ is an RDP K3 surface with $\Sing(Y) = 6 A_2$ (resp.\ $\Sing(Y) = 2 E_6^0$). 

If %%%%%%%%scr:181004
$(H, I, J) = ((x^3 + y^3 + z^3)^2, x^3 + y^3 + z^3, 0)$,
then $X$ is non-normal rational surface with $\Sing(X) = (w = x + y + z = 0)$,
and $Y$ is an RDP K3 surface with $\Sing(Y) = 6 A_2$ 
(resp.\ $\Sing(Y)$ consists of $\pi(\Fix(D)) = 2 E_6^0$ and $3 A_1$ contained in $\pi(\Sing(X))$), 
and $X \times_Y Y' \to Y'$, where $Y' = Y$ (resp.\ $Y' \to Y$ is the resolution of RDPs of other than $2 E_6^0$)
is an example of a non-normal $\mu_3$- (resp.\ $\alpha_3$-) covering.

\end{example}
\begin{example}[$G = \mu_5$] \label{ex:k3:5m} %%%%%%%%scr:180221
Let $F \in k[x,y,z]$ be a homogeneous sextic polynomial 
belonging to $k[x,y^5,z^5,A]$
where $A = yz$ 
and let $X = (w^2 + F = 0) \subset \bP(3,1,1,1)$.
Endow $X$ with a derivation $D$ of multiplicative (resp.\ additive) type with 
\[
(D(w),D(x),D(y),D(z)) = (0,0,y,-z) 
\]

If $F$ is generic, then
$X$ and the quotient $Y = X^D$ are RDP K3 surfaces.
Write 
\[
 F = a_6 x^6 + a_4 x^4 A + a_2 x^2 A^2 + a_0 A^3 + b x y^5 + c x z^5.
\]
Define a derivation $D'$ on $Y$ by
\begin{gather*}
D'(x^5) = D'(y^5) = D'(z^5) = 0, \\
D'(w x^2) = 3 x \frac{\partial F}{\partial A}, 
D'(wA) = \frac{\partial F}{\partial x}, 
D'(x^3 A) = - w x^2,
D'(xA^2) = -2 wA.
\end{gather*}
Then it satisfies $Y^{D'} = X^{(5)}$ and $D'^5 = e D'$,
where $e = a_2^2 - 3 a_0 a_4$.
If $e \neq 0$ then $e^{-1/4} D'$ is of multiplicative type
and if $e = 0$ then $D'$ is of additive type.
This gives a $3$-dimensional family $Y$ of $\mu_5$-actions which degenerate to $\alpha_5$-actions in codimension $1$.
One can check that if $F$ is generic then $\Sing(X)$ is $4A_4$,
and if $F$ is generic with $e = 0$ then $\Sing(X)$ is $2E_8^0$.

If %%%%%%%%scr:181008/1-5, 181203/4-6
$F = (A - x^2)^3 + x (2 x^5 + y^5 + z^5)$,
then $X$ has a single singularity at $(w,x,y,z) = (0,1,-1,-1)$, which is a non-RDP, $X$ is a rational surface,
and $Y$ is an RDP K3 surface with $\Sing(Y) = 4 A_4 + A_2$, where $A_2$ is the image of the non-RDP. 
Let $Y' \to Y$ be the resolution of the $A_2$ point, then $\Sing(X \times_Y Y')$ is a single non-RDP. 
\end{example}

\begin{example}[$G = \mu_7$] \label{ex:k3:7m} %%%%%%%%scr:180221
Let $a \in k$, 
$F = w^2 + x_1^5 x_2 + x_2^5 x_4 + x_4^5 x_1 + a x_1^2 x_2^2 x_4^2 \in k[w, x_1, x_2, x_4]$
and $X = (F = 0) \subset \bP(3,1,1,1)$.
Let $b = (a^{-3} - 1)^{1/3} \in k \cup \set{\infty}$, hence $b = 0$ if and only if $a^3 = 1$.
Then $\Sing(X)$ consists of the points $(0,x_1,x_2,x_4)$ satisfying
\[ (x_1^5 x_2 : x_2^5 x_4 : x_4^5 x_1 : a x_1^2 x_2^2 x_4^2) = (1 + 4jb : 1 + 2jb : 1 + jb : 4)  \]
for some $j \in \set{1,2,4}$,
and it is $3 A_6$ if $b \neq 0$ and a single non-RDP if $b = 0$.
$X$ admits a derivation $D$ of multiplicative type
with 
\[
D(w) = 0, \; D(x_i) = i x_i.
\]
whose quotient $Y = X^D$ is an RDP K3 surface.
If $b \neq 0$ then $\Sing(Y) = \pi(\Fix(D))$ is $3 A_6$,
and if $b = 0$ then $\Sing(Y) = \pi(\Fix(D)) \cup \pi(\Sing(X))$ is $3 A_6 + A_1$.
In the latter case, let $Y' \to Y$ be the resolution of the $A_1$ point, then $\Sing(X \times_Y Y')$ is a single non-RDP 
whose completion is isomorphic to $k[[X,Y,Z]] / (X^2 + Y^4 + Z^7 + \dots)$.

$Y$ admits a derivation $D'$ defined by 
\begin{align*}
 D'(x_i^7)                 &= 0, \\
 D'(x_i x_{2i}^2 x_{4i}^4) &= i^2 w x_{2i} x_{4i}^3, \\
 D'(w x_i x_{2i}^3)        &= i^2 (- x_{2i}^7 + 2 x_i x_{2i}^2 x_{4i}^4 - 2a x_i^2 x_{2i}^4 x_{4i}),
\end{align*}
$i = 1,2,4$, where the indices are considered modulo $7$, 
satisfying $D'^7 = (1 - a^3) D'$. %%%%%%%%scr:181118/1
\end{example}

\begin{example}[$G = \alpha_5$] \label{ex:k3:5a:non-RDP}
Let $Y$ be the RDP K3 surface 
$w^2 + (y^2 - 2xz)^3 + z(x^5 + y^5 + z^5) = 0$,
equipped with the derivation $D'$ defined by $D'(w) = 0$, $D'(x) = y$, $D'(y) = z$, $D'(z) = 0$.
Then $\Sing(Y)$ is $2 E_8^0$ at $w = y^2 - 2xz = x + y + z = 0$.
Then $\thpower{(Y^{D'})}{1/p}$ is the $\alpha_5$-covering of $Y$,
with a single singularity that is non-RDP.
\end{example}

\begin{example}[$G = \mu_5$ (resp.\ $G = \alpha_5$)] \label{ex:k3:5:non-normal}
Let $a \in k$ and assume $a (a^3 - 2) \neq 0$ (resp.\ $a = 0$).
Let $S$ be the elliptic RDP K3 surface 
$y^2 = x^3 + a x^2 + t^5 (t-1)^5$,
equipped with the derivation $D' = \partial / \partial t$
having $1$-dimensional fixed locus at $t = \infty$.
Then $\Sing(S)$ is $4 A_4$ at $t = 0$, $t = 1$, $t^5 (t - 1)^5 + 2 a^3 = 0$
(resp.\ $2 E_8^0$ at $t = 0$, $t = 1$).
$S$ admits a non-normal $\mu_5$- (resp.\ $\alpha_5$-) covering,
birational to $\thpower{(S^{D'})}{1/p}$.
We see that $S^{D'}$ is a certain compactification of $y^2 = x^3 + a x^2 + T (T-1)$, where $T = t^5$.
\end{example}
\begin{example}[$G = \mu_7$] \label{ex:k3:7m:non-normal}
Let $S$ be the elliptic RDP K3 surface
$y^2 = x^3 + t^7 x + 1$,
equipped with the derivation $D' = \partial / \partial t$
having $1$-dimensional fixed locus at $t = \infty$.
Then $\Sing(S) = 3 A_6$ at $-4 (t^7)^3 - 27 = 0$.
Similarly to the previous example,
$S$ admits a non-normal $\mu_7$-covering 
birational to $\thpower{(S^{D'})}{1/p}$.
We see that $S^{D'}$ is a certain compactification of $y^2 = x^3 + T x + 1$, where $T = t^7$.
\end{example}

\begin{example}[$G = \bZ/2\bZ$; See also \cite{Dolgachev--Keum:wild-p-cyclic}*{Examples 2.8}] \label{ex:k3:z2}
Let $F \in k[w,x,y,z]$ be a homogeneous quartic polynomial 
belonging to 
\[
k[w^2 + x^2, y^2 + z^2, wx, yz, wy + xz, wz + xy] ,
\]
and let $X = (F = 0) \subset \bP^3$.
Endow $X$ with an automorphism $g$ of order $2$ with
$g(w,x,y,z) = (x,w,z,y)$.
If $F$ is generic, then $X$ is a smooth K3 surface and $Y = X / \spanned{g}$ is an RDP K3 surface,
with 
\[ \Fix(g) = \set{(\alpha,\alpha,\beta,\beta) \mid \alpha^2 c(w^2 x^2)^{1/2} + \alpha \beta c(w x y z)^{1/2} + \beta^2 c(y^2 z^2)^{1/2} = 0} , \]
where $c(m)$ are the coefficients of the monomials $m$ in $F$.
If $F$ is generic (resp.\ generic with $c(w x y z) = 0$), then $\Sing(Y) = \pi(\Fix(g))$ is $2 D_4^1$ (resp.\ $1 D_8^2$). 

Now let $X \subset \bP^5 = \Proj k[x_1, x_2, y_1, y_2, y_3, y_4]$ be the K3 surface defined by
\[
x_1^2 + x_1 y_1 + y_3 y_2 = 
x_2^2 + x_2 y_2 + y_1^2 + y_3 y_4 = 
y_1 y_3 + y_2 y_4 + y_4^2 = 0,
\]
with automorphism $g$ defined by $g(x_i) = x_i + y_i$, $g(y_i) = y_i$.
Then $\# \Fix(g) = 1$ (at $x_1 = x_2 = y_1 = y_2 = y_4 = 0$), 
and $Y = X / \spanned{g}$ is an RDP K3 surface with $\Sing(Y) = \pi(\Fix(g)) = 1 E_8^2$.

\end{example}
\begin{example}[$G = \bZ/3\bZ$] \label{ex:k3:z3}
Let $F \in k[w,x,y,z]$ be a homogeneous quartic polynomial 
belonging to 
\[
k[w, x + y + z, xy + yz + zx, xyz, (x-y)(y-z)(z-x)],
\]
and let $X = (F = 0) \subset \bP^3$.
Endow $X$ with an automorphism $g$ of order $3$ with
$g(w,x,y,z) = (w,y,z,x)$.
If $F$ is generic
(e.g.\ if $F = w^4 + x^4 + y^4 + z^4 - \lambda^3 wxyz$ with $\lambda \neq 0,1$),
then $X$ is a smooth K3 surface,
$\Fix(g) = \set{(0,1,1,1), (\lambda,1,1,1)}$ where $\lambda = (-c(wxyz)/c(w^4))^{1/3}$,
 and $Y = X / \spanned{g}$ is an RDP K3 surface with $\Sing(Y) = 2 E_6^1$.

\end{example}
\begin{example}[$G = \bZ/5\bZ$, and $G = \alpha_5$] \label{ex:k3:z5}
Let $a, b_{-1}, b_0, b_1 \in k$
with $b_{-1} b_1 \neq 0$.
Let $b = b(t) = b_{-1} t^{-1} + b_0 + b_1 t$
and $c = c(t) = t b(t) = b_{-1} + b_0 t + b_1 t^2$.
Let $S$ and $T$ be two RDP K3 surfaces defined by 
\begin{align*}
S \colon y^2  &= x^3 + a t^4 x + t^5 c    , \\
T \colon Y^2 &= X^3 + a^5 t^4 X + t c^5 .
\end{align*}
Let $\xi = t^{-2} X + ab$.
Let $\Delta = - 4 a^3 - 27 b^2$.
Let $\map[\rationalto]{f}{T}{S}$ be the rational map defined by $f(X,Y) = $
\[
 \biggl(
t^2 \frac{\xi^5 - ab \xi^4 - a^2 \Delta \xi^3 - a \Delta^3 \xi}{(2 a \xi^2 + \Delta^2)^2}, 
Y \frac{\xi^6 + a^2   \Delta \xi^4 - 2 b   \Delta^2 \xi^3 - a \Delta^3 \xi^2 + 2 \Delta^5}{(2 a \xi^2 + \Delta^2)^3} 
\biggr).
\] 
Over $k(t)$, this defines a separable (resp.\ inseparable) isogeny of degree $5$ between ordinary (resp.\ supersingular) elliptic curves if $a \neq 0$ (resp.\ $a = 0$).

Suppose $b$ is generic and $a \neq 0$.
Then $T$ and $S$ are RDP K3 surfaces with $4 A_4$ and $2 E_8^1$ respectively.
Let $\tilde{T} \to T$ be the resolution.
Then $f$ induces a finite morphism $\tilde{T} \to S$
that is the quotient morphism of a $\bZ/5\bZ$-action generated by 
the translation by a $5$-torsion point $(X,Y) 
 = (\frac{2}{e^2} \Delta - ab, 2 \Delta (e^3 + \frac{b}{e^3}))$,
$e^4 = 2a$.

Suppose $a = 0$ and $\disc(c) = b_0^2 - 4 b_{-1} b_1 \neq 0$ (so $c$ is not a square).
Then $T$ and $S$ are both RDP K3 surfaces with $2 E_8^0$.
Let $\tilde{T} \to T$ be the resolution,
$C$ be the unique $4 A_4$ configuration contained in the union of the two fibers over $t = 0$ and $t = \infty$,
and $\tilde{T} \to T'$ be the contraction of $C$.
Then $T'$ is an RDP K3 surface with $4 A_4$,
and $f$ induces a finite morphism $\map{f'}{T'}{S}$
which is an $\alpha_5$-quotient morphism.
Define a derivation $D'$ on $S$ by
$D'(x) = 2 c'(t)x$, $D'(y) = 3 c'(t)y$, $D'(t) = c(t)$.
We have $D'^5 = (\disc(c))^2 D'$.
This defines a $\mu_5$-action on $S$ 
whose quotient is $\thpower{T'}{5}$.

Suppose $a = 0$ and $\disc(c) = b_0^2 - 4 b_{-1} b_1 = 0$ (so $c$ is a square).
Then $\Sing(S)$ contains $2 E_8^0$,
the derivation $D'$ on $S$ defined as above has divisorial fixed locus, and the corresponding $\alpha_5$-covering of $S$ is non-normal.
\end{example}

\subsection{Inseparable morphisms of degree $p$ between RDP K3 surfaces}

We give an example for each case with $r > 1$ mentioned in Theorem \ref{thm:insep K3}.

\begin{example}[Kummer surfaces and generalized Kummer surfaces (cf. \cite{Katsura:generalizedkummer})] \label{ex:insep K3 with abelian cover}
	Let $r \in \set{2,3,4,6}$.
	Let $p$ be a prime with $p \equiv 1 \pmod{r}$.
	Let $\map{\bar{\pi}}{A}{B}$ be a purely inseparable isogeny of degree $p$ between abelian surfaces in characteristic $p$,
	(automatically) induced by a derivation, say $D$.
	Suppose we have symplectic automorphisms $g_A \in \Aut_0(A)$ and $g_B \in \Aut_0(B)$ of same order $r$
	satisfying $\bar{\pi} \circ g_A = g_B \circ \bar{\pi}$ 
	and $g_A^*(D) = \zeta D$ for a primitive $r$-th root $\zeta$ of unity.
	Here $\Aut_0$ is the group of automorphisms preserving the origin.
	Then $\map{\pi}{A / \spanned{g_A}}{B / \spanned{g_B}}$ is a purely inseparable morphism of degree $p$ between RDP K3 surfaces,
	whose covering as in Theorem \ref{thm:insep K3} is $\bar{\pi}$.

	The singularities of the quotients are as in Table \ref{table:singularities of abelian-quotient K3} \cite{Katsura:generalizedkummer}*{Table in page 17}: %%%%%%%%scr:180726/-
	$16 A_1$, $9 A_2$, $4 A_3 + 6 A_1$, $A_5 + 4 A_2 + 5 A_1$ for $r = 2,3,4,6$ respectively.
	
	Examples of such $\bar{\pi}, g_A, g_B$ are given as follows.
	If $r = 2$, take $\bar{\pi}$ arbitrarily and let $g_A = [-1]_A$, $g_B = [-1]_B$.
	If $r = 3,4,6$, take an elliptic curve $E$ equipped with an automorphism $h \in \Aut_0(E)$ of order $r$,
	and let $\map{\bar{\pi}}{A = E \times E}{B = E \times E^{(p)}}$ and $g_A = h \times h^{-1}$, $g_B = h \times (h^{(p)})^{-1}$.
	Then $g_B$ is symplectic since $p \equiv 1 \pmod{r}$.
\end{example}

\begin{rem}
	If $\map{\bar{\pi}}{A}{B}$ be a purely inseparable morphism of degree $p$ between non-supersingular abelian surfaces in characteristic $p = 2$,
	then $\map{\pi}{A / \set{\pm 1}}{B /\set{\pm 1}}$ is a $\mu_2$- or $\alpha_2$-quotient morphism between RDP K3 surfaces.
	More precisely, if $\prank(A) = 2$ (resp.\ $\prank(A) = 1$)
	then both $\Sing(A / \set{\pm 1})$ and $\Sing(B / \set{\pm 1})$ are $4D_4^1$ (resp.\ $2D_8^2$)
	(Katsura \cite{Katsura:Kummer2}*{Proposition 3}),
	and both $\bar{\pi}$ and $\pi$ are $\mu_2$-quotient (resp.\ either both are $\mu_2$-quotient or both are $\alpha_2$-quotient).

	If $A$ is (and hence $B$ is) supersingular, then $A / \set{\pm 1}$ is not birational to a K3 surface,
	instead it is a rational surface with a single non-RDP singularity 
	(Katsura \cite{Katsura:Kummer2}*{Proposition 3}), and so is $B / \set{\pm 1}$.
\end{rem}

\begin{example} \label{ex:insep K3 with K3 cover}
For each pair of $G \in \set{\mu_p, \alpha_p}$ and $r > 1$ appearing in Theorem \ref{thm:insep K3}(\ref{case:K3 mu},\ref{case:K3 alpha}),
we give an example of an RDP K3 surface $\bar{X}$ with a derivation $D$ of multiplicative type or additive type
and a symplectic automorphism $g \in \Aut(X)$ of order $r$
such that $\bar{Y} = \bar{X}^{D}$ is an RDP K3 surface and
$g^*(D) = \zeta D$ for a primitive $r$-th root $\zeta$ of unity,
hence $g$ induces a symplectic automorphism $g' \in \Aut(Y)$ (of order $r$),
and 
the induced morphism $\map{\pi}{X = \bar{X}/\spanned{g}}{Y = \bar{Y}/\spanned{g'}}$ 
has $\map{\bar{\pi}}{\bar{X}}{\bar{Y}}$ as its minimal covering as in Theorem \ref{thm:insep K3}.

[$\mu_5$, $r = 4$]
Let $\bar{X} = (x_1^3 x_2 - x_2^3 x_4 + x_4^3 x_3 - x_3^3 x_1 = 0) \subset \bP^3$ be the quartic RDP K3 surface
(with $4 A_4$ at $\set{(x_1 : x_2 : x_3 : x_4) = (1 : 2 e^3 : e : 3 e^2) \mid e^4 = -1}$),
and define a derivation $D$ and an automorphism $g$ of $\bar{X}$ by 
$D(x_i) = i x_i$, 
$g(x_i) = x_{(2i \bmod 5)}$.
Then both $D$ and $g$ are symplectic, and $g^* D = 2^{-1} D$.
Hence $\map{\pi}{X = \bar{X}/\spanned{g}}{Y = \bar{Y}/\spanned{g}}$ is an example with $\bar{\pi}$ a $\mu_5$-quotient and $r = 4$.

[$\mu_7$, $r = 3$]
Suppose $b \neq 0$ in Example \ref{ex:k3:7m}, and let $g(w, x_1, x_2, x_4) = (w, x_4, x_1, x_2)$.
Then $g$ is symplectic and $g^* D = 2 D$.

[$\alpha_5$, $r = 2$]
Suppose $e = 0$ in Example \ref{ex:k3:5m} and suppose moreover $b = c$, and let $g(w, x, y, z) = (-w, x, z, y)$.
Then $g^* D = - D$ and $g^* D' = - D'$.

[$\mu_3$ (resp.\ $\alpha_3$), $r = 2$]
In Example \ref{ex:k3:3} suppose 
that $H$ and $I$ are invariant under $(x,y,z) \mapsto (x,z,y)$ (resp.\ $(x,y,z) \mapsto (x,-y,z)$).
For example, let $F = x^6 + y^6 + z^6 + xyz(y^3 + z^3)$
(resp.\ $F = x^6 + y^6 + z^6 + x (xz + y^2) (x^3 - z^3)$).
Let $g(w,x,y,z) = (-w,x,z,y)$ (resp.\ $g(w,x,y,z) = (-w,x,-y,z)$). Then $g^*(D) = -D$.
\end{example}

\subsection*{Acknowledgments}
I thank 
Simon Brandhorst, Hiroyuki Ito, Tetsushi Ito, Yukari Ito, Shigeyuki Kondo, Hisanori Ohashi, and Takehiko Yasuda
for helpful comments and discussions.

\end{document}